\documentclass[11pt]{amsart}
\usepackage{graphicx}
\usepackage{amssymb}
\usepackage{xcolor}
\date{}

\newtheorem{theorem}{Theorem}[section]
\newtheorem{definition}[theorem]{Definition}
\newtheorem{ex}[theorem]{Example}
\newtheorem{lemma}[theorem]{Lemma}

\newtheorem{proposition}[theorem]{Proposition}
\newtheorem{corollary}[theorem]{Corollary}
\newtheorem{remark}[theorem]{Remark}

\newcommand{\m}{\mathfrak{m}}
\renewcommand{\S}{\mathbb{S}}

\newcommand{\D}{\mathbb{D}}
\newcommand{\N}{\mathbb{N}}
\newcommand{\Z}{\mathbb{Z}}

\newcommand{\R}{\mathbb{R}}
\newcommand{\C}{\mathbb{C}}
\newcommand{\SA}{\mathcal{A}}

\newcommand{\SM}{\mathcal{M}}
\newcommand{\SG}{\mathcal{G}}

\newcommand{\SX}{\mathcal{X}}

\newcommand{\La}{\Lambda}
\newcommand{\la}{\lambda}

\newcommand{\Br}{\operatorname{Br}}

\newcommand{\Aut}{\operatorname{Aut}}
\newcommand{\DT}{\operatorname{DT}}

\newcommand{\Symp}{\operatorname{Symp}}

\newcommand{\PSL}{\operatorname{PSL}}

\newcommand{\Gr}{\operatorname{Gr}}
\newcommand{\uf}{\operatorname{uf}}

\newcommand{\st}{\text{st}}

\newcommand{\Trace}{\operatorname{Tr}}
\newcommand{\Mod}{\operatorname{Mod}}
\newcommand{\Leg}{\operatorname{Leg}}
\newcommand{\Hom}{\operatorname{Hom}}
\newcommand{\w}{\mathfrak{w}}
\newcommand{\FM}{\mathfrak{M}}

\begin{document}

	\title{Legendrian Loops and Cluster Modular Groups}
	
	\author{James Hughes}
\address{Duke University, Dept. of Mathematics, Durham, NC 27708, USA}
\email{jhughes@math.duke.edu}
\begin{abstract}
    This work studies Legendrian loop actions on exact Lagrangian fillings of Legendrian links in $(\R^3, \xi_{\st})$. By identifying the induced action of Legendrian loops as generators of cluster modular groups, we establish the existence of faithful group actions on the exact Lagrangian fillings of several families of Legendrian positive braid closures, including all positive torus links. In addition, we leverage a Nielsen-Thurston-like classification of cluster automorphisms to provide new combinatorial and algebraic tools for proving that a Legendrian loop action has infinite order.    
\end{abstract}



\keywords{}

	\maketitle

\section{Introduction}

 In this work, we study actions of Legendrian loops of Legendrian links in $(\R^3, \xi_{\st})$ on the Lagrangian concordance monoid. We first show that by interpreting Legendrian loops as generators of cluster modular groups, we can explicitly realize these group actions for several families of Legendrian positive braid closures. Using this cluster-theoretic perspective, we also provide new combinatorial and algebraic tools for determining whether a Legendrian loop has infinite order. We then apply these tools to show that a large class of Legendrian loop actions have infinite order, recovering several previously known results.

\subsection{Context}
Recent work of Casals and Gao produced infinitely many exact Lagrangian fillings for many Legendrian torus links \cite{CasalsGao}, providing the first such examples. They generate these fillings by understanding Legendrian loops, that is, elements of the fundamental group of the space of Legendrian embeddings of a given knot or link.  
Legendrian loops produce an action on exact Lagrangian fillings by concatenating the trace of the loop with an existing filling. In addition to the result cited above, Legendrian loops have been used to produce various novel examples of families of Legendrian and Lagrangian  submanifolds, including the construction of infinitely many closed exact Lagrangian submanifolds of Weinstein manifolds \cite{CasalsGao, CasalsNg}, infinitely many spherical spun Legendrians \cite{GolovkoSpuns}, and infinitely many Legendrian spheres in higher dimensions \cite{capovillasearle2023newton}.

Despite the recent spate of results in this direction, the work of Casals and Gao remains the only one to describe a non-cyclic group acting on the Lagrangian concordance monoid. In particular, they obtain faithful $\PSL_2(\Z)$ and $\Mod(\Sigma_{0, 4})$ actions in the case of Legendrian torus links $\La(3,6)$ and $\La(4,4)$ \cite[Theorem 1.1]{CasalsGao}. While their work predates the construction of a cluster structure on the decorated sheaf moduli $\FM(\La)$ \cite[Theorem 1.1]{CasalsWeng}, they draw heavily on the work of Fraser \cite{fraser2018braid} in understanding cluster automorphisms of top-dimensional positroid cells of the Grassmannian.  

Since the initial work of Casals and Gao, several works have appeared relating cluster theory and the classification of exact Lagrangian fillings of positive braid closures. 
These works include the proposal of a conjectural $ADE$ classification \cite[Conjecture 5.3]{CasalsLagSkel}, the existence of infinitely many exact Lagrangian fillings for several families of Legendrians links \cite{CasalsZaslow, GSW}, computations of the cohomology of the augmentation variety \cite{CGGLSS}, and the existence of the conjectural number of fillings of several families of Legendrian links \cite{ABL22, Hughes2021, CasalsGao23}. 

\subsection{Main results}

The main goal of this work is to initiate a systematic study of groups arising from Legendrian loop actions on the Lagrangian concordance monoid. 
We do so by interpreting Legendrian loops as \emph{cluster} automorphisms of the decorated sheaf moduli $\FM(\La)$. Cluster automorphisms form a group called the cluster modular group, and the results of Casals and Gao can be understood as producing a subgroup of the cluster modular group of the sheaf moduli $\FM(\La)$. 

\subsubsection{Cluster modular groups generated by Legendrian loops}

For a given cluster algebra $\mathbb{A}$, the cluster modular group $\SG(\mathbb{A})$ is only known in a limited number of cases. To describe explicit groups in these cases, we start by restricting our attention to finite, affine, and extended affine Dynkin type cluster varieties arising as the sheaf moduli of certain Legendrian links. These Legendrian links are Legendrian isotopic to the rainbow closures $\La(\beta)$ of certain positive braids $\beta\in \Br_n^+$. See Figure \ref{fig: BraidClosure} for a depiction of these Legendrian braid closures as well as the Legendrian isotopic $(-1)$-closures of $\beta\Delta_n^2$, as pictured in Figure \ref{fig: BraidClosure}, with the full twist $\Delta_n^2$ given by the square of $\Delta_n=\prod_{i=1}^n\prod_{j=1}^{n-i} \sigma_j$. Given a Dynkin type $X$, we define the Legendrian link $\La(X)$ to be the $(-1)$-framed closure of the corresponding braid: 
\begin{itemize}
    \item $A_{n-1}= \sigma_1^{n+2}$
    \item $D_{n}= \sigma_1^{n-2}\sigma_2\sigma_1^2\sigma_2\Delta_3^2$ for $n\geq 4$
    \item $E_6, E_7, E_8=\sigma_1\sigma_2^2\sigma_1^2\sigma_2^{i-3}\Delta_3^2$ for $i=6, 7,$ and $8$, respectively 
    \item $\tilde{D}_n=(\sigma_2\sigma_1\sigma_3\sigma_2)^2\sigma_1^{n-4}\Delta_3^2$ for $n\geq 4$
    \item $\tilde{E}_n =(\sigma_2\sigma_1\sigma_3\sigma_2)^2\sigma_3\sigma_1^{n-5}\Delta_3^2$ for $n\in\{6, 7, 8\}$
    \item $E_7^{(1,1)}=(\sigma_2\sigma_1\sigma_3\sigma_2)^2\sigma_3^2\sigma_1^2\Delta_3^2$
      \item $E_8^{(1,1)}=(\sigma_2\sigma_1\sigma_3\sigma_2)^2\sigma_3\sigma_1^4\Delta_3^2$ 
\end{itemize}

Here $\tilde{D}$ and $\tilde{E}$ denote affine Dynkin types $D$ and $E$. Similarly, the superscript  
$(1, 1)$ indicates extended affine type. 
Note that  $\La(E_7^{(1, 1)})$ and $\La(E_8^{(1,1)})$ are Legendrian isotopic to the Legendrian torus links $\La(4, 4)$ and $\La(3, 6)$, respectively and that all of the above links are isotopic to rainbow closures $\La(\beta)$ for some $\beta\in \Br_n^+$. 

 	\begin{center}
		\begin{figure}[h!]{ \includegraphics[width=.8\textwidth]{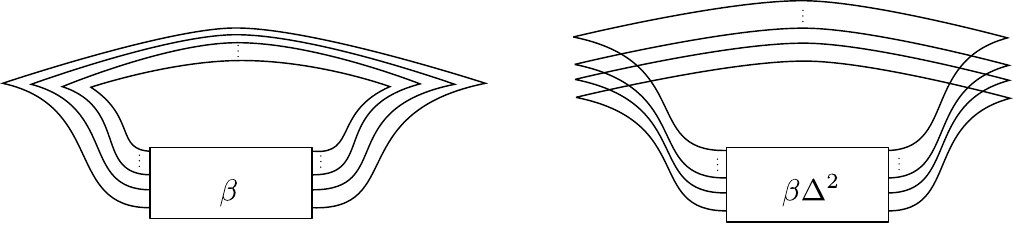}}\caption{Front projections of the Legendrian isotopic links given as the rainbow closure (left) and $(-1)$-framed closure (right) of the positive braids $\beta$ and $\beta\Delta^2$. Here $\Delta$ denotes a half twist of the braid.}
			\label{fig: BraidClosure}\end{figure}
	\end{center} 

Let us denote by $\mathcal{H}$ the set of Legendrian links above. 
To obtain a generating set for the cluster modular group, we must also include the cluster automorphism known as the Donaldson-Thomas (DT) transformation, first considered in this contact-geometric context in \cite[Section 4.1]{GSW} and then described geometrically by Casals and Weng in \cite[Section 5.4]{CasalsWeng}. The square of the DT transformation can be realized as a Legendrian loop, but a single power is defined as the composition of a Legendrian isotopy with a strict contactomorphism. Studying the induced actions of Legendrian loops and the DT transformation, we show that Legendrian loops generate a finite index subgroup of the cluster modular group in the cases we consider.

\begin{theorem}\label{thm: gens} 
For $\La\in \mathcal{H}$, Legendrian loop actions generate a finite-index subgroup of the cluster modular group $\SG(\FM(\La))$. Moreover, for any  $\La\in\mathcal{H}$ excluding $\La(\tilde{D}_n)$, Legendrian loop actions and the $\DT$ transformation generate the group $\SG(\FM(\La))$. 
\end{theorem}

For Legendrian links of affine and extended affine type, the cluster modular groups are infinite, allowing us to recover the fact that these Legendrian links admit infinitely many distinct exact Lagrangian fillings.

\begin{remark}
    As a consequence of \cite[Theorem 1.1]{CasalsGao23}, any cluster automorphism of $\FM(\La(\beta))$ can be given as the induced action of a sequence of Lagrangian disk surgeries applied to $\mathbb{L}$-compressing cycles of $H_1(L)$ for any filling $L$ of $\La$. However, realizing any given sequence of mutations as a Legendrian loop is a nontrivial question, as demonstrated in the case of $\La(\tilde{D}_n)$ in Appendix \ref{sec: appendix}. More precisely, there is no known algorithm for showing that for a Legendrian loop $\varphi$ whose induced action produces a sequence of mutations $\mu_n\circ \dots \circ\mu_1$ the Lagrangian filling $L'=\mu_n\circ \dots \circ\mu_1$ produced by the sequence of Lagrangian disk surgeries at the corresponding $\mathbb{L}$-compressing cycles is Hamiltonian isotopic to the Lagrangian filling $\varphi(L)$.   
\end{remark}

The methods we use to prove Theorem \ref{thm: gens} suggest ways to realize subgroups of cluster modular groups as Legendrian loop actions, even in cases where the full cluster modular group is not known. In the case of Legendrian torus links $\La(k, n)$, the sheaf moduli $\FM(\La(k, n))$ is isomorphic to the affine cone of the Grassmannian $\Gr^\circ(k, n+k)$.  
In \cite[Conjecture 8.2]{fraser2018braid}, Fraser describes a conjectural presentation for the cluster modular group of $\Gr^\circ(k, n+k)$ by studying braid group actions on this space, giving us a roadmap for realizing these group actions as these induced action of some Legendrian isotopy. 
See Subsection \ref{sec: Torus Links} for a more detailed description of this conjectural cluster modular group $\mathcal{G}'(k, n)$. By generalizing work of Casals and Gao in \cite{CasalsGao}, we show the following:

\begin{theorem}\label{thm: TorusLinkCMG}
    For any $k, n \geq 2$, Legendrian loop actions of $\La(k, n)$ together with the DT transformation generate a subgroup of $\SG(\FM(\La(k, n))$ isomorphic to $\mathcal{G}'(k,n+k)$. 
\end{theorem}

  If we view Legendrian loops as elements of $\pi_1(\Leg(K))$ where $\Leg(K)$ is the space of Legendrian embeddings of a knot or link $K$, then Theorems \ref{thm: gens} and \ref{thm: TorusLinkCMG} can be interpreted as a statement about actions of $\pi_1(\Leg(K))$ on $\FM(\La)$ when $\La$ is a max-tb Legendrian representative of $K$ and $\FM(\La)$ admits a cluster structure. In particular, following evidence given by Theorem 1.2.4 of \cite{fernández2022homotopy}, we conjecture that $\pi_1(\Leg(K))$ is isomorphic to a finite-index subgroup of the cluster modular group $\SG(\FM(\La(\beta)))$. Work of Fern\'andez and Min in this context also suggests a way to understand cluster modular groups of more complicated cluster varieties by realizing the corresponding Legendrian as a cable of a Legendrian $\La$ with a known cluster modular group $\SG(\FM(\La))$.

Following a remark in \cite[Section 5]{CasalsLagSkel}, we also consider Legendrian links with exact Lagrangian fillings that are invariant under finite group actions. Let $G$ be a finite group acting by exact symplectomorphisms or anti-exact symplectomorphisms 
on $(\R^4, \omega_{\st})$, thereby inducing an action by contactomorphisms on the boundary $(\R^3, \xi_{\st})$. See Subsection \ref{sub: G-actions} for explicit constructions of the group actions that we consider. We define an exact Lagrangian $G$-filling of $\La$ 
to be an exact Lagrangian filling $L$ of $\La$ such that $G(L)=L$ and $G(\La)=\La$. Similarly, we refer to a Legendrian loop $\varphi$ with trace $\Trace(\varphi)$ fixed by $G$ as a $G$-loop.

In general, one can obtain a new cluster algebra by folding an existing one along a $G$-action satisfying certain properties. Work of An, Bae, and Lee \cite[Theorem 1.4]{ABL22} establishes the existence of the conjectured number of $G$-fillings in this setting; see \cite[Conjecture 5.4]{CasalsLagSkel} for a precise formulation. We prove a $G$-equivariant version of Theorem \ref{thm: gens} for the following Legendrian links and $G$-actions:
\begin{align*}
 \mathcal{H}^G:=\{(\La(A_{2n-1}), \Z_2),(\La(D_4), \Z_3), (\La(\tilde{E}_6), \Z_3), (\La(\tilde{D}_{2n}), \Z_2), (\La(\tilde{D}_4), \Z_2),\\ (\La(D_n), \Z_2), (\La(E_6),\Z_2), (\La(\tilde{E}_6),\Z_2), (\La(\tilde{E}_7),\Z_2), (\La(\tilde{D}_4),\Z_2)\}
\end{align*}
where the pair $(\La, G)$ represents a Legendrian link and a $G$-action fixing $\La$.  
Note that while the positive braids we give above do not necessarily admit obvious $G$-actions, the Legendrian links we consider admit fronts that do exhibit the required symmetry. We compile the data of the choice of such fronts in Table \ref{tab: folded}.
Denote by $\FM(\La)^G$ the $G$-invariant subset of $\FM(\La)$, i.e. the moduli of $G$-invariant sheaves with singular support on $\La$.  

\begin{theorem}\label{thm: folded gens}
    For $(\La, G)\in \mathcal{H}^G$, the cluster modular group $\SG(\FM(\La)^G)$ is generated by $G$-invariant Legendrian loops of $\La$ and the DT transformation. 
\end{theorem}

Implicit in the statement of Theorem \ref{thm: folded gens} is the fact that the ring of regular functions $\C[\FM(\La)^G]$ is a skew-symmetrizable cluster algebra for these choices of $\La$ and $G$ with a $\DT$ transformation obtained from the DT transformation on $\FM(\La)$. In such cases, imposing the $G$-invariance condition on sheaves in $\FM(\La)$ corresponds to the folding procedure used to obtain skew-symmetrizable cluster algebras from skew-symmetric ones, as is made precise in Theorem \ref{thm: folded cluster}. See Subsection \ref{sub: folding sheaves} for a contact-geometric characterization of this cluster structure using the ingredients of \cite[Theorem 1.1]{CasalsWeng}. We compile all of the information about the cluster modular groups of Theorem \ref{thm: folded gens} and specific $G$-actions in Table \ref{tab: folded} below. 

\subsubsection{Properties of Legendrian loops}
The second goal of this work is to use properties of cluster modular groups to provide new methods for understanding Legendrian loop actions. While the intimate connection between cluster modular groups and mapping class groups is an area of ongoing research, work of Ishibashi gives a partial Nielsen-Thurston classification for cluster automorphisms \cite{Ishibashi2018}. By interpreting Legendrian loops in this context, we obtain insight into certain fixed-point properties of their actions. In addition, a process known as cluster reduction gives us a method for linking many Legendrian loops actions to mapping classes of decorated triangulations of punctured surfaces. We apply these insights to give new tests for determining if a Legendrian loop $\varphi$ produces infinitely many distinct exact Lagrangian fillings.

Given a cluster $\mathcal{A}$-space $\mathcal{A}$, the cluster-theoretic analogue of Teichm\"uller space is its positive real part $\mathcal{A}_{>0}$, obtained by declaring all cluster variables to be positive real numbers and observing that the mutation formula is subtraction free. Analogous to the fixed point property of periodic mapping classes, the following statement is an immediate corollary of Ishibashi's work on cluster modular groups.

\begin{theorem}\label{thm: fixed points}
    Let $\La$ be a Legendrian link whose sheaf moduli $\FM(\La)$ admits a cluster structure. Any finite order Legendrian loop $\varphi$ of $\La$ has a fixed point in $\FM(\La)_{>0}$. 
\end{theorem}

As an immediate corollary, we obtain the fact that any Legendrian loop action without fixed points in $\FM(\La)_{>0}$ necessarily yields infinitely many distinct exact Lagrangian fillings. See Example \ref{ex: T36FixedPoints} for a computation that recovers the fact that one of the Legendrian loops of $\La(3,6)$ considered by Casals and Gao has infinite order.

In addition to the above fixed point properties, we produce a combinatorial method for determining when a Legendrian loop has infinite order by studying the Legendrian loop action on the intersection quiver of a given exact Lagrangian filling. We apply this method to study Legendrian $\vartheta$-loops: loops of Legendrian links $\La=\La(\beta\Delta^2, i, \gamma)$ obtained by removing the $i$th strand of the $(-1)$-framed closure of $\beta\Delta^2$ and replacing it with the $(-1)$-framed closure of a positive braid $\gamma$; see \cite[Definition 2.11]{CasalsNg} or Definition \ref{def: theta loop} below for a more precise description. While all of the Legendrian links we apply this test to are already known to admit infinitely many fillings, we also show that the induced action of $\vartheta$-loops on $\FM(\La)_{>0}$ is properly discontinuous. This allows us to obtain a partial converse to Theorem \ref{thm: fixed points}.

\begin{theorem}\label{thm: fixed pt converse}
Let $\La=\La(\beta\Delta^2, i, \sigma_1^k)$ with $k\geq 3$ and assume that the induced action $\tilde{\vartheta}$ of the corresponding $\vartheta$ loop has infinite order. Then $\tilde{\vartheta}$ has no fixed points in $\FM(\La)_{>0}$. \end{theorem}

The class of Legendrian links that admit $\vartheta$-loops is quite broad, containing positive torus links $\La(k, kn)$, as well as all of the other Legendrian links in the set $\mathcal{H}$ defined above. In fact, Theorem \ref{thm: gens} can be restated using only $\vartheta$-loops as the generators of a finite index subgroup of the cluster modular group.

The key technical result that allows us to conclude Theorem \ref{thm: fixed pt converse} is the cluster reduction of the induced action of the $\vartheta$-loops in question to (tagged) mapping classes of surface-type cluster algebras. As a result, the combinatorial criterion we obtain for understanding the order of a loop can be stated as follows:

\begin{theorem}\label{thm: combinatorial criterion}
    The induced action of a Legendrian loop $\vartheta$ of $\La(\beta\Delta^2, i, \sigma_1^k)$ for $k\geq 3$ is infinite order if and only if the cluster-reduced surface type quiver is that of an infinite type cluster algebra. 
\end{theorem}

We also give a recipe for obtaining the cluster reduction of $\vartheta$ directly from the front diagram, making the criterion of Theorem $\ref{thm: combinatorial criterion}$ relatively straightforward to apply.

{\bf Organization:} Section \ref{sec: background} below reviews the necessary background material related to constructions of exact Lagrangian fillings of Legendrian links, while Section \ref{sec: sheaves} covers their sheaf-theoretic invariants. Section \ref{sec: Clusters} explains how to build a cluster structure from these elements, following \cite{CasalsWeng}. In Sections \ref{sec: CMGs} and \ref{Sec: Folding} we prove Theorems \ref{thm: gens}, \ref{thm: TorusLinkCMG}, and \ref{thm: folded gens} and compile information about the cluster modular groups associated to the specific Legendrian links in Tables \ref{tab: generating}, \ref{tab: CMG Gr}, and \ref{tab: folded}. Finally, Section \ref{sec: NielsenThurston} contains proofs of Theorems  \ref{thm: fixed points} \ref{thm: fixed pt converse}, and \ref{thm: combinatorial criterion} as well as various example computations.

{\bf Acknowledgements: } Many thanks to Roger Casals for his support and encouragement throughout this project. Thanks also to Daping Weng for patiently answering all of my questions about cluster theory, as well as to Lenny Ng, Eduardo Fern\'andez, and Orsola Capovilla-Searle for helpful conversations about Legendrian loops. Finally, thanks to Dani Kaufman for useful discussions about cluster modular groups and to Eugene Gorsky for comments on a very early draft of this work. This work was partially supported by NSF grant DMS-1942363.

\section{Contact-geometric background}\label{sec: background}

We begin with the necessary background on Legendrian links and their exact Lagrangian fillings.
 The standard contact structure $\xi_{st}$ in $\R^3$ is the 2-plane field given as the kernel of the 1-form $\alpha_{\st}=dz-ydx$. A link $\La \subseteq (\R^3, \xi_{st})$ is Legendrian if $\La$ is always tangent to $\xi_{st}$. As $\La$ can be assumed to avoid a point, we can equivalently consider Legendrians $\La$ contained in the contact 3-sphere $(\mathbb{S}^3, \xi_{st})$ \cite[Section 3.2]{Geiges08}. We consider Legendrian links up to Legendrian isotopy, i.e. ambient isotopy through a family of Legendrians. 

The symplectization $\Symp(M, \ker(\alpha))$ of a contact manifold $(M, \ker(\alpha))$ is the symplectic manifold $(\R_t\times M, d(e^t\alpha))$. Given two Legendrian links $\La_-, \La_+\subseteq (\R^3,\xi_{\st})$, 
	
 an exact Lagrangian cobordism $L\subseteq \Symp(\R^3, \ker(\alpha_{\st}))$ from $\La_-$ to $\La_+$ is a cobordism $\Sigma$ such that there exists some $T>0$ satisfying the following: 

	\begin{enumerate}
		\item $d(e^t\alpha_{\st})|_\Sigma=0$ 
		\item $\Sigma\cap ((-\infty, T]\times \R^3)=(-\infty, T]\times \La_-$ 
		\item $\Sigma\cap ([T, \infty)\times \R^3)=[T, \infty) \times \La_+$ 
		\item $e^t\alpha_{\st}|_\Sigma=df$ for some function $f: \Sigma\to \R$ that is constant on $(-\infty, T]\times \La_-$ and $[T, \infty)\times \La_+$. 
	\end{enumerate}

	An exact Lagrangian filling of the Legendrian link $\La\subseteq (\R^3, \xi_{\st})$ is an exact Lagrangian cobordism $L$ from $\emptyset$ to $\La$ that is embedded in $\Symp(\R^3, \ker(\alpha_{\st}))$. Equivalently, we consider $L$ to be embedded in the symplectic 4-ball with boundary $\partial L$ contained in contact $(\S^3, \xi_{\st})$.

We will depict a Legendrian link $\La \subseteq (\R^3, \xi_{st})$ in either of two projections; the front projection $\Pi:(\R^3, \xi_{\st}) \to \R^2$ given by $\Pi(x, y, z)=(x, z)$ or the Lagrangian projection $\pi: (\R^3, \xi_{\st})\to \R^2$ given by $\pi(x, y, z)=(x,y)$.

\subsection{Legendrian loops}\label{sub: loops defs}
Let $\La$ be the $(-1)$-framed closure of a positive braid. A Legendrian loop is, by definition, a Legendrian isotopy of $\La$ that fixes $\La$ setwise at time one. By construction, the trace of any Legendrian loop $\varphi$ produces an exact Lagrangian concordance $\Trace(\varphi):=\{\{t\}\times \la_t| 0\leq t \leq 1\}$ from $\La$ to itself in the symplectization $\Symp(\R^3, \alpha_{\st})$. One then obtains an action on the exact Lagrangian fillings of $\La$ via concatenation; given an exact Lagrangian filling $L$ of $\La$, the filling $\varphi(L)$ is defined to be the filling $L\cup_\La \Trace(\varphi)$. 

The Legendrian loops we consider here can be decomposed into elementary Legendrian isotopies. In the front projection, these elementary isotopies correspond to Reidemeister III moves, planar isotopies, and (counter)clockwise rotations of a crossing. In terms of the braid word $\beta$, these are given by the braid move $\sigma_i\sigma_{i+1}\sigma_i=\sigma_{i+1}\sigma_i\sigma_{i+1}$, commutation $\sigma_i\sigma_j=\sigma_j\sigma_i$ for $|i-j|\geq 2$, and cyclic permutation $\sigma_iw=w\sigma_i$. We refer to the latter Legendrian isotopy as a cyclic rotation and denote it by $\delta$, as in \cite{CasalsGao}.

Note that since $\Delta_n^2$ is a central in $\Br_n$, any Legendrian link given as the $(-1)$-framed closure of $\beta\Delta_n^2$ admits a Legendrian loop $\delta^{|\beta|}$ described in the $(-1)$-closure as a sequence of cyclic rotations, one for each crossing in $\beta$. We refer to such a Legendrian loop as a full cyclic rotation. The full cyclic rotation features heavily in \cite{GSW}, where the authors used it to prove the existence of infinitely many exact Lagrangian fillings of many Legendrians of the form $\La(\beta)$.

\subsubsection{$\vartheta$-loops} We consider a particular class of Legendrian loops known as $\vartheta$-loops, first appearing in the specific setting of Legendrian torus links considered by Casals and Gao in \cite{CasalsGao}. We present the more general definition given in \cite[Section 2.4]{CasalsNg}. Any braid $\beta$ yields a permutation $w(\beta)\in S_n$ via its Coxeter projection $w:\Br_n\to S_n$ encoding where strands of the braid start and end. In the case that $w(\beta)$ admits a fixed point $i\in \{1, \dots, n\}$, then the $(-1)$-framed closure of $\beta$, which we will denote by $\La$, contains a standard Legendrian unknot $\La_i$ corresponding to the $i$th strand of $\beta$. Since $\La_i$ is a Legendrian (un)knot, there is an open neighborhood $\mbox{Op}(\La_i)$, disjoint from $\La\backslash \La_i$, that is contactomorphic to $(J^1(\La_i), \xi_{\st})$. Let $\gamma \in \Br_m^+$ be a positive braid in $J^1\S^1$ and denote by $\La(\gamma)_i$ the Legendrian link obtained by satelliting $\gamma$ about $\La_i$. Further denote by $\La(\beta, i, \gamma)$ the Legendrian link $\La\backslash\La_i\cup \La(\gamma)_i$, which we will refer to as the Legendrian link obtained by satelliting $\gamma$ around the Legendrian unknot $\La_i$. 

By construction, in the open neighborhood $\mbox{Op}(\La_i)$ there is a contact flow along the radial coordinate of the companion unknot that yields a compactly supported contact isotopy $\Theta_t$ of $\La_i$. Denote by $\tilde{\Theta}_t$ the extension of $\Theta_t$ by the identity to all of $\R^3$.

\begin{definition}\label{def: theta loop}
    Let $\La=\La(\beta, i, \gamma)$ be a Legendrian satellite link with $\gamma\in \Br_m^+$. A $\vartheta$-loop of $\La$ is a compactly supported contact isotopy $\tilde{\Theta}_t$ fixing $\La(\gamma)_i$ setwise by time one.
\end{definition}

 See Figure \ref{fig:LoopEx} for an example of a $\vartheta$-loop depicted in $J^1\S^1$. Unless otherwise specified, we will impose the additional condition that $\gamma\in \Br_2^+$ for the remainder of this article. 

\subsubsection{K\'alm\'an loops}
We consider a class of loops of Legendrian torus links originally studied in \cite{Kalman}. K\'alm\'an showed that the loops are nontrivial elements of $\pi_1(\Leg(K(n,k)))$ by studying their induced action on the augmentation variety of $\La(k, n)$. These loops are readily understood by observing that for $\La(k, n)$, the full cyclic rotation $\delta^{(k-1)n}$ factors as $(\delta^{k-1})^n$. More precisely, in the front given as the $(-1)$-framed closure of $\beta=(\sigma_1\dots \sigma_{k-1})^{n+k}$, the K\'alm\'an loop $\rho=\delta^{k-1}$ is defined to be the Legendrian loop given by cyclic rotation of the first $k-1$ crossings. K\'alm\'an's main result \cite[Theorem 1.3]{Kalman} is that the induced action of this loop on the augmentation variety of $\La(k, n)$ has order $n+k$.

\subsubsection{Donaldson-Thomas transformations} Finally, we define the Donaldson-Thomas transformation, denoted by $\DT$. While the DT transformation is not known to be induced by a Legendrian loop, its square $\DT^2$ is Hamiltonian isotopic to the full cyclic rotation in the case of rainbow closures $\La(\beta)$ \cite[Lemma 4.2]{GSW}. For a more general class of Legendrians, Casals and Weng describe a Legendrian isotopy and a strict contactomorphism that together induce the $\DT$ transformation \cite[Section 5]{CasalsWeng}. For rainbow closures $\La(\beta)$, this procedure can be described as first performing half of a full cyclic rotation $(\delta^{|\beta|})^{1/2}$ by consecutively rotating all of the crossings past the left cusps in a counterclockwise direction. This isotopy is then composed with the strict contactomorphism $x\mapsto -x, z\mapsto -z$ to return to the original front. See Subsection \ref{sub: LoopsFromFences} for an alternative definition in terms of the combinatorics of plabic fences.

\subsection{Legendrian weaves}\label{sec:weaves}

Let us now describe Legendrian weaves, a geometric construction of Casals and Zaslow that can be used to produce exact Lagrangian fillings of a Legendrian link \cite{CasalsZaslow}. The key idea of their construction is to combinatorially encode a \emph{Legendrian} surface $\w$ in the 1-jet space $J^1\D^2=T^*\D^2\times \R_z$ by the singularities of its front projection in $\D^2\times \R_z$. The Lagrangian projection of $\w$ then yields an exact Lagrangian surface in $T^*\D^2$.

The contact geometric setup of the Legendrian weave construction is as follows. We construct a filling of $\La$ by first describing a local model for a Legendrian surface $\w$ in $J^1\D^2=T^*\D^2\times \R_z$. We equip $T^*\D^2$ with the symplectic form $d(e^r\alpha)$ where $\ker(\alpha)=\ker(dy_1-y_2d\theta)$ is the standard contact structure on $J^1(\partial \D^2)$ and $r$ is the radial coordinate. This choice of symplectic form ensures that the flow of $e^r\alpha$ is transverse to $J^1\S^1\cong \R^2\times \partial \D^2$ thought of as the cotangent fibers along the boundary of the 0-section. The Lagrangian projection of $\w$ is then a Lagrangian surface in $(T^*\D^2, d(e^r\alpha))$. Moreover, since $\w\subseteq (J^1\D^2, \ker (dz-e^r\alpha))$ is a Legendrian, we immediately obtain the function $z:\pi(\w)\to \R$ satisfying $dz=e^r\alpha|_{\pi(\w)}$, demonstrating that $\pi(\w)$ is exact. 

The boundary of $\pi(\w)$ is taken to be a positive braid $\beta$ in $J^1\S^1$  so that we may regard it as a Legendrian link in a contact neighborhood of $\partial \D^2$. As the 0-section of $J^1\S^1$ is Legendrian isotopic to a max-tb standard Legendrian unknot, we can take $\partial\pi(\w)$ to equivalently be the standard satellite of the standard Legendrian unknot. Diagramatically, this implies that the braid $\beta$ in $J^1\S^1$ can be given as the $(-1)$-framed closure of $\beta$ in contact $\S^3$.

\subsubsection{$N$-Graphs and Singularities of Fronts} To construct a Legendrian weave surface $\w$ in $J^1\D^2,$ we combinatorially encode the singularities of its front projection in a colored graph. 
Local models for these singularities of fronts are classified by work of Arnold \cite[Section 3.2]{ArnoldSing}. The three singularities that appear in our construction describe elementary Legendrian cobordisms and are pictured in Figure \ref{fig: wavefronts}.

	\begin{center}
		\begin{figure}[h!]{ \includegraphics[width=.8\textwidth]{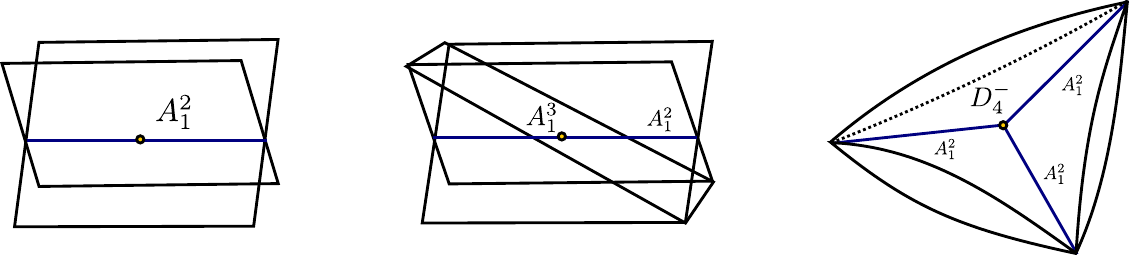}}\caption{Singularities of front projections of Legendrian surfaces. Labels correspond to notation used by Arnold in his classification.}
			\label{fig: wavefronts}\end{figure}
	\end{center}

Since the boundary of our singular surface $\Pi(\w)$ is the front projection of an $N$-stranded positive braid, $\Pi(\w)$ can be pictured as a collection of $N$ sheets away from its singularities. We describe the behavior at the singularities as follows:

\begin{enumerate}
    \item The $A_1^2$ singularity occurs when two sheets in the front projection intersect. This singularity can be thought of as the trace of a constant Legendrian isotopy in the neighborhood of a crossing in the front projection of the braid $\beta\Delta^2$. 
    \item The $A_1^3$ singularity occurs when a third sheet passes through an $A_1^2$ singularity. This singularity can be thought of as the trace of a Reidemeister III move in the front projection.
    \item A $D_4^-$ singularity occurs when three $A_1^2$ singularities meet at a single point. This singularity can be thought of as the trace of a 1-handle attachment in the front projection. 
\end{enumerate}

Having identified the singularities of fronts of a Legendrian weave surface, we encode them by a colored graph $\Gamma\subseteq \D^2$. The edges of the graph are labeled by Artin generators of the braid and we require that any edges labeled $\sigma_i$ and $\sigma_{i+1}$ meet  at a hexavalent vertex with alternating labels while any edges labeled $\sigma_i$ meet at a trivalent vertex. To obtain a Legendrian weave $\w(\Gamma)\subseteq (J^1\D^2,\xi_{\st})$ from an $N$-graph $\Gamma$, we glue together the local germs of singularities according to the edges of $\Gamma$. First, consider $N$ horizontal sheets $\D^2\times \{1\}\sqcup \D^2\times \{2\}\sqcup \dots \sqcup \D^2\times \{N\}\subseteq \D^2\times \R$ and an $N$-graph $\Gamma\subseteq \D^2\times \{0\}$. We construct the associated Legendrian weave $\w(\Gamma)$ as follows:

	\begin{itemize}
		\item Above each edge labeled $\sigma_i$, insert an $A_1^2$ crossing between the $\D^2\times\{i\}$ and $\D^2\times \{i+1\}$ sheets so that the projection of the $A_1^2$ singular locus under $\pi:\D^2\times \R \to \D^2\times \{0\}$ agrees with the edge labeled $\sigma_i$.   
		\item At each trivalent vertex $v$ involving three edges labeled by $\sigma_i$, insert a $D_4^-$ singularity between the sheets $\D^2\times \{i\}$ and $\D^2\times\{i+1\}$ in such a way that the projection of the $D_4^-$ singular locus agrees with $v$ and the projection of the $A_2^1$ crossings agree with the edges incident to $v$.
		\item At each hexavalent vertex $v$ involving edges labeled by $\sigma_i$ and $\sigma_{i+1}$, insert an $A_1^3$ singularity along the three sheets in such a way that the origin of the $A_1^3$ singular locus agrees with $v$ and the $A_1^2$ crossings agree with the edges incident to $v$.
	\end{itemize} 
	\begin{center}
		\begin{figure}[h!]{ \includegraphics[width=\textwidth]{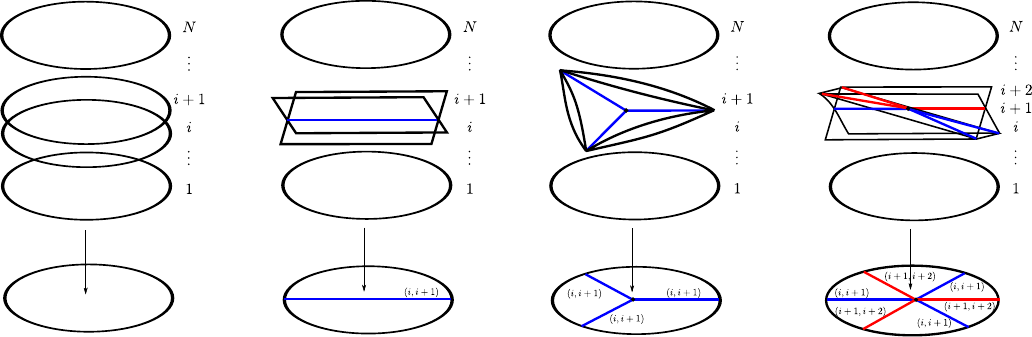}}\caption{The weaving of singularities of fronts along the edges of the $N$-graph. Gluing these local models according to the $N$-graph $\Gamma$ yields the weave $\w(\Gamma)$.}
			\label{fig:Weaving}\end{figure}
	\end{center}
	
	If we take an open cover $\{U_i\}_{i=1}^m$ of $\D^2\times \{0\}$ by open disks, refined so that any disk contains at most one of these three features, we can glue together the resulting fronts according to the intersection of edges along the boundary of our disks. Specifically, if $U_i\cap U_j$ is nonempty, then we define $\Pi(\w(U_1\cup U_2))$ to be the front resulting from considering the union of fronts $\Pi(\w(U_1))\cup\Pi(\w(U_j))$ in $(U_1\cup U_2)\times \R$.

	\begin{definition}\label{def: weave}
	    The Legendrian weave $\w(\Gamma)\subseteq (J^1\D^2, \xi_{st})$ is the Legendrian lift of the front $\Pi(\w(\cup_{i=1}^m U_i))$ given by gluing the local fronts of singularities together according to the $N$-graph $\Gamma$.
	\end{definition}

	The immersion points of a  Lagrangian projection of a weave surface $\w$ correspond precisely to the Reeb chords of $\w$. In particular, if $\w$ has no Reeb chords, then its Lagrangian projection $L(\w)$ is an embedded exact Lagrangian filling of $\partial(\w)$. In the Legendrian weave construction, Reeb chords correspond to critical points of functions giving the difference of heights between sheets. Every weave surface in this work admits an embedding where the distance between the sheets in the front projection grows monotonically in the direction of the boundary, ensuring that there are no Reeb chords.


	\subsubsection{Homology of Weaves}

In this subsection, we describe the homology of a Legendrian weave $\w(\Gamma)$. The smooth topology of $\w(\Gamma)$ is that of an $N$-fold branched cover over $\D^2$ with simple branched points corresponding to each of the trivalent vertices of $\Gamma$. 
	Assuming that $\w(\Gamma)$ is connected, the genus of $\w(\Gamma)$ is then computed using the Riemann-Hurwitz formula: 
	$$
	g(\w(\Gamma))=\frac{1}{2}(v(\Gamma)+2-N\chi(\D^2)-|\partial\w(\Gamma)|)$$
	where $v(\Gamma)$ is the number of trivalent vertices of $\Gamma$ and $|\partial\w(\Gamma)|$ denotes the number of boundary components of $\Gamma$.


We now describe a recipe for combinatorially identifying particular elements of $H_1(\w(\Gamma));\Z)$ known as $\mathbb{L}$-compressing cycles. These are cycles such that their Lagrangian projection in $L(\w(\Gamma))$ bounds an embedded Lagrangian disk in the complement of $L(\w(\Gamma))$. As described in Section \ref{sec: Clusters}, $\mathbb{L}$-compressing cycles play a crucial role in defining the cluster seed associated to $L(\w)$. We first consider an edge connecting two trivalent vertices. Closely examining the sheets of our surface, we can see that each such edge corresponds to a 1-cycle. We refer to such a 1-cycle as a short {\sf I}-cycle. Similarly, any three edges of the same color that connect a single hexavalent vertex to three trivalent vertices correspond to another 1-cycle, which we refer to as a short {\sf Y}-cycle. See Figure \ref{fig:ShortCycles}  for examples of a short {\sf I}-cycle (left) and a short {\sf Y}-cycle (right). We can also consider a sequence of edges starting and ending at trivalent vertices and passing directly through any number of hexavalent vertices, as pictured in Figure \ref{fig:longI-cycle}. Such a cycle is referred to as a long {\sf I}-cycle. Finally, we can combine any number of {\sf I}-cycles and short {\sf Y}-cycles to describe a wide class of 1-cycles as trees with leaves on trivalent vertices and edges passing directly through or branching at hexavalent vertices.

 	\begin{center}
		\begin{figure}[h!]{ \includegraphics[width=.45\textwidth]{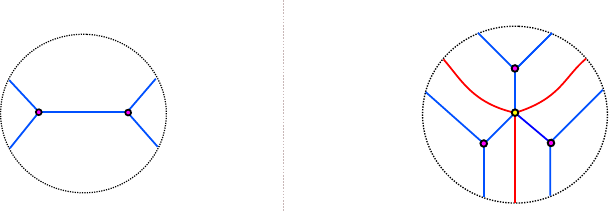}}\caption{Two local models of $\mathbb{L}$-compressing cycles in $H_1(\mathfrak{w})$.
			}
			\label{fig:ShortCycles}\end{figure}
	\end{center}

	\begin{center}
		\begin{figure}[h!]{ \includegraphics[width=.95\textwidth]{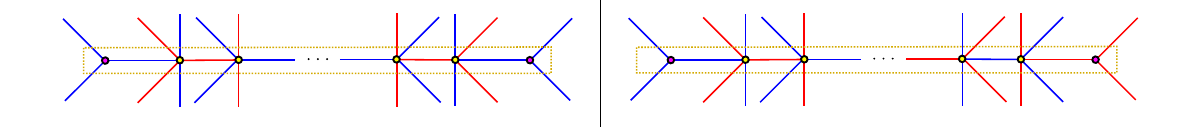}}\caption{A pair of long {\sf I}-cycles. The cycle on the left passes through an even number of hexavalent vertices, while the cycle on the right passes through an odd number. 
			}
			\label{fig:longI-cycle}\end{figure}
	\end{center}

	The	intersection form $\langle \cdot, \cdot \rangle$ on $H_1(\w(\Gamma))$ plays a key role in understanding cluster structures coming from Lagrangian fillings. If we consider a pair of 1-cycles $\gamma_1, \gamma_2\in H_1(\w(\Gamma))$ with a fixed orientation and a nonempty geometric intersection in $\w$, then we can assign a sign to their intersection. 
 We refer to the signed count of the intersection of $\gamma_1$ and $\gamma_2$ as their algebraic intersection and denote it by $\langle \gamma_1, \gamma_2\rangle.$ 
	\subsubsection{Mutations of weaves}\label{sub: mutation}

	In this section we describe Legendrian mutation, a geometric operation that we can use to generate additional Legendrian weaves. Given a Legendrian weave $\w$ and a 1-cycle $\gamma\in H_1(\w;\Z)$ that bounds an embedded Lagrangian disk in the complement $T^*\D^2\backslash L(\w)$, the Legendrian mutation $\mu_\gamma(\w)$ outputs a Legendrian weave smoothly isotopic to $\w$ but that is generally not Legendrian isotopic to $\w$.

	Combinatorially, we can describe Legendrian mutation in terms of the $N$-graph associated to a weave. Figure \ref{fig:Mutations} (left) depicts mutation at a short {\sf I}-cycle, while Figure \ref{fig:Mutations} (right) depicts mutation at a short {\sf Y}-cycle. See \cite[Section 4.9]{CasalsZaslow} for a more general description of mutation at long {\sf I}- and {\sf Y}-cycles in $N$-graphs. The geometric operation above coincides with the combinatorial manipulation of the $N$-graphs. Specifically, for any two $N$-graphs, $\Gamma$ and $\Gamma'$ related by either of the combinatorial moves described in Figure \ref{fig:Mutations}, the corresponding Legendrian weaves $\w(\Gamma)$ and $\w(\Gamma')$ are mutation-equivalent relative to their boundary \cite[Theorem 4.2.1]{CasalsZaslow}.
	
	\begin{center}\begin{figure}[h!]{ \includegraphics[width=.9\textwidth]{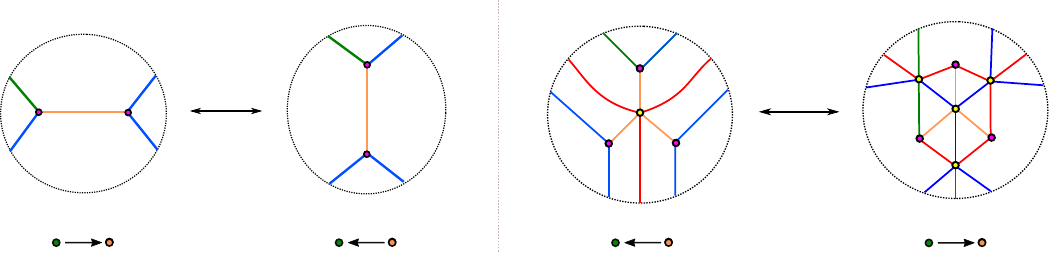}}\caption{Mutations at the orange short {\sf I}-cycle (left) and short {\sf Y}-cycle (right) together with the associated intersection quivers. In both cases, the dark green edge depicts the effect of mutation on an arbitrary cycle intersecting the orange cycle.}\label{fig:Mutations}\end{figure}
	\end{center}

	\subsubsection{Quivers from Weaves}\label{quivers}
	
	We complete our discussion of Legendrian weaves by describing quivers and how they arise via the intersection form on $H_1(\w;\Z).$ A quiver is a directed graph without loops or oriented 2-cycles. In the Legendrian weave setting, the data of a quiver can be extracted from a given weave and a basis of its first homology via the intersection form. The intersection quiver is defined as follows: for every basis element $\gamma_i\in H_1(\w(\Gamma);\Z)$ we have a vertex $v_i$ in the quiver; there are $k$ arrows pointing from $v_j$ to $v_i$ if $\langle \gamma_i, \gamma_j\rangle=k$ for $k>0$.

	The combinatorial operation of quiver mutation at a vertex $v$ is defined as follows, see e.g. \cite[Definition 2.1.2]{FWZ1}. First, for every pair of incoming edges and outgoing edges, we add an edge starting at the tail of the incoming edge and ending at the head of the outgoing edge. Next, we reverse the direction of all edges adjacent to $v$. Finally, we cancel any directed 2-cycles. If we started with the quiver $Q$, then we denote the quiver resulting from mutation at $v$ by $\mu_v(Q).$

	The following theorem relates the two operations of quiver mutation and Legendrian mutation:

	\begin{theorem}[\cite{CasalsZaslow}, Section 7.3]
		Given an $N$-graph $\Gamma$, Legendrian mutation at an embedded cycle $\gamma$ induces a quiver mutation of the associated intersection quivers, taking $Q(\Gamma, \{\gamma_i\})$ to $\mu_\gamma(Q(\Gamma, \{\gamma_i\})).$ \hfill $\Box$
	\end{theorem}

	\section{Microlocal theory of sheaves}\label{sec: sheaves}

	In this section we define the moduli stacks $\FM(\La, T)$ and $\SM_1(\La)$ of microlocal rank-one sheaves with prescribed singular support and describe the sheaf-theoretic components of the cluster $\mathcal{A}$ and $\mathcal{X}$ spaces considered in \cite{CasalsWeng}.

\subsection{Sheaves on Legendrian links}

We discuss here the microlocal theory of sheaves as it appears in the setting of Legendrian links and their exact Lagrangian fillings. In full generality, the sheaf moduli we consider are constructed as the moduli of certain objects of a particular subcategory of the dg-derived category of sheaves of $k$-modules with singular support contained in a given Legendrian link. For a brief review of the appropriate categorical framework, see \cite[Section 2.7]{CasalsWeng} or \cite[Appendix A]{CasalsLi}.

Given a constructible sheaf $F\in Sh(M)$, the singular support $SS(F)$ of $F$ is a conical Lagrangian in $T^*M$. Quotienting by the natural $\R$-action yields a Legendrian in the unit-cotangent bundle $T^\infty M$. Fixing a co-orientation along any front projection $\Pi_{xz}(\La) \subseteq \R^2$ of a Legendrian $\La$ allows us to equivalently consider sheaves $F$ with $SS(F)$ as a subset of $\R^3$ or $J^1\S^1$. In \cite{STZ_ConstrSheaves}, the authors analyze singular support conditions in terms of a stratification of $\R^2$ induced by a front projection. By \cite[Proposition 5.17]{STZ_ConstrSheaves}, as all Legendrian links $\La\subseteq (\R^3, \xi_{\st})$ we consider admit front projections in $(J^1\S^1, \xi_{\st})$ without cusps, we may consider sheaves of vector spaces concentrated in degree 0 rather than working with sheaves of chain complexes.
We then obtain a combinatorial description of sheaves $F$ with singular support $SS(F)\subseteq \La$ as follows. For $\Pi_{xz}(\La)$ a front of a Legendrian satisfying the conditions given above, we associate to each region of $\R^2\backslash\Pi_{xz}(\La)$ a vector space -- the stalk of $F$ within that region -- such that the unbounded region of $\Pi_{xz}(\La)$ is assigned the zero vector space. The appropriate singular support conditions are encoded by conditions on the maps induced by restriction near arcs, crossings, and cusps.

\begin{proposition}[\cite{STZ_ConstrSheaves}] \label{prop: STZ}
    Let $\La$ be a front with binary Maslov potential. A sheaf $F$ on $\R^2$ satisfying $SS(F)\subseteq \La$ is given by the following 
    conditions on the vector spaces assigned to each region of $\Pi_{xz}(\La)$.

    \begin{enumerate}
        \item The zero vector space is assigned to the unbounded region.
        \item The vector spaces assigned to any two regions separated by an arc differ in dimension by $r$.
        \item At each cusp, the composition of maps in Figure \ref{fig:SingSupp} is the identity.
        \item At each crossing, the sequence $0\to S\to E\oplus W \to N\to 0$ is exact and the four maps form a commuting square. 
    \end{enumerate}
\end{proposition}

	\begin{center}\begin{figure}[h!]{ \includegraphics[width=.8\textwidth]{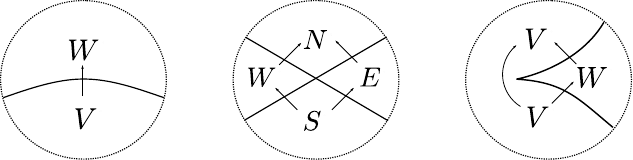}}\caption{Local models of a front projection of $\La$ and the maps induced by restriction of sheaves. }\label{fig:SingSupp}\end{figure}
	\end{center}

 Such sheaves are necessarily constructible with respect to a stratification induced by a front projection of $\La$. The crossing condition (4) can also be phrased as the existence of an isomorphism between the cokernels of the maps $S\to E$ and $W\to N$ -- equivalently, $S\to W$ and $E\to N$. This isomorphism at each crossing of $\Pi_{xz}(\La)$ yields a rank-$r$ local system on $\La$ that we refer to as the microlocal monodromy of $\La$. The number $r$ is the microlocal rank of $F$. We require $r=1$ and denote by $\SM_1(\La)$ the moduli of microlocal rank-one sheaves $F$ with $SS(F)\subseteq \La$. By \cite{GKS_Quantization}, $\SM_1(\La)$ is a Legendrian isotopy invariant of $\La$.
 
 Following Proposition \ref{prop: STZ}, we can describe the sheaf moduli as a space of flags with certain transversality conditions given by the front projection as follows. Let $\La\subseteq J^1\S^1$ be the $(-1)$-framed closure of an $N$-stranded positive braid $\beta\in \Br_N^+$. The crossing singularities of the front $\Pi_{xz}(\La)\subseteq S^1\times \R$ divide $S^1\times \R$ into regions $[a_i, a_{i+1}]\times \R$. Given our description of $\SM_1(\La)$ above, each sheaf $F\in \SM_1(\La)$ can be characterized as the assignment of a locally constant vector space in each region of $[a_i, a_{i+1}]\times \R$, divided by the strands of $\beta$. The microlocal rank-one condition implies that the rank of these vector spaces increases by one as we pass from a region to the one immediately above it. This sequence of vector spaces $0\subseteq V_i^{1}\subseteq \dots \subseteq V_i^{n}=\C^n$ can be thought of as a flag $V_i^{\bullet}$ in $\C^N$. The singular support condition at the crossing corresponding to the Artin generator $\sigma_j$ implies that the two flags $V_{i}^\bullet$ and $V_{i+1}^{\bullet}$ differ at the $j$th position; in other words, $V_{i}^{j}$ is transverse to $V_{i+1}^{j}$. The space $\SM_1(\La)$ can therefore be understood as the space of flags satisfying these transversality conditions, modulo a choice of basis.

\subsection{Sheaves on Legendrian weaves}

In order to build cluster structures on $\SM_1(\La)$ and the decorated sheaf moduli $\FM(\La)$ described in the following subsection, we must explain how an embedded exact Lagrangian filling of $\La$ yields a toric chart $(C^\times)^{n}$ in the (decorated) sheaf moduli of $\La$. By work of Jin and Treumann, an embedded exact Lagrangian filling $L$ of $\La$ has sheaf moduli equivalent to the space of local systems on $L$ via the microlocal monodromy functor applied to the Legendrian lift of $L$ \cite[Section 1.7]{JinTreumann17}. Restricting to rank-one local systems on $L$ yields an embedding of a toric chart $(\C^\times)^{b_1(L)}\hookrightarrow \SM_1(\La)$ where $b_1(L)$ denotes the first Betti number of $L$. We can compute these local systems explicitly by analyzing the singularities of the fronts of Legendrian weaves.

 As described in \cite[Section 5.3]{CasalsZaslow} the data of $\SM_1(\w(\Gamma))$ is equivalent to providing: \begin{enumerate}
		\item[(i)] An assignment to each region $R$ (connected component of $\D^2\backslash \Gamma$) of a flag $V^\bullet(R)$ in the vector space $\C^N$.
		\item[(ii)] For each pair $R_1, R_2$ of adjacent regions sharing an edge labeled by $\sigma_i$, we require that the corresponding flags satisfy 
		$$V^j(R_1)=V^j(R_2), \qquad 0\leq j\leq N, j\neq i, \qquad \text{ and } \qquad V^i(R_1)\neq V^i(R_2).$$  
	\end{enumerate}
	Finally, we consider the space of flags satisfying (i) and (ii) modulo the diagonal action of $GL_N(\C)$ on $V^\bullet$. By \cite[Theorem 5.3]{CasalsZaslow}, the flag moduli space is isomorphic to the space of microlocal rank-one sheaves $\SM_1(\w(\Gamma))$. 
 
 To better understand local systems on $\w(\Gamma)$, we give examples of the flag moduli space in a neighborhood of homology cycles of $\w(\Gamma)$. In the short {\sf I}-cycle case, when the edges are labeled by $\sigma_1$, the moduli space is determined by four lines $a\neq b \neq c\neq d\neq a$, as pictured in Figure \ref{fig:flags} (left). 
 Around a short {\sf Y}-cycle, the data of the flag moduli space is given by three distinct planes $A\neq B \neq C\neq A $ contained in $\C^3$ and three distinct lines $a\subsetneq A, b\subsetneq B, c\subsetneq C$ with $a\neq b\neq c\neq a,$ as pictured in Figure \ref{fig:flags} (right).
	
	\begin{center}\begin{figure}[h!]{ \includegraphics[width=.6\textwidth]{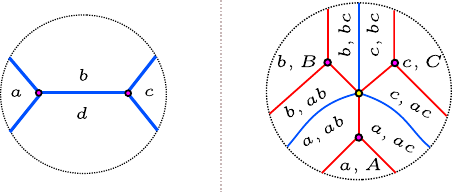}}\caption{The data of the flag moduli space given in the neighborhood of a short {\sf I}-cycle (left) and a short {\sf Y}-cycle (right). Lines are represented by lowercase letters, while planes are written in uppercase. The intersection of the two lines $a$ and $b$ is written as $ab$. } 
			\label{fig:flags}\end{figure}
	\end{center}

As in the case of a Legendrian $\La$ in $J^1\S^1$, Condition (4) of Proposition \ref{prop: STZ} allows us to compute the microlocal monodromy about an absolute cycle $\gamma\in H_1(L(\w))$ as a composition of isomorphisms between cokernels of maps of vector spaces. Since it is locally defined, we can compute the microlocal monodromy about an {\sf I}-cycle or {\sf Y}-cycle using the data of the flag moduli space in a neighborhood of the cycle. If we have a short {\sf I}-cycle $\gamma$ with flag moduli space described by the four lines $a, b, c, d$, as in Figure \ref{fig:flags} (left), then the microlocal monodromy about $\gamma$ is given by the cross ratio 
	$$\frac{a\wedge b}{b\wedge c}\frac{c\wedge d}{d\wedge a}$$
	Similarly, for a short {\sf Y}-cycle with flag moduli space given as in Figure \ref{fig:flags} (right), the microlocal monodromy is given by the triple ratio
	$$\frac{B(a)C(b)A(c)}{B(c)C(a)A(b)}$$ 

where we interpret the plane $B$ as a covector in $\C^3$ to define the pairing $B(a)$. 


\subsubsection{The decorated sheaf moduli}

In preparation for describing cluster $\mathcal{A}$ spaces in the next section, we introduce the decorated sheaf moduli $\FM(\La, T)$. To do so, we must first specify a trivialization of the abelian local system given by the microlocal monodromy about each component of $\La$. Let $T=\{t_1, \dots, t_k\}$ be a set of marked points on $\La$ with every component of $\La$ carrying at least one marked point. Label the connected components of $\La\backslash T$ by the pair of endpoints of each segment $(t_i, t_{i+1})$ where indices are taken modulo the number of marked points on the appropriate component. The decorated moduli stack $\FM(\La, T)$ is then given by the additional data of a trivialization for every segment of each component of $\La$.

\begin{equation*}
    \FM(\La, T):=\{(F, \phi_1, \dots, \phi_k) | F\in \SM_1(\La), \phi_i \text{ is a trivialization of } \m_\La \text{ on } (t_i, t_{i+1})\}
\end{equation*}

Two framings are equivalent if they differ by a global factor of $\C^\times$. As in \cite[Section 2.8]{CasalsWeng} we define the moduli space $\SM_1(\La, T)$ analogously by the addition of the trivialization data associated with the marked points. In this work, we choose to suppress $T$ from our notation when our discussion does not depend on the presence of marked points.

The data of these trivializations have two primary motivations. First, it allows us to describe $\FM(\La, T)$ as a smooth affine scheme rather than an Artin stack. Second, the trivializations give us the necessary data for defining microlocal merodromies as a microlocal version of parallel transport, yielding regular functions on $\FM(\La, T)$, as we now describe.

To describe cluster-$\mathcal{A}$ coordinates on $\FM(\La, T)$, we start with a basis $\{\gamma_i\}_{i=1}^n$ of $H_1(L, T)$ containing a maximal linearly independent subset of $\mathbb{L}$-compressing cycles of $L$. We then identify the lattice $H_1(L, T)$ with an isomorphic lattice $H_1(L\backslash T, \La\backslash T)$, and consider the dual basis of cycles $\{\gamma_i^\vee\}_{i=1}^n$.\footnote{In order to more naturally incorporate frozen cluster variables, one can also consider intermediate lattices $M$ and $N$ satisfying $H_1(L) \hookrightarrow N \hookrightarrow H_1(L, T)$ and $H_1(L, \La)\twoheadleftarrow M \twoheadleftarrow H_1(L\backslash T, \La\backslash T)$. See \cite[Remark 3.50]{CasalsWeng} for more details.} 

Given an oriented relative cycle $\gamma^\vee\in H_1(L\backslash T, \La\backslash T)$ starting at $s$ and ending at $t$ we first note that the framing data of $\FM(\La, T)$ specifies two vectors $\phi_{s} \in \Phi_{s}$ and $\phi_{t}\in \Phi_{t}$ at $s$ and $t$. The result of parallel transport along $\gamma^\vee$ yields a nonzero vector $\gamma^\vee(\phi_{s})\in \Phi_{t}$.\footnote{In this section, we omit the discussion of sign curves appearing in \cite[Section 4.5]{CasalsWeng}.} 

\begin{definition}
    Given an $\mathbb{L}$ compressing cycle $\gamma\in H_1(L, T)$, the microlocal merodromy along its dual relative cycle $\gamma^\vee\in H_1(L\backslash T, \La\backslash T)$ is the function $A_{\gamma^\vee}=\gamma^\vee(\phi_{s})/\phi_{t}$.

\end{definition}

One can compute $A_{\gamma^\vee}$ explicitly from a Legendrian weave $\w$ by lifting the framing data to a set of decorations. Given a flag $V^\bullet$ with framing data $\phi_i=V^i/V^{i-1}$, we can construct a volume form $\alpha_i\in \bigwedge^i V^i$ by first lifting $\phi_i$ to a nonzero vector $\tilde{\phi}_i\in V_i$ and then setting $\alpha_i=\tilde{\phi}_i\wedge \dots \wedge \tilde{\phi}_1.$ At an edge of $\w$ labeled by $\sigma_i$, we have flags $\mathcal{L}^\bullet$ and $\mathcal{R}^\bullet$ to the left and right of the $\sigma_i$ edge with framing data $\{\la_i\}$ and $\{\rho_i\}$, respectively. Denote by $\{\alpha_i\}$ and $\{\beta_i\}$ the decorations corresponding to the framings on $\mathcal{L}^\bullet$ and $\mathcal{R}^\bullet$. The parallel transport of $\la_i$ along an oriented curve $\eta$ from the $i$th sheet on the left to the $i+1$st sheet on the right can then be computed by 
\begin{equation}\label{eq: merodromy1}\eta(\la_i)=\left(\frac{\tilde{\la}_i\wedge \tilde{\rho}_{i}\wedge \alpha_{i-1}}{\tilde{\rho}_{i+1}\wedge\beta_i}\right)\rho_{i+1}.\end{equation}
Similarly, the parallel transport of $\la_{i+1}$ along the oriented curve $\eta'$ from the $i+1$st sheet on the left to the $i$th sheet on the right yields 

\begin{equation}\label{eq: merodromy2}\eta(\la_{i+1})=\left(\frac{\alpha_{i+1}}{\tilde{\rho}_i\wedge \tilde{\la}_{i}\wedge \alpha_{i-1}}\right)\rho_{i}.\end{equation}

Composing Equations \ref{eq: merodromy1} and \ref{eq: merodromy2}, allows us to compute the microlocal merodromy along any oriented curve in $\w$. By \cite[Proposition 4.29]{CasalsWeng}, the microlocal merodromy $A_{\gamma^\vee}$ along a relative cycle $\gamma^\vee \in H_1(\w\backslash T, \La\backslash T)$ dual to an $\mathbb{L}$-compressing cycle $\gamma$ is a regular function on $\FM(\La, T)$. As we explain in the following section, this regular function is actually a cluster variable, allowing us to define a cluster-$\mathcal{A}$ structure on $\FM(\La, T)$.

\section{Cluster algebras and cluster ensembles}\label{sec: Clusters}
	In this section, we define cluster algebras and cluster varieties and describe how they appear in a contact-geometric context. We start by introducing the basic definitions and then discuss some specific cluster algebras related to tagged triangulations of surfaces. The subsection ends with a description of cluster modular groups. See \cite{FWZ1, FWZ2} for an introductory reference on cluster theory.  

\subsection{Cluster algebras}

To an initial quiver $Q_0$ with $n$ vertices, we associate an initial set of variables $a_1, \dots a_n$, one for each vertex. Together, the $n-$tuple $\bf{a}=(a_1, \dots, a_n)$ and the quiver $Q_0$ form a cluster seed $(\bf{a}, Q_0)$. We designate a subset of vertices $Q_0^{mut}$ to be the mutable part of the quiver. The vertices in $Q_0\backslash Q_0^{mut}$ are designated as frozen and we require that there are no arrows between them. For a cluster seed $(\bf{x}, Q)$ we can denote by $b_{ij}$ the multiplicity of arrows in $Q$ from vertex $i$ to vertex $j$. We then obtain a skew-symmetric matrix, $B=(b_{ij})$, known as the exchange matrix, encoding the information of the quiver.
There are two types of cluster algebras, type $\mathbb{A}$ or type $\mathbb{X}$, depending on the precise form of the mutation formula relating different cluster variables.

\begin{definition}
    Let $(\bf{a}, Q)$ be a cluster seed and $k\in Q^{mut}$ be a mutable vertex. The cluster $\mathbb{A}$ seed mutation $\mu_k$ is an operation taking as input the seed $(\bf{a}, Q)$ and outputting the new seed $(\bf{a}', Q')$ where $Q'$ is related to $Q$ by quiver mutation at vertex $k$ and $\bf{a}'$ is related to $\bf{a}$ by $a_i'=a_i$ for all $i\neq k$ and 
  
   $$a_ka_k'=\prod_{b_{ik}>0} a_i^{b_{ik}} + \prod_{b_{ik<0}} a_i^{-b_{ik}}.$$

\end{definition}

Note that seed mutation is an involution, so that $\mu_k^2(\bf{a}, Q)=(\bf{a}, Q)$.

Denote by $\mathcal{F}$ the field of rational functions $\C(a_1, \dots a_n)$ and consider an initial seed $(\bf{a}, Q_0)\subseteq \mathcal{F}$. 
\begin{definition}
The type $\mathbb{A}$ cluster algebra generated by $(\bf{a}, Q_0)$ is the $\C$-algebra generated by all cluster variables arising in arbitrary mutations of the initial seed.     
\end{definition}

The type $\mathbb{X}$ cluster algebra is generated from an initial seed $(\bf{x}, Q_0)$ by the mutation formula 

$$x_j'=\begin{cases}
    x_j^{-1} & i=j\\
    x_j(y_k+1)^{-b_{kj}} & j\neq k, b_{kj}\leq 0 \\
    x_j(x_k^{-1}+1)^{-b_{kj}} & j \neq k, b_{kj}\geq 0
\end{cases} $$

A cluster algebra is of finite type if it has only finitely many distinct cluster seeds. Otherwise, it is of infinite type. Cluster algebras admit an ADE classification.

\begin{theorem}[Theorem 1.4, \cite{FominZelevinsky_ClusterII}]
Cluster algebras are of finite type if and only if their quiver is mutation equivalent to a Dynkin diagram of finite type with any orientation given to its edges     
\end{theorem}

Two quivers are mutation equivalent if one can be obtained by applying a sequence of quiver mutations to the other. For the simply-laced cases, this classification restricts our attention to ADE-type. We discuss some of the combinatorial ingredients for understanding cluster algebras of types $A_n$ and $D_n$ in Subsection \ref{sub: clustercombinatorics}. Beyond  cluster algebras of finite type, the next simplest families are cluster algebras arising from quivers of finite mutation type. These are cluster algebras with an underlying quiver that is mutation equivalent to only finitely many quivers. Cluster algebras from quivers of finite mutation type are classified in \cite{FeliksonShapiroTumarkin}.  
Among the finite mutation type cluster algebras, we have types $\tilde{A}_n, \tilde{D}_n,$ and others corresponding to triangulations of surfaces (See e.g. \cite{FominShapiroThurston}), as well as types $\tilde{E}_6, \tilde{E}_7, \tilde{E}_8, E_6^{(1,1)}, E_7^{(1,1)}, E_8^{(1,1)}$ and two additional exceptional quivers.

\subsection{Cluster ensembles from Legendrian links}

Following \cite{FockGoncharovClusterEnsemble}, 
we give a brief description of the notion of a (skew-symmetric)\footnote{For the skew-symmetrizable case, see Section \ref{Sec: Folding}.} cluster ensemble. We then give the necessary ingredients to understand how this structure arises from the pair of sheaf moduli $\FM(\La, T)$ and $\SM_1(\La, T)$.

A cluster ensemble consists of a pair of schemes $\mathcal{A}$ and $\mathcal{X}$ formed by birationally gluing algebraic tori according to certain input data. The spaces $\mathcal{A}$ and $\mathcal{X}$ are dual in the following sense. Consider an integer lattice $N$ with a skew-symmetric bilinear form containing a saturated sublattice  $N_{\uf}$ of $N$ known as the unfrozen sublattice. Denote the dual lattice $\Hom(N, \Z)$ by $M$. The lattice $N$ forms the character lattice of a cluster torus in the $\mathcal{X}$ variety, while the lattice $M$ forms the character lattice of a cluster torus in the $\mathcal{A}$ variety. Cluster tori are glued together by the birational map induced by the mutation formulas given in the previous subsection, with the skew-symmetric bilinear form giving the information of the exchange matrix. The ring of regular functions $\mathcal{O}(\mathcal{A})$ forms an (upper) cluster algebra.

An abbreviated statement of the main theorem of Casals and Weng tells us that for certain families of Legendrians $\La$, the moduli $\SM_1(\La, T)$ and $\FM(\La, T)$ form a cluster ensemble. The class of Legendrians they consider arises from a combinatorial construction known as a (complete) grid plabic graph. This class includes all Legendrian links considered in this work. A more precise summary of the main result of Casals and Weng is as follows:

\begin{theorem}[Theorem 1.1 \cite{CasalsWeng}]
    For $\La$ a Legendrian arising from a complete grid plabic graph, the decorated sheaf moduli $\FM(\La, T)$ admits a cluster $\mathcal{A}$ structure. Moreover, there is an explicitly constructed Legendrian 
weave filling $L$ of $\La$ with intersection quiver and sheaf moduli $\FM(L)$ giving the data of the initial seed.
 \end{theorem}

 Recall that an $\mathbb{L}$-compressing cycle of a filling $L$ is a homology cycle $\gamma\in H_1(L)$ that bounds an embedded Lagrangian disk in the complement of $L$. Casals and Weng obtain mutable cluster $\mathcal{A}$ coordinates by computing microlocal merodromies along relative homology cycles of $L$ dual to $\mathbb{L}$-compressing cycles \cite[Section 4]{CasalsWeng}. They also obtain mutable cluster $\mathcal{X}$ variables of $\SM_1(\La)$ as microlocal monodromies about these $\mathbb{L}$-compressing cycles. These two homology lattices $H_1(L, \mathfrak{t})$ and $H_1(L\backslash \mathfrak{t}, \La \backslash \mathfrak{t})$ yield the dual lattices appearing in the definition of a cluster ensemble. Legendrian mutation of an $\mathbb{L}$-compressing cycle induces a cluster-$\mathcal{X}$ mutation on toric charts in $\SM_1(\La, T)$, while Legendrian mutation of its dual induces a cluster-$\mathcal{A}$ mutation on toric charts in $\FM(\La, T)$. Frozen variables correspond to either marked points or to homology cycles in a Legendrian weave filling $L$ that do not bound embedded Lagrangian disks in the complement $\D^4\backslash L$. 

\begin{remark}
    The cluster structures of \cite{CasalsWeng} are defined up to quasi-cluster equivalence. In order to simplify exposition, we omit any consideration of quasi-cluster equivalences and refer the interested reader to \cite[Appendix A]{CasalsWeng} for further details.
\end{remark}

\subsection{Cluster modular groups}

Given a cluster ensemble, one can consider maps that act on the ensemble by permuting the cluster tori in a way that respects the cluster structure. More precisely, a cluster automorphism of a cluster algebra $\mathbb{A}$ is a permutation of the cluster variables of $\mathbb{A}$ that sends cluster seeds to cluster seeds and commutes with mutation. 
Any cluster automorphism can be defined by the image of a single seed and necessarily preserves the underlying quiver up to a simultaneous change of orientation on all of the arrows \cite[Proposition 2.4]{AssemSchifflerShramchenko}. This leads to the following definition: 

\begin{definition}
    The (orientation preserving) cluster modular group $\SG(\mathbb{A})$ of a cluster algebra $\mathbb{A}$ is the group of maps $\pi$ permuting cluster variables and commuting with mutation such that the induced map on quivers $Q({\bf x})\to Q(\pi({\bf x}))$ is an (orientation preserving) quiver automorphism. 
\end{definition}

\begin{ex}
    For an $A_2$ cluster algebra, the cluster modular group $\SG(A_2)$ is isomorphic to $\Z_5$ and is generated by a $2\pi/5$ rotation of the triangulation corresponding to the initial quiver. Viewing each diagonal $D_{i,j}$ of the triangulation as the image of a Pl\"ucker coordinate $\Delta_{i,j}$ in the top-dimensional positroid strata of the Grassmannian $\Gr(2, 5)$, we can see that this cluster automorphism is given by the map $\Delta_{i, j}\mapsto\Delta_{i-1, j-1}$ for all $1\leq i< j\leq 5$.
\end{ex}

For all classes of cluster algebras discussed in this work, any cluster automorphism $\varphi$ can be given as a finite sequence of mutations $\mu_{v_1}, \mu_{v_2}, \dots, \mu_{v_m}$ and a permutation $\pi\in S_n$ of the quiver vertex labels. In order to avoid confusion, we fix the notation $\mu_{v_1}, \mu_{v_2}, \dots, \mu_{v_m}$ as denoting a sequence of mutations starting with $\mu_{v_1}$ and ending with $\mu_{v_m}$. When we need to specify the particular data of a cluster automorphism $\varphi$, we denote it by the tuple $\varphi=(\mu_{v_1}, \mu_{v_2}, \dots, \mu_{v_m}; \pi)$ with the permutation $\pi$ expressed in cycle notation. Following the conventions of \cite{AssemSchifflerShramchenko, KaufmanGreenberg}, we allow for cluster automorphisms $\varphi$ defined solely as a permutation of the quiver vertex labels without any mutations.

\begin{ex}
    Consider an $A_2$ quiver with vertices labeled 1 and 2 and an edge from 1 to 2. The cluster automorphism $\varphi=(\mu_1; (1\,2))$ also generates $\SG(A_2)$ and corresponds to a rotation of an initial triangulation of the pentagon by $6\pi/5$.
\end{ex}

In their initial work defining cluster automorphisms, \cite{AssemSchifflerShramchenko}, Assem, Schiffler and Shramchenko investigate cluster modular groups of finite, affine, and surface-type cluster algebras. In \cite{KaufmanGreenberg}, the authors give an alternate presentation for affine type cluster modular groups and compute cluster modular groups for extended affine types. In the extended affine case, part of their work builds on Fraser's investigation of cluster modular groups of Grassmannians in \cite{fraser2018braid}. Results relevant to this manuscript are summarized in Tables \ref{tab: generating}, \ref{tab: CMG Gr}, and \ref{tab: folded}.

We introduce the following notion in order to compare cluster automorphisms defined on different initial seeds.

\begin{definition}
    Two cluster automorphisms $\varphi_1$ and $\varphi_2$ are conjugate if they act identically on the set of cluster charts.
\end{definition}

In our proof of Theorem \ref{thm: gens}, we use the combinatorics of tagged triangulations to show that cluster automorphisms induced by Legendrian loops are conjugate to cluster automorphisms coming directly from quiver combinatorics.

 \subsubsection{Combinatorics of (tagged) triangulations}\label{sub: clustercombinatorics}
	
The combinatorics of (tagged) triangulations of punctured surfaces play a key role in defining and understanding many of the simpler classes of cluster algebras. For $A_n, D_n$ and $\tilde{D}_n$ type, many of the computations of cluster modular groups are most accessible through the combinatorics of tagged triangulations. More generally, we can define surface-type cluster algebras as cluster algebras whose underlying quiver comes from a (tagged) triangulation of a punctured surface. In this subsection, we discuss the necessary combinatorial ingredients for understanding tagged triangulations of surfaces in the context of cluster theory.

Let $\Sigma_{n_1, \dots, n_k}$ be a surface with $n_i$ marked points on the $i$th boundary component. We allow for $n_i=0$ and interpret this as a puncture of the surface. Following \cite{FominShapiroThurston}, we define (tagged) arcs, and (tagged) triangulations.

\begin{definition}
An arc $\gamma\in \Sigma_{n_1, \dots, n_k}$ is a curve in $\Sigma$ such that:
\begin{itemize}
    \item the endpoints of $\gamma$ lie on marked points;
    \item the interior of $\gamma$ does not intersect itself;
    \item the interior of $\gamma$ is disjoint from $\partial \S$ and marked points;
    \item $\gamma$ does not cut out an unpunctured monogon or an unpunctured bigon. 
\end{itemize}
\end{definition}

The last condition ensures that no arc is contractible to a point or into the boundary of $\Sigma$. We consider arcs equivalent up to isotopy. Two (isotopy classes of) arcs are said to be compatible if there are two arcs in their respective isotopy classes that do not intersect in the interior of $\Sigma_{n_1, \dots, n_k}$. A triangulation of $\Sigma_{n_1, \dots, n_k}$ is then a maximal pairwise compatible collection of (isotopy classes of) arcs. 

From a triangulation $\mathcal{T}$, we can produce a cluster algebra as follows. 
For every edge $\gamma_i$ in $\mathcal{T}$, we assign a vertex $v_i$ in our quiver $Q_\mathcal{T}$. The vertices corresponding to boundary edges of $\mathcal{T}$ are declared to be frozen. We add an edge from $v_i$ to $v_j$ if there is a face where $\gamma_j$ is counterclockwise from $\gamma_i$ and at least one of $v_i$ or $v_j$ is mutable. Note that we must cancel any oriented two-cycles once we have accounted for all of the edges in this fashion. To each vertex in the quiver, we assign a cluster variable which can be interpreted as measuring the length of the arc in an appropriate hyperbolic geometric context \cite{FominThurston18}. Mutation is given by exchanging one diagonal of a quadrilateral for the other.  

As defined, triangulations of $\Sigma_{n_1, \dots, n_k}$ do not realize every possible cluster seed in the corresponding cluster algebra if the number of punctures is at least one. This is due to the appearance of self-folded triangles, which produce arcs that cannot be mutated at. In order to represent all possible cluster seeds as a triangulation, we require additional decorations. 

We arbitrarily divide an arc into two ends and allow each end to be either tagged or untagged. To produce a tagged triangulation, we introduce additional compatibility relations. 

\begin{definition}[Definition 7.4 \cite{FominShapiroThurston}]
    Two tagged arcs $\gamma_1$ and $\gamma_2$ in $\Sigma_{n_1, \dots, n_k}$ are compatible if the following conditions hold:
    \begin{itemize}
        \item the untagged arcs corresponding to $\gamma_1$ and $\gamma_2$ are compatible;
        \item if the untagged arcs corresponding to $\gamma_1$ and $\gamma_2$ represent distinct isotopy classes and they share an endpoint $a$, then the tagging at the ends of $\gamma_1$ and $\gamma_2$ incident to $a$ coincide;
        \item if the untagged arcs corresponding to $\gamma_1$ and $\gamma_2$ lie in the same isotopy class, then at least one end of $\gamma_1$ must be tagged in the same way as the same end of $\gamma_2$.
    \end{itemize}
\end{definition}

To obtain a quiver from a tagged triangulation, we treat tagged arcs as normal arcs. For an arc $\gamma$ sharing the same endpoints as $\gamma'$, we use the face obtained by deleting $\gamma'$ to compute the direction of the arrows in the quiver. See Figure \ref{fig:TwicePuncturedTaggedTriangulations} (left) for an example. The cluster algebras constructed from (tagged) triangulations of a disk with either 0, 1, or 2 punctures are of type $A_{n}, D_{n},$ and $\tilde{D}_n$, respectively.

\subsubsection{Surface type cluster modular groups}
We highlight here results on cluster modular groups related to the theory of mapping class groups and tagged triangulations, as they will reappear in key technical arguments in Sections \ref{sec: CMGs} and \ref{sec: NielsenThurston}. We define the mapping class group $\Mod(\Sigma_{n_1, \dots, n_k})$ of a surface $\Sigma_{n_1, \dots, n_k}$ with $n_i$ marked points on the $i$th boundary component as the group of orientation-preserving homeomorphisms of $\Sigma$ fixing the set of punctures up to homeomorphisms isotopic to the identity. We interpret the case of $n_i=0$ as an interior puncture. The tagged mapping class group $\Mod_{tag} (\Sigma_{n_1, \dots, n_k})$ is defined to be the semidirect product of $\Mod(\Sigma_{n_1, \dots, n_k})$ with $\Z_2^p$ where $p$ is the number of interior punctures of $\Sigma$. The product structure is specified by the action of simultaneously swapping tags at all arcs incident to a given puncture. 

Denote by $\mathcal{A}(\Sigma_{n_1, \dots, n_k})$ the surface-type cluster algebra associated to $\Sigma_{n_1, \dots, n_k}$. The following theorem relates the tagged mapping class group of a surface and the cluster modular group of the associated cluster algebra.

\begin{theorem}[Proposition 8.5, \cite{BridgelandSmith_QuadDiff}]
Assume that $\Sigma_{n_1, \dots, n_k}$ is not a once or twice-punctured disk with four or fewer marked points on the boundary. Then
 $\Mod_{tag}(\Sigma_{n_1, \dots n_k})\cong \SG(\mathcal{A}(\Sigma_{n_1, \dots, n_k}))$.
\end{theorem}

\subsection{Legendrian loops from plabic fences}\label{sub: LoopsFromFences}

In this subsection, we explain how to geometrically realize Kaufman and Greenberg's presentation of cluster modular groups for affine type. We introduce plabic fences as a combinatorial means of obtaining a sequence of mutations induced by a Legendrian loop. We then describe fronts with initial quivers that are mutation equivalent to $T_{\bf{n}}$ quivers and Legendrian loops that induce automorphisms conjugate to the $\tau_i$ generators of $\Gamma_\tau$.

\subsubsection{Plabic fences} 

Due to the variety of weave equivalence moves, it is sometimes difficult to determine a sequence of mutations that induces the same cluster automorphism as a Legendrian loop. Even when we can compute the sequence of mutations via other combinatorial means, it can be challenging to show that the mutation sequence in the weave agrees with the Legendrian loops. For example, see Appendix \ref{sec: appendix} for a somewhat involved computation of a short mutation sequence induced by a Legendrian loop. Therefore, in order to determine a sequence of mutations corresponding to the induced action of most of the Legendrian loops we consider, we require a combinatorial way to relate quivers to weaves. In the case of $\La(A_n)$, this can be done by understanding the K\'alm\'an loop as rotations of triangulations dual to 2-weaves, as in \cite{Hughes2021b}. For Legendrian links of the form $\La=\La(\beta)$, we rely on the combinatorics of plabic graphs. 

Plabic graphs are a general combinatorial object related to cluster theory and were first studied by Postnikov \cite{postnikov2006total}. The term `plabic fence' that we use for the family graphs we work with was coined in \cite{FominPylyavskyyShustinThurston} and refers to a particular form of plabic graph, as we describe below. We follow the recipe of \cite{CasalsWeng} for producing a Legendrian link from a plabic fence and use the graph to encode the combinatorics of Legendrian loop mutations. 

\begin{definition}
    A plabic fence is a planar graph with univalent or trivalent vertices colored either black or white constructed as follows:
    \begin{enumerate}
        \item Stack $n$ horizontal lines of the same length on top of each other, each starting with a white vertex on the left and ending with a black vertex on the right.
        \item Add vertical edges between adjacent pairs of horizontal lines with trivalent vertices where they meet colored so that each endpoint of the vertical edge is a different color.
    \end{enumerate}

\end{definition}

From the plabic fence $\mathbb{G}$, we can extract a quiver $Q_\mathbb{G}$ via the following process.
    First, assign a mutable vertex to each  face of $\mathbb{G}$. Then, for every edge $e\in \mathbb{G}$ connecting two faces, add an arrow between the corresponding vertices oriented so that the white endpoint of $e$ is to the right of the edge when traveling in the direction of orientation. 

To a plabic fence $\mathbb{G}$, we associate a 
Legendrian link $\La_\mathbb{G}$ and an initial Legendrian weave surface with boundary $\La_\mathbb{G}$, following the recipe from \cite[Sections 2 and 3]{CasalsWeng}. To obtain a positive braid $\beta$ from $\mathbb{G}$ such that $\La_\mathbb{G}\cong\La(\beta)$, first label the horizontal lines of $\mathbb{G}$ from bottom to top with the numbers 1 through $n$ and label the vertical edges between horizontal lines $i$ and $i+1$ by $\sigma_i$. We refer to a vertical edge by the color of its top vertex, so that a vertical edge between lines $i$ and $i+1$ is white if the vertex on line $i$ is colored white. In the front projection, one should think of the black vertical edges as denoting crossings appearing above cusps, while white vertical edges correspond to crossings below cusps. More precisely, from $\mathbb{G},$ we obtain braid words $\beta_1$ from white vertical edges and $\beta_2$ from black vertical edges by scanning $\mathbb{G}$ from left to right and appending a $\sigma_i$ (resp. $n-i$) to $\beta_1$ (resp. $\beta_2$) for every white (resp. black) vertical edge between lines $i$ and $i+1$. The Legendrian link is Legendrian isotopic to the rainbow closure of the braid $\beta_1\beta_2^\circ$ (equivalently $\beta_2^\circ \beta_1$) where $\beta_2^\circ$ is obtained from $\beta_2$ by reading the braid word from right to left and replacing $\sigma_i$ with $\sigma_{n-i}$ for all $i.$

 \begin{center}
		\begin{figure}[h!]{ \includegraphics[width=.3\textwidth]{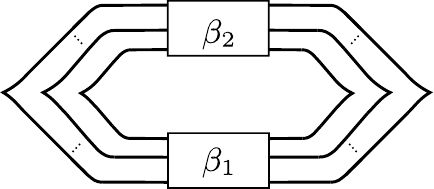}}\caption{Legendrian link obtained from a plabic fence. The braid word $\beta_1$ corresponds to white vertical edges, while the braid word $\beta_2$ corresponds to black.}
			\label{fig: PlabicFenceLink}\end{figure}
	\end{center}

From $\La_{\mathbb{G}}$, Casals and Weng construct an initial weave $\w(\mathbb{G})$ \cite[Definition 3.24]{CasalsWeng}. By \cite[Theorem 1.2]{CasalsLi}, the combinatorial data of the seed corresponding to the plabic graph $\mathbb{G}$ agrees with this choice of initial weave $\w(\mathbb{G})$. More precisely, the Legendrian $\La_{\mathbb{G}}$ given here is Legendrian isotopic to the Legendrian obtained as the conormal lift of zig-zag strands and the conjugate Lagrangian surface associated to an initial seed is Hamiltonian isotopic to the Lagrangian projection of the initial weave. As a result, the intersection quiver of $\w(\mathbb{G})$ agrees with the initial quiver $Q_\mathbb{G}$ coming from the plabic fence.

\subsubsection{Mutations induced by elementary Legendrian isotopies}
In order to compute the sequence of mutations induced by Legendrian loops we decompose our Legendrian loops into a series of simple Legendrian isotopies and describe how to combinatorially realize them in the plabic fence. The first Legendrian isotopy we consider is a Reidemeister III move, which interchanges crossings $\sigma_i\sigma_{i+1}\sigma_i$ with $\sigma_{i+1}\sigma_{i}\sigma_{i+1}$. The plabic fences $\mathbb{G}_\beta$ and $\mathbb{G}_{\beta'}$ corresponding to the two braid words $\beta$ and $\beta'$ then differ by the local move pictured in Figure \ref{fig:plabic braid move}. Combinatorially, we can see that $Q_{\mathbb{G}_\beta}$ differs from $Q_{\mathbb{G}_{\beta'}}$ by a mutation at the vertex corresponding to the face. This elementary Legendrian isotopy also induces a map between $\FM(\La(\mathbb{G}_\beta), T)$ and $\FM(\La(\mathbb{G}_{\beta'}), T)$ that is an isomorphism, but generally not an automorphism. The following lemma states that this induced isomorphism agrees with cluster mutation.

\begin{figure}
    \centering
    \includegraphics{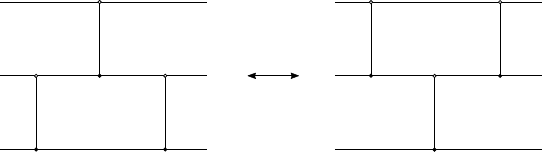}
    \caption{Local move of a plabic fence corresponding to the braid move exchanging $\sigma_i\sigma_{i+1}\sigma_i$ and $\sigma_{i+1}\sigma_i\sigma_{i+1}$.}
    \label{fig:plabic braid move}
\end{figure}

\begin{lemma}[Proposition 7.16, \cite{CLSBW23}]\label{lemma: plabic mutation}
    Let $\mathbb{G}_\beta$ and $\mathbb{G}_{\beta'}$ be plabic fences related by the local move pictured in Figure \ref{fig:plabic braid move} and denote the quiver vertex corresponding to the unique face of $\mathbb{G}$ by $k$. The initial seeds of $\FM(\La(\mathbb{G}_\beta), T)$ and $\FM(\La(\mathbb{G}_{\beta'}), T)$  
    are related by mutation at the cluster variable $a_k$.
\end{lemma}

In addition to the Reidemeister III move, we also consider a sequence of isotopies given by the cyclic shift, which modifies the braid word by conjugation. In a plabic fence with all white vertical edges, we can interpret this isotopy as flipping the leftmost white edge to black, moving it past all of the white vertical edges in the row to its right, and then flipping it back to white. Each time we move a black edge past a white one, we perform the local move pictured in Figure \ref{fig:plabic square move}, known as a square move. As with the Reidemeister III move, this local move induces an isomorphism between sheaf moduli that corresponds to a cluster mutation between the initial seeds.

\begin{figure}
    \centering
    \includegraphics{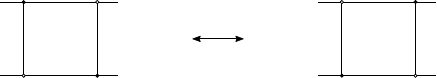}
    \caption{A square move in a plabic fence, corresponding to mutation at the vertex represented by the face.}
    \label{fig:plabic square move}
\end{figure}

\begin{lemma}\cite[Section 5.3]{CasalsWeng}\label{lemma: square move}
    Let $\mathbb{G}_\beta$ and $\mathbb{G}_{\beta'}$ be plabic fences related by the local move pictured in Figure \ref{fig:plabic braid move} and denote the quiver vertex corresponding to the unique face of $\mathbb{G}$ by $k$. The initial seeds of $\FM(\La(\mathbb{G}_\beta), T)$ and $\FM(\La(\mathbb{G}_{\beta'}), T)$ are related by mutation at the cluster variable $a_k$.
\end{lemma}

\subsubsection{Legendrian loop mutation sequences}\label{sub: Mutation sequences}

As noted above, an elementary Legendrian isotopy in general only induces an isomorphism of $\FM(\La, T)$. On the other hand, a Legendrian loop necessarily induces an automorphism of $\FM(\La, T)$. The automorphism induced by a Legendrian $\varphi$ preserves intersection quivers and extends to an action on cluster charts by considering the concatenation of the trace of $\varphi$ with fillings of $\La$. As a result, the automorphism induced by a Legendrian loop is a \emph{cluster} automorphism. Following the argument in the previous subsection, the sequence of mutations giving the specific cluster automorphism induced by $\varphi$ can be computed by decomposing $\varphi$ into elementary Legendrian isotopies.

For Legendrian satellite loops, we can obtain an explicit description of the induced sequence of mutations. Consider $\La=\La(\beta\Delta_n^2, i, \sigma_1^k)$ with $k\geq 3$. Here $\beta\in \Br_n^+$, $\Delta_n^2$ is a full twist on $n$ strands, and $\sigma_1^k\in \Br_2^+$. As detailed in Section \ref{sub: loops defs}, we obtain a Legendrian $\vartheta$-loop by rotating a single crossing of the satellited braid $\sigma_1^k$ following the pair of satellited strands in the contact neighborhood of the $i$th strand of the $(-1)$-framed closure of $\beta$. Note that since $k\geq 3$, the link $\La(\beta\Delta_n^2, i, \sigma_1^k)\cong \La(\beta')$ for some braid $\beta'\in \Br_{n+1}^+$. Therefore, there is some grid plabic graph $\mathbb{G}(\beta')$ such that $\La_{\mathbb{G}(\beta')}\cong \La(\beta\Delta_n^2, i, \sigma_1^k)$.

By construction, every smallest subword of the form $\sigma_j\dots \sigma_j$ of $\La_\mathbb{G}$ yields a face of $\mathbb{G}$. For $\La(\beta')\cong\La(\beta\Delta_n^2, i, \sigma_1^k)$, certain faces of the plabic graph $\mathbb{G}(\beta')$ correspond to the satellited braid $\gamma=\sigma_1^k$. These faces can be identified from $\beta$ and $i$ by describing the relationship between $\beta$ and $\beta'$: 
 Let $\La$ denote the $(-1)$-framed closure of $\beta\Delta_n^2$. For any crossings $\sigma_j$ appearing above the $i$th strand $\La_i$ of $\La$, we replace $\sigma_j$ by $\sigma_{j+1}$ in $\beta'$, as the appearance of the additional satellited strand shifts the strands above it up by one. Similarly, any crossings appearing below $\La_i$ are fixed. If $\La_i$ appears at height $j$ at some point in the subword, then for any crossing $\sigma_{j}$ involving $\La_i$, we replace $\sigma_j$ by $\sigma_{j}\sigma_{j+1}$. Under these conditions, we also replace $\sigma_{j-1}$ by $\sigma_{j-1}\sigma_{j}$.  
 
Applying the above observation to the grid plabic graph $\mathbb{G}$ allows us to understand the faces of $\mathbb{G}$ corresponding to $\gamma$ by analyzing the braid word. In particular, for any two consecutive crossings of $\beta$ involving the $i$th strand the corresponding subword of the satellite braid $\beta'$ is given as follows:
\begin{itemize}
    \item $\sigma_j\sigma_{j+1}\mapsto \sigma_{j+1}\sigma_j\sigma_{j+2}\sigma_{j+1}$
    \item $\sigma_{j-1}\sigma_{j}\mapsto \sigma_{j-1}\sigma_j\sigma_{j-2}\sigma_{j-1}$
    \item $\sigma_j\sigma_{j}\mapsto \sigma_{j+1}\sigma_j\sigma_{j}\sigma_{j+1}$
    \item $\sigma_{j-1}\sigma_{j-1}\mapsto \sigma_{j}\sigma_{j+1}\sigma_{j+1}\sigma_{j}$
\end{itemize}
As a result, we see that for any pair of crossings of $\beta$ involving $\La_i$, we have a distinguished face of the plabic fence $\mathbb{G}(\beta')$ that can be identified with crossings of the satellited braid $\La(\sigma_1^k)_i\subseteq \La(\beta\Delta_n^2, i, \sigma_1^k)$.

Denote the subset of the quiver corresponding to these faces by $Q_\gamma$ and label the vertices from left to right by $\{v_{\gamma_1}, \dots, v_{\gamma_p}\}$.  Denote also by $Q_i$ the subset of the quiver corresponding to the faces in the $i$th row of $\mathbb{G}$ corresponding to pairs of crossings $\sigma_i \dots \sigma_i$ with vertices $\{v_{i_1}, \dots, v_{i_q} \}$. Note that $Q_\gamma$ and $Q_i$ necessarily have nonempty intersection so that some of the vertex labeling is redundant. We denote the vertex labels in $Q_\gamma \cap Q_i$ by $\{i_{r_j}=\gamma_{r_j}\}_{j=1}^m$ From the above description we can extract a combinatorial recipe for the action of the Legendrian $\vartheta$-loop.

\begin{lemma}\label{lemma: satellite loop mutations}
 Let $\La=\La(\beta\Delta_n^2, i, \sigma_1^k)$ for $k\geq 3$. The $\vartheta$-loop described above induces the sequence of mutations  $$\gamma_p, \gamma_{p-1}, \dots, \gamma_1, \gamma_1, \pi_{\gamma}^{-1}(i_1), \pi_{\gamma}^{-1}(i_2), \dots, \pi_{\gamma}^{-1}(i_q)$$
 and the permutation of quiver vertex labels is given in cycle notation by $$(i_1\, \dots \, i_{r_{1}} \, \gamma_{r_{1}-1} \, \dots \, \gamma_{1})\dots (i_{r_{m-1}+1}\, \dots \, i_{r_{m}} \, \gamma_{r_{m}-1} \, \dots \, \gamma_{r_{m-1}+1}).$$ 
\end{lemma}

\begin{proof}
For $\La(\beta\Delta_n^2, i, \sigma_1^k)$, the $\vartheta$-loop is given by a sequence of Reidemeister III moves and a cyclic rotation that pass a single crossing $c$ of $\La(\sigma_1^k)_i$ around the braid. The first Reidemeister III move passes the crossing $c$ through the rightmost strand passing over or under $\La(\sigma_1^k)_i$, forming a face of $\mathbb{G}$. Each successive Reidemeister III move passes the crossing $c$ across a subsequent strand of $\La(\beta\Delta_n^2, i, \sigma_1^k)$. By Lemma \ref{lemma: plabic mutation}, this Reidemeister III move corresponds to a mutation of the quiver $Q_{\mathbb{G}}$ identified with the particular face of $\mathbb{G}$ formed by the strand and the crossing $c$, and yields a new plabic graph $\mathbb{G}'$. By construction we can readily identify the vertices belonging to $Q_\gamma$ as a subset of $Q_{\mathbb{G}'}$. 

After performing all of the Reidemeister III moves, the mutation sequence up to this point is precisely $\gamma_p, \dots, \gamma_1$. The resulting plabic fence is identical to $\mathbb{G}$ except that the vertical white edge representing the crossing $c$ is now the leftmost edge of the $i$th row of the plabic fence. Following the sequence of mutations, we see that the corresponding face is labeled by the vertex $v_{\gamma_1}$. To complete the $\vartheta$-loop, we perform a cyclic rotation by flipping the leftmost white edge corresponding to $c$ to a black edge and performing a sequence of square moves that push it as far to the right as possible. By Lemma \ref{lemma: square move}, these correspond to mutations of each vertex in row $i$, starting with the vertex $v_{\gamma_1}$ of $Q_\gamma$ and then proceeding in order through the vertices of in row $i$. For any vertex $r_j$ in $Q_i\cap Q_\gamma$, the vertex in row $i$ after the mutations $\gamma_p, \dots, \gamma_1$ is $\pi_\gamma^{-1}(v_{r_j})$ while any other vertex in row $i$ is fixed by $\gamma_p, \dots, \gamma_1$ The resulting mutation sequence is therefore given by $\gamma_p, \dots, \gamma_1, \gamma_1, \pi_\gamma^{-1}(i_1), \dots, \pi_\gamma^{-1}(i_q)$, as desired.

The permutation of quiver vertex labels resulting from this sequence of mutations follows from the plabic fence computation. In particular, each Reidemeister III move inducing a mutation in the sequence $\gamma_p, \dots, \gamma_1$ leads to replacing the quiver vertex label $v_{\gamma_{j+1}}$ by the quiver vertex label $v_{\gamma_j}$. The sequence of mutations $\pi_\gamma^{-1}(i_1),\dots, \pi_\gamma^{-1}(i_q)$ induces a similar relabeling, with the additional caveat that we need to account for the nonempty intersection between $Q_i$ and $Q_\gamma$. As a result, each cycle in our permutation corresponds to a single vertex $v_{r_j}$ in $Q_i\cap Q_\gamma$ and is given by shifting the vertices in $Q_i$ between $v_{r_j}$ and $v_{r_{j-}}$ to the right by one and then shifting the vertices in $Q_\gamma$ between $v_{r_j}$ and $v_{r_{j-1}}$ to the left by one. Note that we include $v_{r_j}$ but not $v_{r_{j-1}}$ in this particular cycle of the permutation. 
\end{proof}

\begin{ex}
Consider the positive braid $\beta=\sigma_1^2\sigma_3\sigma_2\sigma_3^2\sigma_2\sigma_1^2\Delta_4^2\in \Br_4^+$. The Legendrian link $\La(\beta, 2, \sigma_1^4)$ is isotopic to $\La(\sigma_2\sigma_1^2\sigma_2\sigma_4\sigma_3\sigma_2\sigma_4\sigma_3^2\sigma_4\sigma_2\sigma_3\sigma_1\sigma_2^2\sigma_1\sigma_2^2)$ with a corresponding plabic fence $\mathbb{G}_0$, pictured in Figure \ref{fig: PlabicFenceExample}. If we label the faces from left to right as shown, we have that $Q_i=\{v_2, v_3, v_6, v_{7}, v_{9}, v_{10}, v_{11}\}$, while $Q_\gamma=\{v_1, v_{3}, v_5, v_7, v_{8}, v_{10}\}$. The mutation sequence induced by the $\vartheta$ loop is then given by $$\tilde{\vartheta}=(10, 8, 7, 5, 3, 1, 1, 2, 5, 6, 8, 9, 11; (2\, 3\, 1)(6\, 7\, 5)(9\, 10\, 8)$$
\end{ex}

\begin{figure}
    \centering
\includegraphics[width=.7\textwidth]{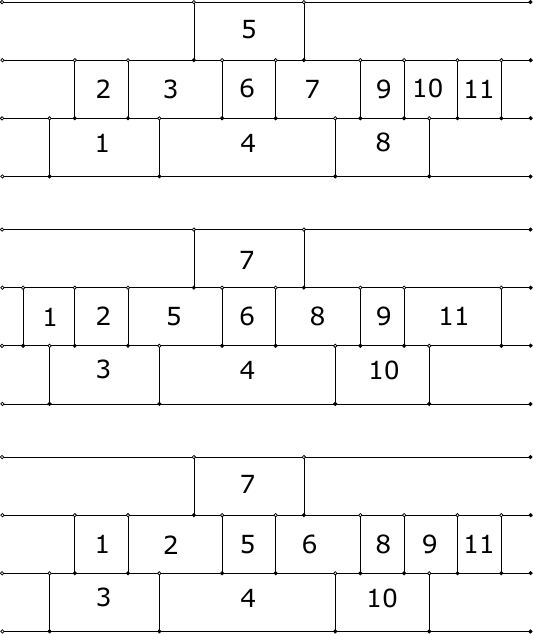}
    \caption{A sequence of plabic fences depicting the mutation sequence of the Legendrian $\vartheta$-loop. Starting with the initial plabic (top), we obtain the result of performing the Reidemeister III moves inducing the mutation sequence $\gamma_p, \dots, \gamma_1$ (middle). Applying a cyclic rotation results in the plabic fence depicted at the bottom. }
    \label{fig: PlabicFenceExample}
\end{figure}

Finally, we describe the sequence of mutations associated to the $\DT$ transformation in the case of Legendrian links $\La(\beta)$. This description comes from giving a recipe for realizing $\DT$ in terms of the combinatorics of plabic fences, following \cite[Section 5]{CasalsWeng}. In the Legendrian, this procedure is roughly described as rotating all of the crossings past the cusps and then composing with the strict contactomorphism $x\mapsto -x, z\mapsto -z$; see Subsection \ref{sub: loops defs}. Starting with a plabic fence of all white edges, the Legendrian isotopy can be combinatorially realized by repeatedly flipping the rightmost white edge to black and moving it past the remaining white edges in its row using square moves until it becomes the rightmost black edge. Once all white edges are flipped to black, the contactomorphism then takes all black edges to white edges, returning to the initial plabic fence. Recording the square moves involved in this process gives an explicit mutation sequence.

\section{Legendrian loops as generators of cluster modular groups}\label{sec: CMGs}

In this section, we describe presentations for cluster modular groups of of skew symmetric cluster algebras of finite, affine, and extended affine Dynkin type and match them with Legendrian loop actions. We start by describing the quiver combinatorics used in \cite{KaufmanGreenberg} to give presentations for cluster modular groups of affine and extended affine types. We then compare the induced action of our Legendrian loops with the presentations of cluster modular groups given below. Following our discussion of cluster modular groups for Dynkin types, we also include a similar computation for cluster automorphisms arising from Legendrian loops of torus links.

\subsection{Presentations of cluster modular groups via $T_{\bf n}$ quivers}\label{sub: TnQuivers}

In \cite{KaufmanGreenberg}, Kaufman and Greenberg use a particular family of quivers to explicitly describe the cluster modular groups of affine and extended affine cluster algebras.

We make use of their work by showing that the initial quivers from plabic fences for affine and extended affine type 
yield nearly identical combinatorial presentations. We begin by defining Kaufman and Greenberg's quivers for simply-laced type.
\begin{definition}\label{def: TnQuiver}
    Given a vector ${\bf n}=(n_1, \dots, n_k)$ of natural numbers $n_i\geq 2$, a $T_{\bf n}$ quiver is a quiver with a pair of special vertices $v_1$ and $v_0$ and a collection of $k$ `tails' of vertices of length $n_1, \dots, n_k$, as pictured in Figure \ref{fig:T_nQuiver}.\footnote{The quiver pictured in Figure \ref{fig:T_nQuiver} is actually opposite to the one considered in \cite{KaufmanGreenberg}. The pictured orientation is chosen to match our previous conventions.} 
\end{definition}

\begin{figure}
    \centering
    \includegraphics{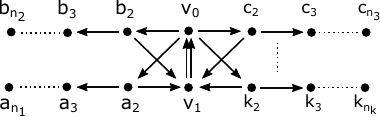}
    \caption{A $T_{n_1, \dots n_k}$ quiver.  Deleting the vertex $v_0$ results in a quiver with central vertex $v_1$ and tails of length $n_i-1$ (not including the vertex $v_1$) where the vertices alternate as either sources or sinks for the two incident edges.}
    \label{fig:T_nQuiver}
\end{figure}

In this work, we will always consider the case of $k=3$. The particular cluster types we consider are listed in  Table \ref{tab: generating} along with the corresponding $T_{\bf n}$ quiver. $T_{\bf{n}}$ quivers admit a particular class of quiver automorphisms, which we will denote $\tau_1, \dots, \tau_k$. Recall that we denote a cluster automorphism by a tuple with the first entry a sequence of quiver mutations and the second entry a permutation describing the relabeling of the quiver vertices. For a tail of length $n_i$ in $T_{\bf n}$, we denote $i_{\text{odd}}=\{i_j | 3\leq j \leq n_i, j \text{ odd}\}$ and $i_{\text{even}}\{i_j | 3\leq j \leq n_i, j \text{ even}\}.$ The automorphism $\tau_i$ is then given by 
$$\tau_i=(\mu_{i_{\text{odd}}}\mu_{i_{\text{even}}} \mu_{i_{2}}\mu_{v_0}\mu_{v_1}; (i_2 \,i_0\, i_1))$$
where $\mu_{i_{\text{odd}}}$ denotes a sequence of mutations starting at $i_j$ for $j<n_i$ the largest number in $i_{\text{odd}}$ and ending at the vertex $i_3$. The mutation sequence $\mu_{i_{\text{even}}}$ is defined analogously. Denote by $\Gamma_\tau$ the subgroup of $\SG$ generated by $\tau_1, \dots, \tau_k$.

\begin{theorem}[Theorem 4.11, \cite{KaufmanGreenberg}]\label{thm: TauRelation}
    $\Gamma_\tau$ is an abelian group with relations $\tau_i^{n_i}=\tau_j^{n_j}$.
\end{theorem}

Using the relation $\tau_i^{n_i}=\tau_j^{n_j}$ from Theorem \ref{thm: TauRelation}, we can define the element $\gamma:=\tau_i^{n_i}$ of $\Gamma_\tau$. Kaufman and Greenberg show that the Donaldson-Thomas transformation can then be given in terms of $\gamma$ and the $\tau_i$ generators by $\DT=\gamma^2\prod_{i=1}^k \tau_i\gamma^{-1}$ \cite[Theorem 4.14]{KaufmanGreenberg}. For the skew-symmetric cluster algebras we investigate here, this simplifies to $\DT=\tau_1\tau_2\tau_3\gamma^{-1}$.

In addition to the $\tau_i$ generators, 
we must also consider graph automorphisms $\sigma\in \Aut(T_{\bf{n}})$ acting on $\Gamma_\tau$ by swapping tails $i$ and $j$ of length $n_i=n_j$. The utility of the $T_{\bf{n}}$ quiver construction is demonstrated by the following theorem.

\begin{theorem}[\cite{KaufmanGreenberg}, Theorems 5.2 and 6.1]
    For a cluster algebra $\mathbb{A}$ of affine type, $\SG(\mathbb{A})\cong \Gamma_\tau \rtimes \Aut(T_{\bf{n}})$. For a cluster algebra $\mathbb{A}$ of extended affine type, $\Gamma_\tau \rtimes \Aut(T_{\bf{n}})$ is isomorphic to a finite index subgroup of $\SG(\mathbb{A})$.
\end{theorem} 
Kaufman and Greenberg also conjecture that $\SG(\mathbb{A})\cong \Gamma_\tau \rtimes \Aut(T_{\bf{n}})$ for cluster algebras $\mathbb{A}$ that admit a $T_{\bf{n}}$ quiver and are not of extended affine type \cite[Conjecture 4.7]{KaufmanGreenberg}.

\subsection{Proof of Theorem \ref{thm: gens}}\label{sub: proof of gens}
 
Let us now proceed to a proof of Theorem \ref{thm: gens}. We start by using Lemma \ref{lemma: satellite loop mutations} to compare the induced action of Legendrian loops to the generating sets of Kaufman and Greenberg described in the previous subsection. We then continue with a case-by-case analysis of each Legendrian link corresponding to a skew-symmetric cluster algebra of finite and affine type.

\subsubsection{Legendrian loops conjugate to $\tau_i$}\label{sub: conjugate}

Let us define a family of positive braids $\beta_{k, n_1, n_2}$ by $$\beta_{k, n_1, n_2}:=(\sigma_2\sigma_1\sigma_3\sigma_2)^{2k}\sigma_1^{n_1-2}\sigma_3^{n_2-2}$$ for $k, n_1, n_2\in \N$ and $n_1, n_2 \geq 2.$ For Legendrian links corresponding to affine and extended affine type cluster algebras, we will require $k=1$ and $n_1$ and $n_2$ agreeing with the corresponding values for $T_{\bf{n}}$ quivers. See Table \ref{tab: generating} below. 
Denote by $\w(k, n_1, n_2)$ the initial weave filling of $\La(\beta_{k, n_1, n_2})$.  

\begin{lemma}
The intersection quiver $Q_{\w(1, n_1, n_2)}$ is mutation equivalent to $T_{n_1, n_2, 2}.$
\end{lemma}

\begin{figure}
    \centering
    \includegraphics{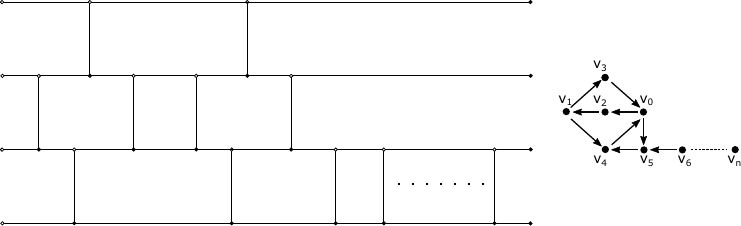}
    \caption{Plabic fence $\mathbb{G}$ corresponding to $\La(\tilde{D}_n)$. The dots correspond to $n-6$ additional vertical edges}
    \label{fig:AffineDnPlabicFence}
\end{figure}

\begin{proof}
    A plabic fence corresponding to $\La(\beta_{1, n_1, n_2})$ can be obtained from the plabic fence in Figure \ref{fig:AffineDnPlabicFence} by the addition of $n_2-2$ vertical edges between the top two horizontal lines. We then get an initial quiver by adding $n_2-2$ vertices to the quiver pictured in Figure \ref{fig:AffineDnPlabicFence} in a manner identical to the tail labeled by vertices $v_6, \dots, v_n$. Mutating at vertices $v_3$ and $v_4$ of $Q_{\w(1, n_1, n_2)}$ yields a $T_{n_1, n_2, 2}$ quiver up to the alternating behavior of the arrows in the tails. This alternation can be obtained by observing that the tails are $A_n-$type subquivers, which are mutation equivalent to any quiver given by a choice of orientation on the underlying Dynkin diagram.  
\end{proof}

To present Legendrian loops of $\La(\beta_{k, n_1, n_2})$, we will consider the Legendrian isotopic link given as the $(-1)$-framed closure of the braid $(\sigma_2\sigma_1\sigma_3\sigma_2)^{2k}\sigma_1^{n_1-2}\sigma_3^{n_2-2}\sigma_1^2\sigma_3^2(\sigma_2\sigma_1\sigma_3\sigma_2)^2$. Note that in either front, the Legendrian link $\La(\beta_{k, n_1, n_2})$ is Legendrian isotopic to the link given by simultaneously satelliting $\gamma_1= \sigma_1^{n_1}$ and $\gamma_2=\sigma_1^{n_2}$ about the two strands of the Legendrian Hopf link given as the $(-1)$-framed closure of the braid $\sigma_1^4$. With this description, we immediately obtain two $\vartheta$-loops, depicted in Figure \ref{fig:LoopEx}. Note that when $n_i=2$, $\vartheta_i$ can still be identified because of our choice of braid word for $\Delta^2$ even if no such loop is apparent in the Legendrian isotopic rainbow closure  $\La((\sigma_2\sigma_1\sigma_3\sigma_2)^2\sigma_1^{n_1-2}\sigma_3^{n_2-2})$.

\begin{figure}
    \centering
    \includegraphics[width=.6\textwidth]{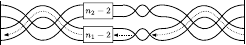}
    \caption{The Legendrian loop $\vartheta_1$ for the braid $\beta_{1, n_1, n_2}$ pictured in $J^1\S^1$. The loop takes one of the $n_1-2$ crossings appearing in the box and drags it around the front, following the path of the dotted arrows.}
    \label{fig:LoopEx}
\end{figure}

Recall that, as stated in Theorem \ref{thm: TauRelation}, the group $\Gamma_\tau$ is Abelian. The following result, along with the relationship between $\vartheta$ loops and the $\tau_i$ generators discussed below, geometrically realizes this property.

\begin{proposition}\label{prop: LoopsCommute}
Let $\La$ be a Legendrian link that is isotopic to a Legendrian link obtained from simultaneously satelliting patterns $\gamma_1$ and $\gamma_2$ about two distinct strands of some companion Legendrian $\La'$. For any exact Lagrangian filling $L$ of $\La$, we have $\vartheta_1\circ \vartheta_2(L)\cong \vartheta_2\circ\vartheta_1(L)$.  
\end{proposition}

\begin{proof}
    Consider the concatenation of $\Trace(\vartheta_2)$ with $ \Trace(\vartheta_1)$. We can produce the desired Hamiltonian isotopy between concordances by producing showing that the composition of Legendrian loops $\vartheta_2\circ\vartheta_1$ is homotopic to $\vartheta_1\circ \vartheta_2$. The first composition can be described as first performing $\vartheta_1$ during time $t_1$ and then performing $\vartheta_2$ during time $t_2$ for $0\leq t_1<\frac{1}{2}$ and $\frac{1}{2}\leq t_2\leq 1$. The required homotopy is then defined by gradually increasing $t_1$ and gradually decreasing $t_2$. 
 
\end{proof}

We now show that our description of Legendrian loops coincides with $\Gamma_\tau$.

	\begin{center}\begin{figure}[h!]{ \includegraphics[width=.8\textwidth]{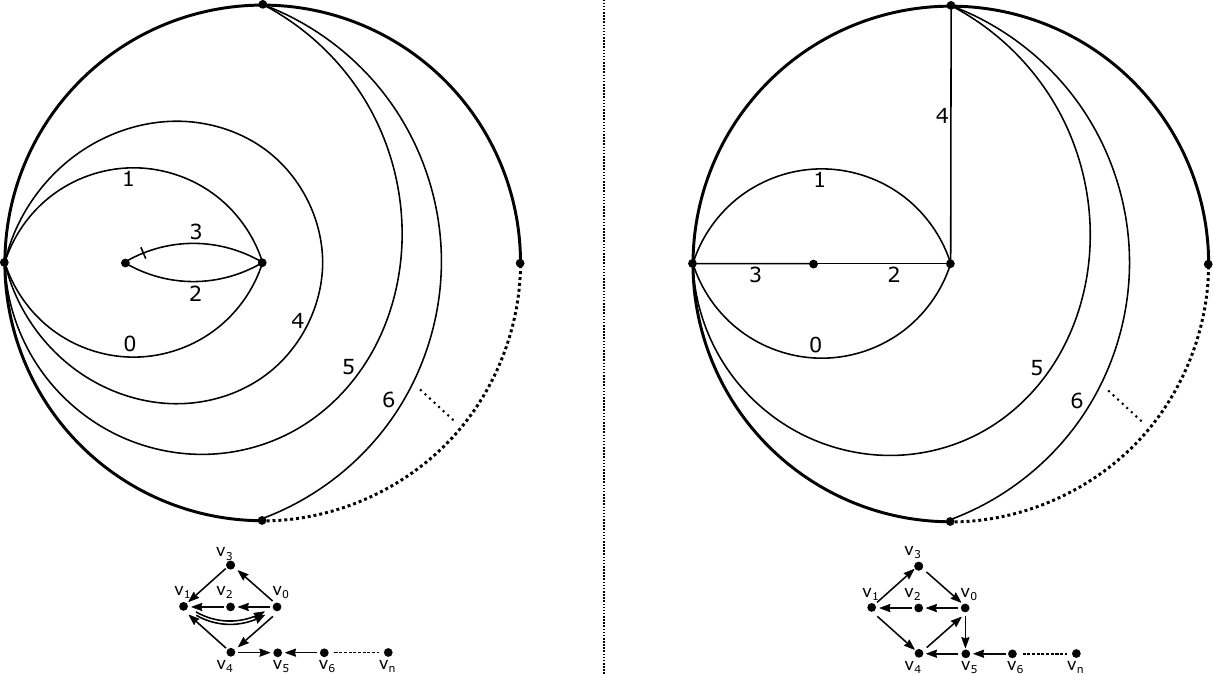}}\caption{A pair of tagged triangulations of $\D^2_{n-2, 0, 0}$. The triangulation on the left corresponds to the quiver $T_{n-2, 2, 2}$, while the triangulation on the right corresponds to the quiver $\mathbb{G}(\La(1, n_1, n_2))$ and is obtained from the triangulation on the left by mutation at edges labeled 3 and 4. The dotted lines on the left represent a zig-zag pattern of $n-6$ edges, while the dotted lines on the right represent $n-6$ edges all sharing the top vertex. 
			}
			\label{fig:TwicePuncturedTaggedTriangulations}\end{figure}
	\end{center}

\begin{lemma}\label{lemma: loops conjugate}
    For any Legendrian $\La(\beta_{1, n_1, n_2})$, 
    the cluster automorphisms $\tilde{\vartheta}_1$ and $\tilde{\vartheta}_2$ induced by the corresponding Legendrian loops are conjugate to $\tau_1$ and $\tau_2$.  
\end{lemma}

\begin{proof}
    We first establish the statement in the case of $\La(\tilde{D}_n)\cong \La(\beta_{1, n-2, 2})$ using the combinatorics of tagged triangulations of a twice-punctured disk. We then leverage these combinatorics for the case of $n_2>2$. 

Consider the plabic fence $\mathbb{G}(\tilde{D}_n)$ depicted in Figure \ref{fig:AffineDnPlabicFence}, corresponding to the Legendrian link $\La(\tilde{D}_n)\cong \La_{1, n-2, 2}$. The sequence of mutations corresponding to $\vartheta_1$ can be determined from $\mathbb{G}(\tilde{D}_n)$ using Lemma \ref{lemma: satellite loop mutations}. In particular, $\vartheta_1$ induces the mutation sequence  \begin{equation}
    \tilde{\vartheta}_1=(5, 0, 3, 1, 1, 4, 6, \dots, n; (1\,4\,5\,0\,3))
\end{equation} while the loop $\vartheta_2$ induces the mutation sequence $\tilde{\vartheta}_2=(1, 4; (1\,4\,0\,3))$. This second mutation sequence is verified by explicit computation in the initial weave filling $\w(\mathbb{G}(\tilde{D}_n))$ in Appendix \ref{sec: appendix}. Note that the use of Legendrian weaves for this computation appears to be necessary, as the crossings involved in the Legendrian loop $\vartheta_2$ do not appear in the rainbow closure $\La((\sigma_2\sigma_1\sigma_3\sigma_2)^2\sigma_1^{n-3})$ and therefore are not straightforwardly captured by the combinatorics of the plabic fence.

Figure \ref{fig:TwicePuncturedTaggedTriangulations} depicts two tagged triangulations of a twice-punctured disk. The triangulation $\mathcal{T}_1$ (left) corresponds to a quiver identical to $Q_{\mathbb{G}(\tilde{D}_n)}$, while the triangulation $\mathcal{T}_2$ (right) corresponds to a quiver identical to $T_{n-2, 2,2}$. The generator $\tau_1$ corresponds to a rotation of the boundary of the disk by $2\pi/n$ \cite[Lemma B.3]{KaufmanGreenberg}. Therefore, to show that $\tau_1$ and $\tilde{\vartheta}_1$ are conjugate, we need to show that $\tilde{\vartheta}_1$ also corresponds to a rotation of the boundary of the disk by $2\pi/n$. This is done explicitly in Figure \ref{fig:Theta1Triangulations}.

\begin{figure}
    \centering
    \includegraphics[width=\textwidth]{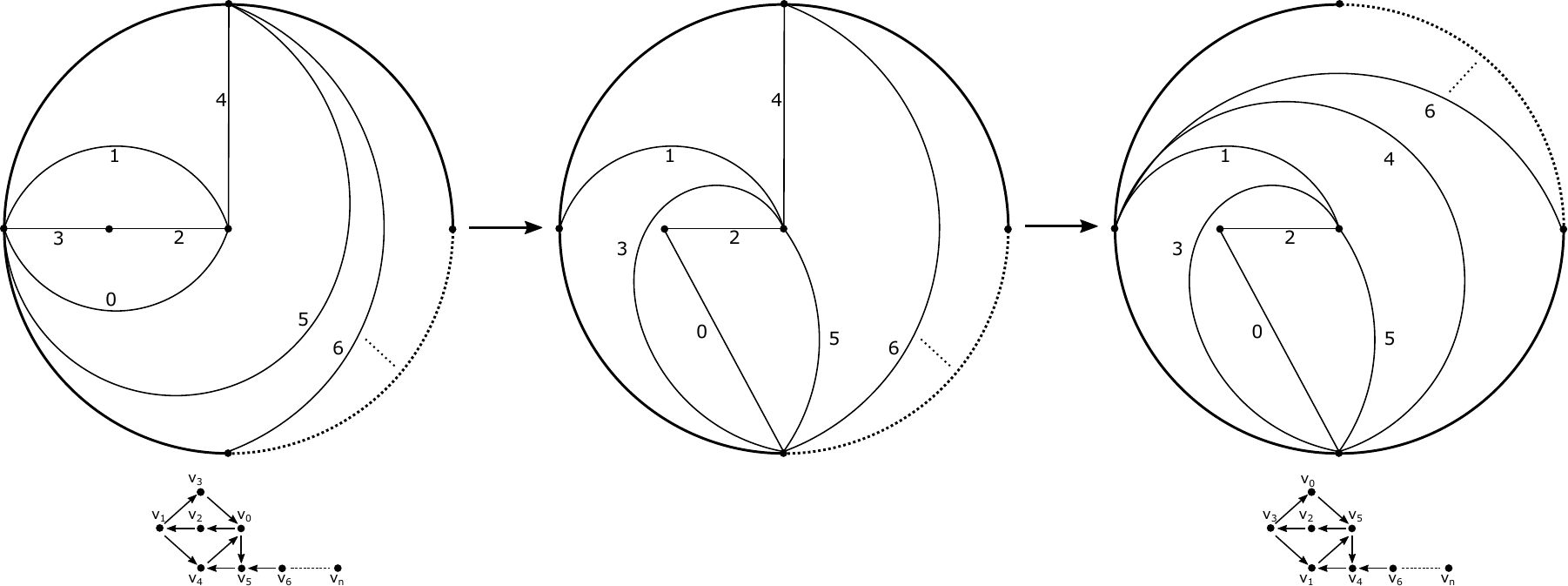}
    \caption{Sequence of mutations induced by the Legendrian loop $\vartheta_1$ in the tagged triangulation corresponding to $Q_{\mathbb{G}(\tilde{D}_n)}$. The second triangulation is obtained from the first by mutating at edges labeled $5,0,$ and $3$ in order. The third triangulation is obtained from the second performing the remaining mutations of $\vartheta_1.$}
    \label{fig:Theta1Triangulations}
\end{figure}

Similar to the case of $\tau_1$, the generator $\tau_2$ corresponds to a half twist about the two punctures in the triangulation $\mathcal{T}_2$. Figure \ref{fig:Theta2Triangulations} shows the computations for $\tilde{\vartheta}_2$. 

\begin{figure}
    \centering
    \includegraphics[width=.65\textwidth]{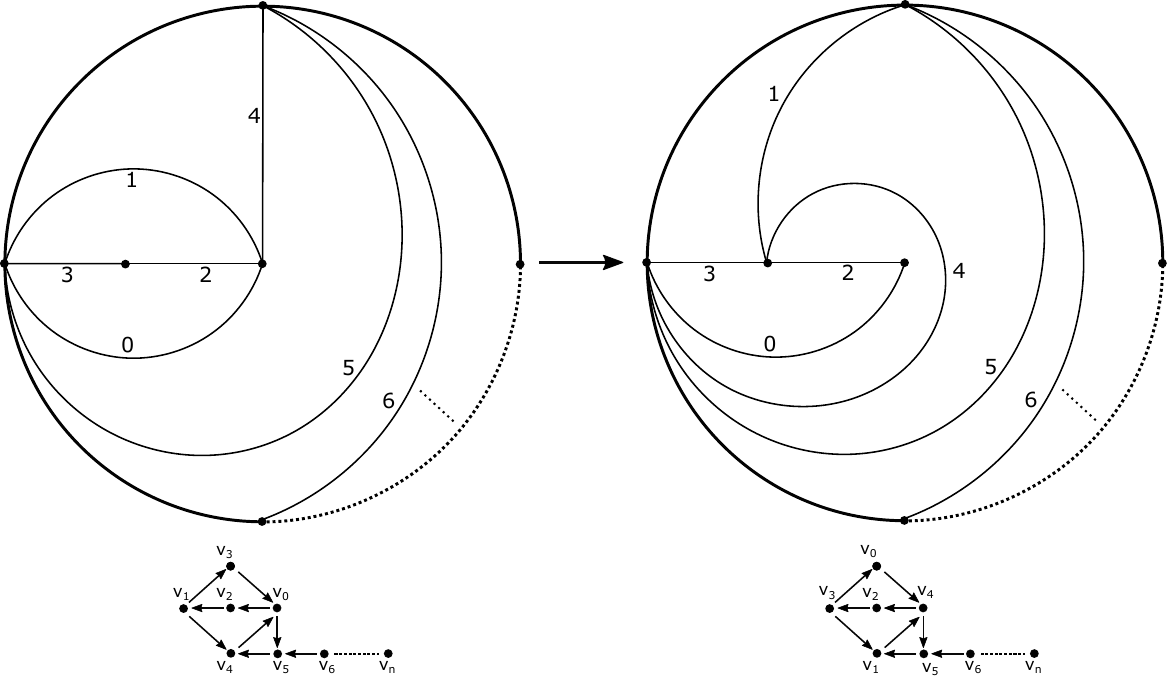}
    \caption{Sequence of mutations induced by the Legendrian loop $\vartheta_2$ in the tagged triangulation corresponding to $Q_{\mathbb{G}(\tilde{D}_n)}$.}
    \label{fig:Theta2Triangulations}
\end{figure}

For $n_2>2$, we observe that the sequence of mutations induced by $\tilde{\vartheta}_1$ or $\tilde{\vartheta}_2$ fixes the quiver vertices corresponding to the other tail. Freezing or deleting these vertices therefore yields a $\tilde{D}_n$ quiver and we can simply apply our above reasoning to show that cluster automorphisms are conjugate.  
\end{proof}

We now give a case-by-case description of $\vartheta$ loops as $\tau_i$ generators of $\Gamma_\tau$. We consider the cases of finite and affine type separately and postpone the cases of $\La(E_7^{(1, 1)})\cong \La(4, 4)$ and $\La(E_8^{(1,1)})\cong \La(3,6)$ to the following subsection as a subcase of Legendrian torus links.

\begin{proof}[Proof of Theorem \ref{thm: gens}] Throughout this proof, we freely reference the specific cluster modular groups and Legendrian links corresponding to a given Dynkin type as listed in Table \ref{tab: generating}.

\textbf{Finite type:}
\begin{enumerate}
    \item[$A_n$:] The K\'alm\'an loop of $\La(A_n)$ has order $n+3$ and generates $\SG(\FM(\La(A_n)))$, as stated in \cite[Corollary 4.2]{Hughes2021b}.

	\begin{center}
		\begin{figure}[h!]{ \includegraphics[width=.4\textwidth]{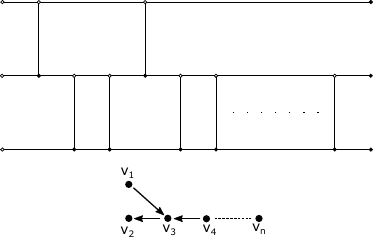}}\caption{Plabic fence $\mathbb{G}(D_n)$ with corresponding quiver. The dots represent $n-5$ additional vertical edges.}
			\label{fig: DnPlabicFence}\end{figure}
	\end{center}

\item[$D_n$:]  We use the combinatorics of tagged triangulations of once-punctured $n$-gons to understand the cluster modular group of $\FM(\La(D_n))$. By \cite[Theorem 1.2]{AssemSchifflerShramchenko}, the generators of the cluster modular group correspond to rotation of the $n$-gon by $2\pi/n$ and simultaneous changing of the tags at the puncture when $n\geq 5$. Let $\mathbb{G}(D_n)$ be the plabic fence pictured in Figure \ref{fig: DnPlabicFence}. The corresponding Legendrian $\La(\mathbb{G}(D_n))$ is Legendrian isotopic to the link $\La(D_n)$ defined in the introduction and admits a description as a Legendrian satellite $\La(\sigma_1^4, 1, \sigma_1^{n-2})$ of the Hopf link. We denote the corresponding loop by $\vartheta$ and use Lemma \ref{lemma: satellite loop mutations} to compute the induced sequence of mutations. In Figure \ref{fig: DnPuncturedDisks} we give a tagged triangulation with a quiver mutation equivalent to $Q_{\mathbb{G}(D_{n})}$ and show by explicit computation that the sequence of mutations induced by $\vartheta$ corresponds to rotation of the punctured $n$-gon by $2\pi/n$. In addition, the $\DT$ transformation induces a sequence of mutations that corresponds to a rotation of the once-punctured $n$-gon by $2\pi/n$ and a simultaneous changing of all the taggings at the puncture. Therefore, the pair $\vartheta$ and $\DT$ generate the cluster modular group $\SG(\FM(\La(D_n)))$. 
	\begin{center}
		\begin{figure}[h!]{ \includegraphics[width=\textwidth]{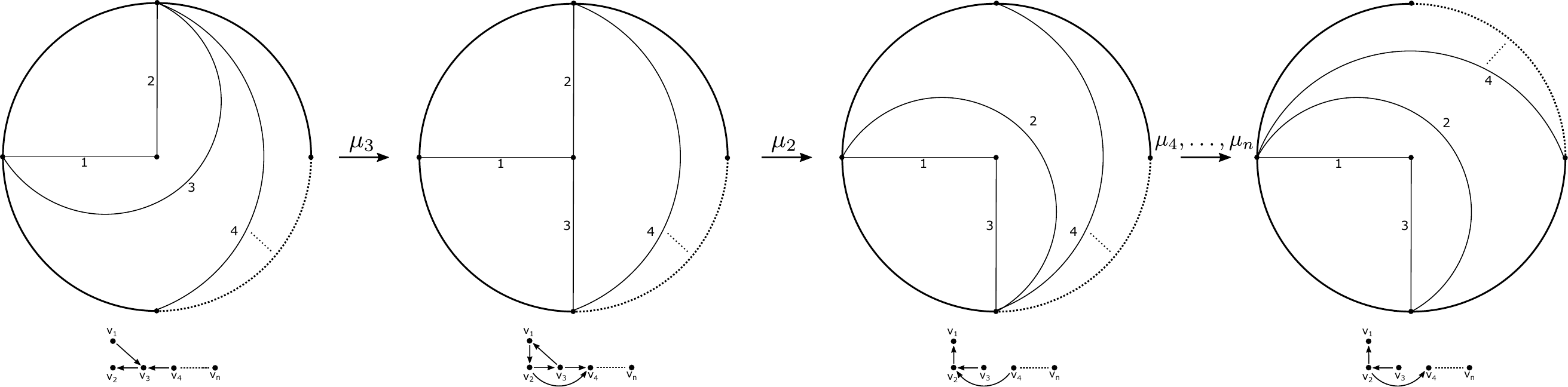}}\caption{Once-punctured disks depicting the mutation sequence $(3, 2, 4, \dots, n; (1\,2\,3))$ induced by the Legendrian loop $\vartheta$ of $\La(D_n)$.}
			\label{fig: DnPuncturedDisks}\end{figure}
	\end{center}
    
   \item[$E_6, E_7, E_8$]  For $\La(E_6), \La(E_7), \La(E_8)$, the cluster modular group is generated solely by the $\DT$ transformation, so the statement follows immediately from the characterization of $\DT$ given in \cite[Theorem 5.8]{CasalsWeng}.\\
   
\end{enumerate}    
\textbf{Affine type:}

\begin{enumerate} 

\item[$\tilde{D}_n$, $n>4$:]  Recall from Subsection \ref{sub: TnQuivers} that the cluster modular group $\SG$ of an affine type cluster algebra is isomorphic to $\Gamma_\tau\rtimes \Aut(T_{\bf{n}})$ and that $\Gamma_\tau$ is generated by $\tau_1$, $\tau_2,$ and $\tau_3$. By Lemma \ref{lemma: loops conjugate}, we can realize $\tau_1$ and $\tau_2$ as Legendrian loops, and the relation $\tau_1^{n_1}=\tau_2^{n_2}=\tau_3^{n_3}$ implies that the subgroup generated by $\tau_1$ and $\tau_2$ is a finite index subgroup of $\SG$. When $n>4,$ Kaufman and Greenberg's presentation contains a quiver automorphism of order 2 that does not obviously appear as an automorphism induced by a Legendrian loop, hence leading to the exclusion of $\La(\tilde{D}_n)$ in the statement of the theorem. 

\item[$\tilde{D}_4$:] The case of $\La(\tilde{D}_4)$ is analogous to $\La(\tilde{D}_n)$ when $n>4$ except that the quiver automorphism group for the $T_{2, 2, 2}$ quiver that Kaufman and Greenberg consider is the symmetric group $S_3$. We can produce an order-two generator of this $S_3$ factor by realizing $\La(\tilde{D}_4)$ as the $(-1)$-closure of the braid $\beta=(\sigma_2\sigma_1^3\sigma_2\sigma_1^3\sigma_2\sigma_3\sigma_1^{n-2})^2$. The $n+8$th power $\delta^{n+8}$ of the cyclic rotation then induces an involutive cluster automorphism, where we conjugate by the induced action of a Legendrian isotopy taking our original front for $\La(\tilde{D}_4)$ to our particular choice of $(-1)$-closure.

 \item[$\tilde{E}_7, \tilde{E}_8$:]   In the case of $\tilde{E}_7$ or $\tilde{E}_8$, an inspection of the quiver $T_{\bf{n}}$ verifies that that $\Aut(T_{\bf{n}})$ is trivial. Therefore, we need only show that we can generate $\Gamma_\tau$ by Legendrian loops and $\DT$. 
    By \cite[Theorem 4.14]{KaufmanGreenberg}, we have that $\DT=\tau_1\tau_2\tau_3\gamma^{-1}$. Solving for $\tau_3=\DT\tau_1^{-1}\tau_2^{-1}\gamma$ therefore gives the remaining generator of $\Gamma_\tau$. 
    
\item[$\tilde{E}_6$:]  In the case of $\tilde{E}_6$, we require a cluster automorphism of order three induced by a Legendrian loop. The Legendrian loop in question is obtained by performing a Legendrian isotopy $\psi$ to obtain the rotationally symmetric braid $(\sigma_2\sigma_1^4)^3$, composing with $\delta^5$ and then performing $\psi^{-1}$ to return to the original front projection. The induced action of this loop has order three, as desired.
\end{enumerate}

\end{proof}

\begin{table}[]
\begin{tabular}{l|l|l}
Cluster Type  & Cluster modular group & 
$T_{\bf{n}}$ quiver type \\
\hline
$A_n$      & $\Z_{n+3}$  
& -- \\

$D_n$ $n>4$ &$\Z_{n}\times\Z_2$  
& --\\

$E_6, E_7, E_8$     & $\Z_{14}, \Z_{10}, \Z_{16}$            
& --\\


$\tilde{D}_n$ $n>4$  & $\Mod_{tag}(\D^2_{n-2, 0, 0})$     & 
(n-2, 2, 2)
\\

$\tilde{E}_6$ & $(S_3\times \Z)$     &  (3, 3, 2)     \\
$\tilde{E}_7$ & $(\Z_2\times \Z)$     &   
 (4, 3, 2)  \\
$\tilde{E}_8$ & $\Z$     &       
(5, 3, 2)\\


$E_7^{(1, 1)}$       &    $\pi_1(Conf(\S^2, 4))\rtimes \Z_2$                                    & (4, 4, 2)   \\                             

$E_8^{(1,1)}$        &             $\PSL_2(\Z)\rtimes \Z_6$                         & (6, 3, 2)              
\end{tabular}
\caption{\label{tab: generating}
}\end{table}

\begin{remark}
The arguments employed in the affine type case apply to the slightly more general family of Legendrian links $\La(\beta_{1, n_1, n_2})$. These links correspond to $T_{n_1, n_2, 2}$ quivers with $n_1\geq n_2 > 2$ and Kaufman and Greenberg conjecture that the cluster modular group obtained from such quivers is isomorphic to $\Gamma_\tau$ when $n_1\neq n_2$ and $\Gamma_\tau \rtimes \Z_2$ when $n_1=n_2$. By Lemma \ref{lemma: loops conjugate} and the fact that $\tau_3$ can be obtained from  $\tau_1, \tau_2$ and $\DT$, we are able to generate $\Gamma_\tau$ by the induced action of Legendrian loops. In addition, the anti-symplectic involution $\iota$ induces the quiver automorphism generating $\Aut(T_{\bf{n}})$ when $n_1=n_2$. As a result, we obtain a symplectic-geometric description of the conjectured cluster modular group of Kaufman and Greenberg \cite[Conjecture 4.7]{KaufmanGreenberg} in this special case. 
\end{remark}

\subsection{Torus links and cluster modular groups of $\Gr^\circ(k, n+k)$}\label{sec: Torus Links}

In this subsection, we show that Legendrian loops of torus links $\La(k,n)$ generate an index-two subgroup of a conjectural cluster modular group of $\Gr^\circ (k, n+k)$, the affine cone of the Grassmannian.  
We start by reviewing the relationship between $\FM(\La(k, n))$ and $\Gr^\circ(k, n+k)$. We then describe the conjectural cluster modular group $\SG'(k, n)$ and an explicit generating set given in \cite{fraser2018braid}. Finally, we show that the Legendrian loops $\Sigma_i$ and $\rho=\delta^k$ induce generators of $\SG'(k, n)$.

\subsubsection{Relationship between $\FM_1(\La(k,n), T)$ and $\Gr^\circ(k, n+k)$}

We now briefly describe an explicit isomorphism between $\FM_1(\La(k,n), T)$ and $\Gr^\circ(k, n+k)$. For more details, we refer the interested reader to \cite[Section 6]{CLSBW23} for a detailed description in the more general setting of Positroid strata of $\Gr^\circ(k, n+k)$.

Consider a front projection of $\La(k, n)\subseteq J^1\S^1$ corresponding to the braid word $\beta_{k, n}=(\sigma_1\dots \sigma_{k-1})^{n+k}$ and add marked points following the convention given by \cite[Theorem A.(iii)]{CLSBW23}. Following the description of $\FM(\La, T)$ in Subsection \ref{sec: sheaves}, the sheaf moduli of $\La(k,n)$ is then given by a sequence of flags $V_1, \dots, V_{k+n}$ each of which satisfies $V_i^{(j)}\pitchfork V_{i+1}^{(j)}$ for every $i$ and $j$. In the particular case of $\La(k,n)$, the choices of (complex) lines $v_1, \dots, v_{n+k}$ corresponding to the regions between pairs of $\sigma_1$ crossings determine all of the other higher-dimensional vector spaces $V^{(j)}, j \geq 2$ making up the remaining data of each of flags. Indeed, the flag $V_i$ is given by $v_i \subseteq v_{i}\wedge v_{i+1}\subseteq \dots \subseteq v_i\wedge \dots \wedge v_{i+k-1}$ with indices taken modulo $n+k$. See Figure \ref{fig: T44Flags} for an example computation in the case of $\La(4, 4)$.

	\begin{center}
		\begin{figure}[h!]{ \includegraphics[width=.4\textwidth]{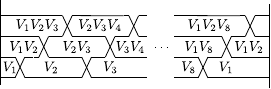}}\caption{Flags making up $\FM(\La(4, 4))$. Here $V_iV_j$ denotes the subspace $V_i\wedge V_j$.}
			\label{fig: T44Flags}\end{figure}
	\end{center}

Let $\Gr(k, n)$ the space of $k$ dimensional subspaces of $\C^n$ and consider its Pl\"ucker embedding. We denote by $\Gr^\circ(k, n)\subseteq \C^{n \choose k}$ the affine cone of the Grassmannian, i.e. the affine subvariety of $\C^{n \choose k}$ whose points satisfy the Pl\"ucker relations. By the transversality conditions defining $\FM(\La(k,n), T)$ and the framing data given by the choice of $T$, the vectors $v_1,\dots, v_{n+k}$ form the columns of a full rank $k\times n+k$ matrix. This determines an explicit isomorphism between $\FM(\La(k, n), T)$ and $\Gr^\circ(k, n+k)$
by interpreting the minors of the resulting matrix $(v_1 \dots v_{n+k})$ as the Pl\"ucker coordinates of a point in the affine cone of the Grassmannian.

The affine variety $\Gr^\circ(k, n)$ admits a cluster structure with an initial seed given by a maximal weakly-separated collection of Pl\"ucker coordinate functions \cite{Scott06}. For certain choices of $k$ and $n,$ we recover Legendrian links of finite and affine type. In these cases, one can verify that the results below recover the cluster modular groups described in the previous subsection.

\subsubsection{Conjectural presentation of $\SG(\Gr^\circ(k, n+k))$} 

We give two types of cluster automorphisms of $\Gr^\circ(k, n+k)$. The first, known as the cyclic shift, is relatively well-studied in the cluster literature and played a key role in proving that the totally non-negative Grassmannian is homeomorphic to a closed ball \cite{GalashinKarpLam22}. Given vectors $v_1, \dots, v_{n+k}$ satisfying the transversality conditions of $Gr^\circ(k, n+k)$, the cyclic shift $\rho$ acts by $(v_1, \dots,  v_{n+k})\mapsto (v_{2}, \dots v_{n+k}, (-1)^{k-1}v_1)$. From the construction, we can immediately observe that $\rho$ has order $n+k$. 

The second type of cluster automorphism is the main object of study in \cite{fraser2018braid}, where Fraser produces a braid group action on $Gr^\circ(k, n+k)$. We first denote $d=\gcd(k, n+k)$ and assume that $d>1$. Given $v_1, \dots, v_{n+k}$ as above and $i\in \{1, \dots, d-1\}$,  the action of $\sigma_i$ on the first $d$ vectors is
\begin{equation}\label{eq: SigmaLoop} 
    (v_1, \dots, v_d)\mapsto (v_1, \dots, v_{i-1}, v_{i+1}, w_1, v_{i+2}, \dots, v_d) 
\end{equation}
with $w_1$ defined by the conditions 
\begin{equation}\label{eq: W1Def}
    v_i\wedge v_{i+1}=v_{i+1} \wedge w_1 \text{ and } w_1 \in \text{span}\{v_2, \dots, v_{i+k}\}
\end{equation}
Note that Equations \ref{eq: SigmaLoop} and \ref{eq: W1Def} together imply that $w_1\in \text{span}\{v_i, v_{i+1}\}\cap \text{span}\{v_2, \dots, v_{i+k}\}$ with the normalization given by Equation \ref{eq: SigmaLoop} ensuring that $w_1$ is uniquely defined. The cluster automorphism $\sigma_i$ is defined on the remaining $n+k-d$ vectors by adding multiples of $d$ to the indices of the vectors and defining $w_j$ analogously.

The extended affine braid group $\widehat{\Br}_{d-1}$ on $d$ strands is generated by elements $\rho, \sigma_1, \dots, \sigma_{d-1}$ with relations 
$\sigma_{i+1}\sigma_i\sigma_{i+1}=\sigma_i\sigma_{i+1}\sigma_i$, $\sigma_i\sigma_j=\sigma_j\sigma_i$ for $|j-i|\geq 2$, and $\rho\sigma_i\rho^{-1}=\sigma_i+1$ for $i\in \{1, \dots, d-2\}$ and $\rho\sigma_{d-1}\rho^{-1}=\rho^{-1}\sigma_1\rho$. Fraser's main result states that $\sigma_1, \dots, \sigma_{d-1}$ defined above together with the cyclic shift generate an extended affine braid group action on $\Gr^\circ (k, n+k)$ \cite[Theorem 5.3]{fraser2018braid}. Denote by $G_d'=\widehat{Br}_{d-1}$ when $d<k$ and $G_k'=Br_k$ when $d=k$. Fraser obtains a homomorphism $G_d'\to \mathcal{G}(\Gr^\circ(k, n+k))$ and conjectures that for $n\geq 2k$, the kernel of this homomorphism is generated by $\rho^n$. For $n=k$, he identifies the additional relation $\sigma_1\dots \sigma_{k-1}^2\dots\sigma_1$. These yield finite index subgroups of the conjectured cluster modular groups, as the addition of the DT transformation to the generating set adds a $\Z_2$ factor given by the relation $\DT^2= \rho^k$. In particular, for $n=k$, Fraser's conjectural cluster modular groups are isomorphic to $\pi_1(Conf(\S^2, k))\rtimes \Z_2$ where $Conf(\S^2, k)$ is the configuration space of $k$ points on the sphere. For the purpose of comparing with the results of \cite{CasalsGao}, note that $Conf(\S^2, k)$ modulo its center is isomorphic to $\Mod(\Sigma_{0, k})$. See Table \ref{tab: CMG Gr} for a summary of which values of $k$ and $n$ yield the groups mentioned above.

 \begin{table}[]
\begin{tabular}{l|l|l}
$\La(k, n)$  & $d$ & $\SG'(k, n)$   \\
\hline
$\La(2, n)\,  n \geq 3$ & $n \mod 2$  & $\Z_{n+2}$ \\

$\La(k, k) \, k\geq 3$ & $k$ & $\pi_1(Conf(\S^2, k))\rtimes \Z_2$\\

$\La(k, n) \, n\geq 2k$    & $k$            
& $(\Br_k\rtimes \Z_2)/\langle \rho^{n+k}\rangle$\\

$\La(k, n) \, n \geq k+1$ & $<k$  & $(\widehat{Br}_{d-1}\rtimes \Z_2)/\langle\rho^{n+k}\rangle$

\end{tabular}
\caption{\label{tab: CMG Gr}
}\end{table}

\subsubsection{Legendrian loops induce Fraser's automorphisms}\label{sub: torus loops}

In \cite{CasalsGao}, the authors show that the K\'alm\'an loop induces the cyclic shift automorphism $\rho$ on $\SM_1(\La(3,6))$ and $\SM_1(\La(4, 4))$ while the $\vartheta$ loops they describe induce $\sigma_i$ actions. Their work predates any proof of a cluster structure on these sheaf moduli, so they rely on combinatorial tools to verify that the induced actions of these Legendrian loops yield faithful group actions on the sheaf moduli. Our goal in this subsection is to describe the relevant Legendrian loops as cluster automorphisms and extend their work to the general setting of $\Gr^\circ(k, n+k)$.

We begin with the following lemma relating the K\'alm\'an loop $\rho$ to the cyclic shift automorphism. 

\begin{lemma}\label{lemma: Kalman loop} 
    The Legendrian loop $\rho$ induces the cyclic shift on $\FM(\La(k, n), T)\cong \Gr^\circ(k, n+k)$.
\end{lemma}
\begin{proof}
    The statement immediately from the description of the isomorphism between $\FM(\La(k, n), T)$ and $\Gr^\circ(k, n+k)$ and the definition of the K\'alm\'an loop in the front given by the braid word $(\sigma_1,\dots,\sigma_{k-1})^{n+k}$. Indeed, the induced action of $\rho$ sends the flag $V_i$ to the flag $V_{i-1}$, which, up to sign, is the desired map $v_i\mapsto v_{i-1}$ on 1-dimensional subspaces.
\end{proof}

To realize the braid group generators of $\SG'(\Gr^\circ(k, n+k))$ as Legendrian loops, we need to expand the class of loops we consider to include the following generalization of Definition \ref{def: theta loop}:
Let $d=\gcd(k,n)$, $m=\frac{k}{d}$, and assume that $d>1$. While the Coxeter projection $\pi: \Br_n \to S_n$ has no fixed points, it does admit a fixed set $\{i, i+d, \dots, i+(m-1)d\}$ of size $m$ where we consider indices $\mod k$. For an adjacent pair of strands, $i$, $i+1$ corresponding to a set of fixed orbits, $\La(k, n)$ -- presented as the $(-1)$-framed closure of $(\sigma_1\dots, \sigma_{k-1})^{n+k}$ for specificity -- is Legendrian isotopic to the satellite of a positive braid about these fixed strands; in other words $\La(k,n)\cong \La(\beta, \{i, \dots, i+(m-1)d\}, \gamma\}$ for some $\beta\in \Br_{k-1}^+$ and $\gamma\in \Br_2^+$, where we have slightly abused notation to replace the natural number $i$ in the definition from the previous section with the set of natural numbers $\{i, \dots, i+(m-1)d\}$. We define $\Sigma_i$ analogously to a $\vartheta$-loop of $\La(\beta, i, \gamma)$ as the isotopy given by the flow of the vector field in the open neighborhood of the strands $\{i, \dots, i+(m-1)d\}$. In the front given by the $(-1)$-closure of $(\sigma_1\dots, \sigma_{k-1})^{n+k}$, the loop $\Sigma_i$ can be represented by a sequence of Reidemeister III describing a single crossing orbiting the link $m$ times, as described below. 

Note that in the case of $d=k$, the fixed set of the Coxeter projection is simply $i$ and the $\Sigma_i$ loop coincides with a $\vartheta$-loop of $\La(\beta, i, \gamma)$. 

The following lemma describing the action of the $\Sigma_i$ loops involves a somewhat more detailed computation.

\begin{lemma}\label{lemma: sigma loop}
    The Legendrian loop $\Sigma_i$ 
    induces the cluster automorphism $\sigma_i$ of $\Gr^\circ(k, n+k)$ under the isomorphism between $\FM(\La(k, n), T)$ and $\Gr^\circ(k, n+k)$.
\end{lemma}

\begin{proof}
    Let $d=\gcd(k, n+k)$ and consider the induced action of the Legendrian loop $\Sigma_i$ on $\FM(\La(k,n))$. Following the computations in Sections 4 and 5 of \cite{CasalsGao}, we decompose the Legendrian isotopy $\Sigma_1$ into a sequence of Reidemeister III moves.

    We restrict our consideration of the front to the first $d(k-1)$ crossings in the front projection $\Pi_{xz}(\La((\sigma_1\dots\sigma_{k-1})^{n+k}))\subseteq J^1\S^1$ and extend by symmetry to the remaining $(n+k-d)(k-1)$ crossings. We label the strands from bottom to top at their left endpoint. Observe that in the first $k(k-1)$ crossings of the initial front projection any fixed pair of strands labeled $i$ and $i+1$ cross exactly twice. In the braid $(\sigma_1\dots, \sigma_{k-1})^k$, these two crossings are exactly the $i$th $\sigma_1$ and the $i+1$st $\sigma_{k-1}$ with indices taken modulo $k$.

The $\Sigma_i$ loop is then given by a sequence of Reidemeister III moves that take the $i$th $\sigma_1$ crossing and pass it through the other strands until we obtain a $\sigma_i$ as the leftmost crossing. Repeating this process symmetrically for each block of $d(k-1)$ crossings and performing a cyclic rotation then produces a $\sigma_{i}$ crossing at the right of the first block of $d(k-1)$ crossings. If $d=k$, then we can perform another series of Reidemeister III moves to pass this crossing through the other strands until we encounter the $i-1$th $\sigma_{k-1}$ crossing. We then perform a sequence of Reidemeister III moves to pass the $\sigma_{k-1}$ crossing through the remaining strands until we arrive back at the front that we started with. If $d<k$, then after passing the $\sigma_{k-1}$ crossing through the remaining strands, we obtain $\sigma_{i+d}$ as the leftmost crossing. We perform another cyclic shift and repeat this process $m$ times in total for $m=\frac{k}{d}$. 

The induced action on $\FM(\La(k,n), T)$ can be described as follows. We can view each Reidemeister III move as transferring a region of the front diagram from the right of a strand to the left of the same strand, fixing all the remaining regions. Under the natural identification of sheaves before and after this isotopy, the vector spaces corresponding to the fixed regions are all preserved and we need only compute the vector space corresponding to the new region. Following our description of $\FM(\La((\sigma_1\dots \sigma_{k-1})^{n+k}), T)$ above, the vector space assigned to the new region is determined in the cases where it is of dimension at least two. In the case where the region appears between a pair of $\sigma_1$ crossings and the assigned vector space $W_j$ has rank one, we determine $W_j$ as follows. By construction, a new region $W_j$ appears whenever we replace the $i+jd$th $\sigma_1$ of $(\sigma_1\dots \sigma_{k-1})^k$ with the $i+jd+1$th $\sigma_{k-1}$ crossing for $j\in \{1, \dots, m-1\}$. This has the effect of replacing $v_{i+jd}$ by $v_{i+jd + 1}$ and $v_{i+jd+ 1}$ by $W_j$ while preserving $v_{i+jd+2}$. The singular support conditions then imply that $W_j$ lies in the two-dimensional subspace $V_{i+jd}^{(2)}$ assigned to the region northwest of it. Similarly, it must also lie in the $k-1$-dimensional subspace $V_{i+jd+1}^{(k-1)}$ that is assigned to the unique upper region directly to the northeast of the $W$ region. Therefore, the span of $W$ is determined by the intersection $V_{i+jd}^{(2)}\cap V_{i+jd+1}^{(k-1)}$ and the particular choice of framing data can be given by the identification of $V_{i+jd}^{(2)}$ with the span of $v_i$ and $W_i$, as required by the singular support conditions. Up to indexing, these are equivalent to the conditions given by Equation \ref{eq: W1Def} determining $w_1$. Thus, the induced action of $\Sigma_i$ precisely matches Equation \ref{eq: SigmaLoop}, as desired.
\end{proof}

\section{Cluster modular groups from folding $\FM(\La)$}\label{Sec: Folding} 
In this section, we study Legendrian loops and Legendrian links that exhibit symmetry under certain finite group actions. These correspond algebraically to folded cluster algebras, skew-symmetrizable cluster algebras that are constructed by identifying symmetries of quivers that respect cluster mutation. The process of folding a skew-symmetric cluster algebra is somewhat delicate in general, as symmetries of a quiver are not necessarily preserved by mutation. In this section, we discuss the algebraic process of folding a cluster algebra by a $G$-action and then give a contact-geometric description of it in terms of $G$-invariant Legendrian links and their $G$-fillings. We then produce generators of cluster modular groups of the folded cluster algebras using our results from Section \ref{sec: CMGs}.

\subsection{Folding clusters via G-actions} \label{sub: folding}

We begin with the properties of a foldable quiver. See \cite[Section 4.4]{FWZ2} or \cite[Section 2.3]{ABL22} for more details. Let $Q$ be a quiver and $G$ be a finite group action with $g\cdot Q=Q$ for all $g\in G$. We write $i\sim i'$ for vertices $v_i$ and $v_{i'}$ lying in the same $G$-orbit and denote by $\#\{i \rightarrow j\}$ the number of edges from $v_i$ to $v_j$. 
\begin{definition}
  The quiver $Q$ is $G$-admissible if the following conditions hold:
  \begin{enumerate}
      \item If $i\sim i'$, then $i$ and $i'$ are either both mutable or both frozen.
      \item For all $i \sim i'$ and any $g\in G$, we have $\#\{i \rightarrow j\}=\#\{g\cdot i \rightarrow g\cdot j\}$.
      \item  For all $i \sim i'$, we have $\#\{i \rightarrow i'\}=0$.
      \item For all mutable $i \sim i'$, we have $\#\{i \rightarrow j\}\#\{i'\rightarrow j\}\geq 0$.
  \end{enumerate}
\end{definition}

Let $I$ be the $G$-orbit of a vertex $v_i$. We denote by $\#\{I \rightarrow J\}$ the sum $\sum_{i\in I} \#\{i\rightarrow j\}$ for some arbitrary $j\in J$. To package the data $\#\{i\rightarrow j\}$ as part of a graph, we produce a weighted quiver. That is, given a $G$-admissible quiver $Q$, we consider the weighted graph $Q^G$ whose vertices are $G$-orbits of vertices in $Q$ where $v_i$ and $v_j$ have an edge of weight $\#\{I \rightarrow J\}$ between. Note that the weighting is distinct from multiplicity, as $\#\{I \rightarrow J\}\neq -\#\{J\rightarrow I\}$ in the folded case. 

Alternatively, we can define a matrix $\tilde{B}^G$ with entries $b_{ij}^G=\sum_{i\in I} b_{ij}$ where $b_{ij}$ is the $(i,j)$ entry of the original exchange matrix $B$. The mutable part of the exchange matrix $\tilde{B}^G$ is skew-symmetrizable, i.e. there is some diagonal matrix $D$ with positive integer entries such that $B^GD$ is skew-symmetric. Here the $j$th nonzero entry of $D$ is given by the size of the orbit of the quiver vertex $v_j$. 

Given a $G$-admissible quiver $Q$ and a mutable $G$-orbit $I$, we can consider a sequence of mutations $\mu_I=\prod_{i\in I}\mu_i$. When both $Q$ and $\mu_I(Q)$ are $G$-admissible, we have that $\mu_I(Q^G)=\mu_I(Q)^G$ where mutation of the folded quiver $Q^G$ is defined as above, replacing $b_{ij}$ by $b_{ij}^G$. We call a quiver $Q$ with a $G$-action globally foldable with respect to $G$ if $Q$ is $G$-admissible and, for any sequence of mutable $G$-orbits $I_1, \dots I_k$, the quiver $\mu_{I_k}\dots \mu_{I_1}(Q)$ is $G$-admissible as well.

 Let $(\bf{a}, Q)$ be a globally foldable seed with respect to a $G$-action. Recall that for $\bf{a}=(a_1, \dots, a_m)$, we denote by $\mathcal{F}$ the field of rational functions in $m$ variables. Let $\mathcal{F}^G$ be the field of rational functions in $m_G$ variables, where $m_G$ denotes the number of $G$-orbits of vertices $v\in Q$, equivalently the number of vertices in the weighted quiver $Q^G$.  

\begin{definition}
 Given a cluster seed $({\bf a}, Q)$, the data of a folded seed is given by the weighted quiver $Q^G$ and the variables $a_{I_1}, \dots a_{I_{m_{g}}}$ assigned to corresponding vertices.   
\end{definition}

The key consequence of the globally foldable condition is that such a definition allows us to mutate at any mutable element in a folded seed to obtain another folded seed. The collection of cluster variables given by arbitrary mutations of the folded seed then defines the folded cluster algebra.

All of the cluster algebras in Table \ref{tab: folded} are globally foldable with respect to the indicated $G$-action. See e.g. \cite[Chapter 5]{FWZ2} for a thorough description of the finite type cases. For affine type, the globally foldable condition is verified by An and Lee \cite[Corollary 3.8]{AnLee22}. For extended affine type, the folded cluster algebras we consider in this work are `exotic' foldings of $E_7^{(1, 1)}$ and $E_8^{(1, 1)}$ of Kauffman and Greenberg, identified in \cite{KaufmanGreenberg}. The folding of $E_7^{(1, 1)}$ is by an involution, which implies global foldability by  
\cite[Lemma 2.8]{kaufman2023special}. For $\La(E_8^{(1, 1)})$, we observe that the $G_2^{(3, 3)}$ quiver obtained by Kaufman and Greenberg has only two distinct mutation classes, allowing direct verification of global foldability. The cluster algebras, quivers, and $G$-actions relevant to this work are given in Figure \ref{fig: FoldingQuivers} below.

	\begin{center}
		\begin{figure}[h!]{ \includegraphics[width=.9\textwidth]{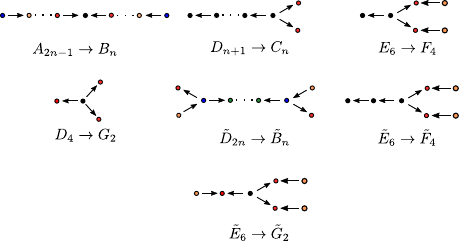}}\caption{Quivers of finite and affine type. $G$-actions giving the specified foldings are realized by identifying vertices of the same color.}
			\label{fig: FoldingQuivers}\end{figure}
	\end{center}

\subsection{$G$-actions on Legendrian links}\label{sub: G-actions}

In this subsection, we give a description of the group actions realizing folding as certain contactomorphisms. Consider a finite group $G$ action on $(J^1\S^1, \xi_{\st})$ by contactomorphisms and 
suppose that the $G$-action extends to the symplectization $\Symp(J^1\S^1, \xi_{\st})$ by exact or anti-exact symplectomorphisms. Given such a group action, we give the following definition.

\begin{definition}
    An exact Lagrangian $G$-filling is an exact Lagrangian filling of a $G$-invariant Legendrian link $\La$ such that the $G$ action on $\Symp(J^1\S^1, \xi_{\st})$ fixes $L$ and the induced action on the boundary contact manifold fixes $\La$. 
\end{definition}

Casals considered these $G$-fillings in his conjectural classification of exact Lagrangian fillings of Legendrian links 
corresponding to Dynkin types $B_n, C_n, F_4$ and $G_2$ \cite[Conjecture 5.4]{CasalsLagSkel}. An, Bae, and Lee subsequently proved the existence of the conjectured number of $G$-fillings for these links \cite[Theorem 1.4]{ABL22}.

We now define a pair of specific $G$-actions that we later show correspond to classical folding of cluster algebras of finite and affine type.

\textbf{Rotation:} For a Legendrian link $\La$ in $(J^1\S^1, \xi_{\st})$, a translation in the $\theta$ coordinate of the $\S^1$ factor induces an action on $\La$. For any exact Lagrangian filling of $\La$ obtained as the Lagrangian projection of a weave $\w\subseteq (J^1\D^2, \xi_{\st})$, the translation action naturally extends to a rotation of the base $\D^2$, inducing an action on $L(\w)$. We denote a rotation through an angle of $\theta$ by $R_{\theta}$.

\textbf{Involution:} 
Consider an anti-symplectic involution of the symplectization $(\R_r\times J^1\S^1, d(e^r\alpha))$ for $\alpha=dr-p_\theta d\theta$ --- thought of as $T^*\D^2$ --- given by $p_\theta\mapsto -p_\theta, p_r\mapsto -p_r$ and fixing the $r$ and $\theta$ coordinates. In the contactization, $(J^1D^2, dz-e^r\alpha)$ we can lift this antisymplectic involution to a contactomorphism by $z\mapsto -z$. We denote this involution by $\iota$.

Note that for any Lagrangian $L$ fixed by the antisymplectic involution $\iota$, we still have $\omega_{\st}|_{\iota^*(L)}=-\omega_{\st}|_{L}=0$. Likewise, if $L$ is exact, then we have 
\begin{equation*}
    \alpha_{\st}|_{\iota^*(L)}=-\alpha_{\st}|_{L}=-dz|_{L}
\end{equation*} for some function $z$ on $L$. As a result, the image of an exact Lagrangian under the antisymplectic involution is again an exact Lagrangian. Similarly, the image of a Legendrian under the contactomorphism $\iota$ remains Legendrian.

Following \cite{ABL22}, we can produce $\iota$-invariant Legendrian surfaces by presenting them as `degenerate' weaves with symmetry. In the front projection, $\iota$ is given by $z\mapsto -z, \theta\mapsto \theta$ and $r\mapsto r$, so that an $\iota$-invariant filling is one which is represented by a weave that is symmetric through a reflection of the $r\theta$-plane. These weaves generally have non-generic fronts where the singular locus of the projection $\Pi:J^1\D^2\to \R_z\times D^2$ is given by overlapping $A^2_1$ singularities. By construction, these non-generic singularities are represented by overlapping edges labeled by $\sigma_i$ and $\sigma_j$ with $|i-j|\geq 2$. The resulting Legendrian weave is nevertheless a non-singular embedded surface in $J^1\D^2$ as the overlapping edges represent singular loci of the front projection that occupy distinct $z$-coordinates. See also \cite[Section 8.1]{CLSBW23}
 for an explicit description of the induced action of $\iota$ on the flag moduli $\FM(\La)$. 

\subsection{Cluster structures on $\FM(\La(\beta))^G$}\label{sub: folding sheaves}

In this subsection we obtain a folded cluster algebra from the sheaf moduli of certain $G$-invariant Legendrians. Denote by $\FM(\La, T)^G$ and $\SM_1(\La, T)^G$ the $G$-invariant moduli stacks of sheaves. More concretely, these moduli are the collection of $G$-invariant flags satisfying the transversality conditions described in Section \ref{sec: sheaves} together with $G$-invariant framing data. Note that the $G$-invariance of framing data requires that the set of marked points $T$ must be $G$-invariant as well. We claim that for certain $G$-actions, these moduli stacks admit skew-symmetrizable cluster structures with contact-geometric descriptions.

  In pursuit of this claim, we restrict the class of Legendrian links we consider in this subsection to the case of rainbow closures $\La(\beta)$. We do so in order to satisfy the hypotheses of \cite[Theorem 1.1]{CasalsGao23}, which allows us to realize any cluster mutation as a Lagrangian disk surgery at some $\mathbb{L}$-compressing cycle. As observed by the authors (see Remark 4.9 and Section 4.10 of  loc.~cit.), their main results likely hold for more general families of Legendrian links as well.

\subsubsection{Lattices from $G$-fillings}
We now define a collection of lattices comprising the fixed data of a cluster ensemble of skew-symmetrizable cluster algebras, following 
the exposition given in \cite[Section 2]{GHK15}. First, consider a lattice $N$ with a skew-symmetric bilinear form. The lattice $N$ contains an unfrozen sublattice $N_{\uf}$, which is a saturated sublattice of $N$. For skew-symmetrizable cluster structures, there is an additional sublattice $N^\circ$ of $N$ that is of finite-index in $N$. We also have dual lattices $M:= \Hom(N, \Z)$ and $M^\circ :=\Hom(N^\circ, \Z)$ where $M$ is necessarily a finite-index sublattice of $M^\circ$. The lattice $N$ is the character lattice for the $\SX$ cluster tori, while the lattice $M^\circ$ gives the character lattice for the $\SA$ cluster tori. The toric coordinates corresponding to the basis vectors of the corresponding lattices yield cluster variables of the appropriate cluster charts.

Now let $\La(\beta)$ be a Legendrian link such that the ring of regular functions $\C[\FM(\La(\beta), T)]$ is a globally foldable cluster algebra. We first consider the lattices of $G$-orbits of elements of $H_1(L, T)$ and $H_1(L\backslash T, \La \backslash T)$, which we denote by $H_1(L, T)^G$ and $H_1(L\backslash T, \La \backslash T)^G$. The lattices $H_1(L, T)^G$ and $H_1(L\backslash T, \La \backslash T)^G$ are not dual, rather, each lattice is contained in the dual of the other. Following the formulation for skew-symmetrizable cluster varieties described above, we designate these dual lattices $M^\circ(\La)=\Hom((H_1(L, T)^G), \Z)$ and $N(\La)=\Hom(H_1(L\backslash T, \La\backslash T)^G), \Z)$ as the character lattices of the cluster charts we construct. Analogous to the construction of Casals and Weng given in \cite[Section 3.8]{CasalsWeng}, the intersection pairing on $H_1(L, T)^G$ induced from $H_1(L, T)$ yields the required skew-symmetric form on the lattices.

\subsubsection{Mutation of $G$-fillings}
 
We say that a collection $\{\gamma_i\}_{i=1}^n$ of $\mathbb{L}$-compressing cycles are simultaneously mutable if mutations at any two of the cycles $\gamma_i, \gamma_{i'}$ in the collection pairwise commute.

\begin{lemma}\label{lemma: simultaneously mutable}
   Let $G$ be a group acting by exact (anti)-symplectomorphisms on $\Symp(J^1\S^1, \xi_{\st})$. For any $G$-filling $L$ of $\La(\beta)$ with a $G$-admissible intersection quiver, all $\mathbb{L}$-compressing cycles in $H_1(L)$ belonging to the same mutable $G$-orbit are simultaneously mutable and mutating at every cycle in a $G$-orbit produces another $G$-filling.
\end{lemma}

\begin{proof}
Let $L$ be a $G$-filling of $\La(\beta)$ with a $G$-admissible intersection quiver. The $G$-admissible condition implies that any two cycles $\gamma_i$ and $\gamma_{i'}$ belonging to the same $G$-orbit have algebraic intersection number $\langle \gamma_i, \gamma_{i'}\rangle=0$. By \cite{CasalsGao23}, there is a Hamiltonian isotopy of $L$ so that the geometric intersection of cycles in $H_1(L)$ matches the algebraic intersection, ensuring that each $\mathbb{L}$-compressing cycle can be mutated without creating immersed cycles. Moreover, the condition that $\#\{i\to j\}\#\{i'\to j\}\geq 0$ implies that after mutation at either $\gamma_i$ or $\gamma_{i'}$, the two cycles do not intersect and mutation at one does not affect any of the cluster $A$-variables adjacent to the other. Therefore, the two mutations commute, as desired. In particular, this implies that for any two cycles $\gamma_i$ $\gamma_{i'}$ belonging to the same $G$-orbit, we can successively perform the Lagrangian disk surgeries $\mu_{\gamma_i}$ and $\mu_{\gamma_{i'}}$, and either ordering produces Hamiltonian isotopic fillings of $\La(\beta)$.

To verify the part of the statement involving producing additional $G$-fillings, we observe that the Hamiltonian isotopy of $L$ given in \cite{CasalsGao23} to ensure that geometric intersections match algebraic intersections is a local move. In particular, this allows us to perform these local Hamiltonian isotopies in a neighborhood of any $\gamma_i$ in the same $G$-orbit in order to produce a $G$-invariant Hamiltonian isotopy. The result of mutating at each of the cycles in the $G$-orbit then produces another $G$-filling.  
\end{proof}

\begin{remark}
    Note that while mutation at every cycle in a $G$-orbit produces another $G$-filling, it does not necessarily preserve $G$-admissibility, which is necessary if we wish to perform additional mutations. As explained in Subsection \ref{sub: folding}, $G$-admissibility is only
    preserved under an arbitrary sequence of mutations of $G$-orbits if the cluster algebra is globally foldable with respect to the $G$-action.
\end{remark}

\subsubsection{Folding $\FM(\La(\beta)$}

We now proceed with a description of cluster ensembles given by $G$-invariant sheaf moduli.

\begin{theorem}\label{thm: folded cluster}
Suppose that $\C[\FM(\La(\beta))]$ is a globally foldable cluster algebra with respect to a group $G$ acting by exact (anti)-symplectomorphisms on $\Symp(J^1\S^1, \xi_{\st})$ and assume that there is an initial $G$-filling $L$ with a $G$-admissible intersection quiver. Then the pair of moduli, $\FM(\La(\beta))^G$ and $\SM_1(\La(\beta))^G$ form a cluster ensemble with every seed induced by a $G$-filling.   
\end{theorem}

\begin{proof}
    Let $\La(\beta)$ be a Legendrian link such that the ring of regular functions $\C[\FM(\La(\beta), T)]$ is a globally foldable cluster algebra. Following the formulation of the character lattices above, the cluster $\SA$-coordinates of the cluster tori corresponding to the lattice $N(\La)$ are given by $G$-invariant microlocal merodromies. Indeed, by the $G$-invariance of $\FM(\La)^G$, the flags and the framing data involved in defining the microlocal merodromy of any cycles belonging to the same $G$-orbit are necessarily equal. As a result, taking the $G$-invariant part of $\FM(\La)$ induces the surjective semifield homomorphism appearing in the definition of folding given in Subsection \ref{sub: folding}. An analogous argument holds for the cluster $\SX$-coordinates coming from $G$-equivariant microlocal monodromies.
The result is a cluster ensemble structure on the pair $\FM(\La)^G$ and $\SM_1(\La)^G$.

To show that every cluster chart of $\FM(\La)^G$ is induced by a $G$-filling, we note that the globally foldable condition implies that mutation commutes with folding. Therefore, we can apply Lemma \ref{lemma: simultaneously mutable} to realize the cluster mutations of the folded cluster algebra as Lagrangian disk surgeries at each element of a $G$-orbit of $\mathbb{L}$-compressing cycles of a $G$-filling. The lemma also implies that the result of performing these Lagrangian disk surgeries at each cycle in the $G$-orbit is again a $G$-filling. Thus, every cluster chart of the folded cluster algebra $\FM(\La)^G$ is induced by some $G$-filling.   
\end{proof}

\begin{remark}
    In forthcoming joint work with Agniva Roy, we give another contact-geometric interpretation of skew-symmetrizable cluster ensembles as the sheaf moduli of particular twist-spun Legendrian tori in $(\R^5, \xi_{\st})$ (see \cite{EkholmKalman} for the initial construction). In other words, studying the mapping torus of the Legendrian loop realizing the above group actions yields another contact-geometric interpretation of exact Lagrangian $G$-fillings of Legendrian links equipped with $G$-actions. 
\end{remark}

We conclude this subsection by establishing the existence of skew-symmetrizable cluster structures on $\FM(\La)^G$ for every $(\La, G)\in \mathcal{H}^G$. See Table \ref{tab: folded} for a compilation of the specific Legendrian links and $G$-actions. 

\begin{proposition}
For every $(\La, G)\in \mathcal{H}^G$, the ring of regular functions $\C[\FM(\La)^G]$ of the $G$-invariant sheaf moduli of $\La$ is a skew-symmetrizable cluster algebra.    
\end{proposition}

\begin{proof}
As discussed in Subsection \ref{sub: folding}, each quiver and $G$-action corresponding to a $(\La, G)$ pair in $\mathcal{H}^G$ is globally foldable. Moreover, \cite[Proposition 5.7]{ABL22} establishes the existence of initial $G$-fillings for all $(\La, G)\in \mathcal{H}^G$ of finite and affine type. Initial $G$-fillings for $(\La(E_7^{(1,1)}), \Z_2)$ and $(\La(E_8^{(1,1)}), \Z_3)$ are given in Figure \ref{fig: SymmetricT36T44}. These fillings are obtained from reduced plabic graphs using a construction given in \cite{CLSBW23}. By Theorem \ref{thm: folded cluster}, we have that $\C[\FM(\La)^G]$ is a skew-symmetrizable cluster algebra for every $(\La, G)\in \mathcal{H}^G$. 
\end{proof}

\subsection{Generators of folded cluster modular groups}

\begin{figure}
    \centering
\includegraphics[width=\textwidth]{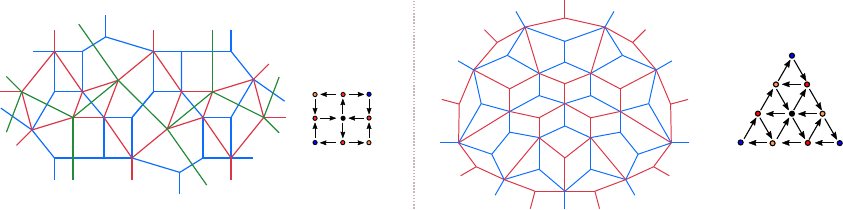}
    \caption{Initial weave fillings and intersection quivers for $\La(E_7^{(1, 1)})$ (left) and $\La(E_7^{(1, 1)})$ (right) exhibiting the required rotational symmetry}
    \label{fig: SymmetricT36T44}
\end{figure}

In this subsection, we describe generators of the cluster modular groups $\SG(\FM(\La)^G)$ as $G$-equivariant Legendrian loop actions for Legendrian links and associated $G$-actions belonging to the set $\mathcal{H}^G$. We first note that for any pair $(\La, G)$ satisfying the hypotheses of Theorem \ref{thm: folded cluster} and admitting a Legendrian $G$-loop $\varphi$, the induced action of $\varphi$ on $\SG(\FM(\La))$ descends to a cluster automorphism on $\SG(\FM(\La))^G$. This follows from the fact that such a $G$-loop necessarily preserves $G$-invariance of Lagrangian fillings of $\La$ and hence the induced action is $G$-equivariant. Using this characterization, we give a case-by-case description of the Legendrian loops that descend to generators of $\SG(\FM(\La))^G$.

\begin{proof}[Proof of Theorem \ref{thm: folded gens}]
We consider the braids appearing in Table \ref{tab: folded} and the indicated group action. Note that the choice of initial braid word appearing in Table \ref{tab: folded} is adopted from \cite{ABL22} and differs from the choice of initial braid word given in the description of the Legendrian links in $\mathcal{H}$ from the introduction in several cases. The reason for this difference is that the braid words from Table \ref{tab: folded} more readily display the symmetry required for constructing $G$-fillings out of Legendrian weaves. To relate Legendrian loops of Legendrian isotopic links $\La_1$ and $\La_2$, one can fix a specific Legendrian isotopy $\psi$ taking $\La_1$ to $\La_2$ and conjugate a Legendrian loop $\varphi$ by $\psi$. The result $\psi^{-1}\circ\varphi\circ \psi$ is a Legendrian loop of $\La_2$ and induces the same action on fillings.

{\bf Finite type:}

\begin{itemize}
    \item[$B_n$:] 
    As stated in the proof of Theorem \ref{thm: gens}, the K\'alm\'an loop $\rho$ generates the cluster modular group $\mathcal{G}(\FM(\La(A_{2n-1})))$ and acts as rotation by $\frac{2\pi}{2n+2}$ on any weave filling of $\La(A_{2n-1})$. Therefore the group action is generated by $\rho^{n+1}$, which implies that $\rho^k$ is a $G$-loop for any $k\in \Z$.

\item[$C_n$:] In order to produce a $\Z_2$ symmetry of $D_{n+1}$, An, Bae, and Lee use the braid given in Table \ref{tab: folded} obtained as a stabilization of the link $\La(D_{n+1})$ we originally defined. Their 4-stranded braid is fixed by $\iota$ and admits a $\vartheta$-loop involving the middle pair of strands that is also fixed by $\iota$. This $\vartheta$-loop is conjugate to the $\vartheta$-loop defined in the proof of Theorem \ref{thm: gens} and therefore generates $\SG(\FM(\La(D_{n+1}))^\iota)$. 

\item[$F_4$:] In the case of $\La(E_6)$, the DT transformation generates the cluster modular group $\SG(\FM(\La(E_6)))$. Since $\La(E_6)$ is fixed under $\iota$, the Legendrian isotopy given as part of the definition of $\DT$ can be chosen so that its trace is an $\iota$-fixed exact Lagrangian cobordism.

     \item[$G_2$:] For $\La(D_4)\cong \La(3, 3)$, the loop $\Sigma_1$ defined in Subsection \ref{sub: torus loops} commutes with $\rho^2$ and is of order 4, yielding a generator of $\SG(\FM(\La(D_4)))^{\rho^2}\cong \Z_4$.
\end{itemize}

{\bf Affine type:}

\begin{itemize}

\item[$\tilde{B}_n$:] Consider the symmetric front for $\La(\tilde{D}_{2n})$ given by the braid word in Table \ref{tab: folded}. From \cite[Theorem 5.2]{KaufmanGreenberg}, we know that the cluster modular group of $\tilde{B}_n$ is generated by the induced action of $\tau_1\in \SG(\tilde{D}_{2n})$  and $\DT$. Denote the Legendrian loop from the proof of Theorem \ref{thm: gens} that induces $\tau_1$ by $\vartheta_1$. Conjugating by the Legendrian isotopy between different choices of braid representatives of $\La(\tilde{D}_{2n})$ allows us to obtain $\vartheta_1$ as a rotation-invariant loop of $\La(\tilde{D}_{2n})$ that induces the required automorphism on $\FM(\La(\tilde{D}_{2n}))^G$.

\item[$\tilde{F}_4$:] The cluster modular group of $\tilde{F}_4$ is generated by the DT transformation \cite[Theorem 5.2]{KaufmanGreenberg}. As with folding $E_6$ by $\iota$, the Legendrian isotopy as part of the $\DT$ transformation can be chosen to be fixed by $\iota$. 

\item[$\tilde{G}_2$:] The case of $\tilde{G}_2$ is analogous to that of $\tilde{F}_4$, replacing the action of $\iota$ by $2\pi/3$ rotation.

\end{itemize}

{\bf Extended affine type:}

\begin{itemize}

\item[$G_2^{(1,1)}$:] We consider the front for $\La(E_8^{(1, 1)})\cong \La(3, 6)$ given by the braid word $(\sigma_1\sigma_2)^9$ and fold by action of the Legendrian loop $\rho^3=\DT^{-2}$. The $\Sigma_i$ loops from Subsection \ref{sub: torus loops} are defined on the first $(k)(k-1)$ crossings for arbitrary torus links $\La(k, n)$ and extended by symmetry, so are necessarily fixed by $\rho^k$. Hence the $\Sigma_i$ loop generators of $\SG(\FM(\La(3,6)))$ together with the DT transformation yield a generating set for $\SG(\FM(\La(3, 6)))^{\rho^3}$.

\item[$B_3^{(1, 1)}$:] The case of folding $\La(E_7^{(1, 1)})\cong \La(4, 4)$ is analogous to the previous case, replacing $(3, 6)$ with $(4, 4)$ as necessary. 

\end{itemize}

\end{proof}

\begin{table}[]
\begin{tabular}{l|l|l|l|l}
Cluster Type & Braid word $(-1)$-closure & Folded type  & CMG (of folded) & $G$-action  \\
\hline
$A_{2n-1}$  & $\sigma_1^{2n+2}$ & $B_n$  &  $\Z_{n+1}$ & $\Z_2$ $(=\rho^{n+1})$  \\

$D_{n+1}$ $n\geq 4$ & $\sigma_2^{n-1}\sigma_1\sigma_3\sigma_2(\sigma_1\sigma_3)^3\sigma_2\sigma_1\sigma_3\sigma_2^2\sigma_1\sigma_3\sigma_2$ &  $C_n$ &  $\Z_{n+1}$  &  involution
\\

$E_6$  & $\sigma_2^2\sigma_1\sigma_3\sigma_2(\sigma_1\sigma_3)^4\sigma_2\sigma_1\sigma_3\sigma_2^2\sigma_1\sigma_3\sigma_2$ &$F_4$   & $\Z_7$   
& involution  \\

$D_{4}$ &   
$(\sigma_2\sigma_1)^6$
&  $G_2$ &  $\Z_4$  & $\Z_3$
\\

$\tilde{D}_{2n}$ & $(\sigma_2\sigma_1^3\sigma_2\sigma_1^3\sigma_2\sigma_3\sigma_1^{n-2})^2$ 
& $\tilde{B}_n$  & $\Z_2\times \Z$ & $\Z_2$
\\

$\tilde{E}_6$ & $\sigma_2^3\sigma_1\sigma_3\sigma_2(\sigma_1\sigma_3)^4\sigma_2\sigma_1\sigma_3\sigma_2^2\sigma_1\sigma_3\sigma_2$ &  $\tilde{F}_4$ &  $\Z$  & involution
\\  

$\tilde{E}_6$ & $(\sigma_2\sigma_1^4)^3$ & $\tilde{G}_2$ &  $\Z$  & $\Z_3$
\\

$E_7^{(1,1)}$ & $(\sigma_1\sigma_2\sigma_3)^4$ &  $C_3^{(2, 2)}$ &  $\Mod(\S^2, 4)\rtimes \Z_2$  & $\Z_2$
\\ 

$E_8^{(1,1)}$ & $(\sigma_1\sigma_2)^9$ &  $G_2^{(1, 1)}$ &  $\PSL_2(\Z)\rtimes \Z_2$  & $\Z_3$

\end{tabular}
\caption{\label{tab: folded}} 
\end{table}

\begin{remark}
 For Legendrian links that do not admit $G$-fillings, one can still study $\FM(\La)^G$ for $G$-actions induced by Legendrian loops. In the case of Legendrian torus links acted on by powers of the K\'alm\'an loop, this appears to correspond to the cyclic symmetry loci of Grassmannians considered in \cite{fraser2020cyclic}. In loc.~cit. Fraser considers generalized cluster structures on $G$-invariant components of the affine cone of the Grassmannian, which suggests that a similar structure could be realized in the contact-geometric setting, perhaps by considering singular exact Lagrangian $G$-fillings of Legendrian torus links. 
\end{remark}

We complete our discussion related to cluster modular groups of skew-symmetrizable cluster varieties with a final observation on the anti-symplectic involution. As mentioned above, a result of Kaufman \cite[Lemma 2.8]{kaufman2023special} implies that any cluster algebra with an involutory cluster automorphism is globally foldable with respect to this involution. This observation allows us to generalize the case of the folding $\tilde{E}_6\to \tilde{F}_4$ to any Legendrian link $\La$ fixed by anti-symplectic involution $\iota$. Indeed, for any Legendrian $\La=\La(\mathbb{G})$ arising as the link of a complete grid plabic graph fixed by $\iota$, $\La$ admits an $\iota$-invariant initial weave filling by modifying the construction of \cite{CasalsWeng} following the notion of `degenerate' $N$-graphs given in \cite[Section 3.1]{ABL22}. For such Legendrians, any pair of Legendrian loops interchanged by $\iota$ induces a cluster automorphism of the folded cluster algebra $\FM(\La, T)^\iota$.

	\section{Nielsen-Thurston classification of Legendrian loops}\label{sec: NielsenThurston}

In this section, we consider qualitative properties of Legendrian loops by investigating connections between cluster modular groups and mapping class groups. We start by describing a Nielsen-Thurston classification of cluster automorphisms due to Ishibashi and then use this framework to give some general statements about fixed point properties of Legendrian loop actions.    

We introduce the following characterization of cluster automorphisms as our definition of a Nielsen-Thurston classification for cluster automorphisms. Our definition differs from Ishibashi's original definition in order to give a more streamlined description in the context of Legendrian loops. See \cite{Ishibashi2018} for alternate characterizations in terms of fixed points of automorphisms acting on the cluster complex.

\begin{definition}
    A cluster automorphism $\varphi\in \SG(\FM(\La))$ is 
    \begin{enumerate}
        \item periodic if it is of finite order.
        \item cluster reducible if it fixes a set of cluster variables.
        \item cluster pseudo-Anosov if no power of $\varphi$ is periodic or cluster reducible.
    \end{enumerate}
\end{definition}

Note that the cyclic subgroup generated by any cluster automorphism will correspond to at least one of these classes. We say that a Legendrian loop is (cluster) periodic, reducible, or pseudo-Anosov if its induced cluster automorphism is of the corresponding type. Below, we provide examples of periodic and reducible Legendrian loops. 

\begin{ex}
Consider the Legendrian torus link $\La(k,n)$ with $2\leq k, n$. As discussed in Section \ref{sec: Torus Links}, $\FM(\La(k, n))$ is known to have the same mutable part as $\Gr^\circ(k, n+k)$, the top-dimensional positroid cell of the Grassmannian, which itself admits a cluster structure. The K\'alm\'an loop induces the cyclic shift $\rho$, which acts on column vectors in the matrix representation of the top-dimensional positroid cell by $v_i \mapsto v_{i-1}$. By construction, $\rho$ has order $k+n$, which implies that the K\'alm\'an loop is always periodic.  
\end{ex}

\begin{ex}\label{ex: reducible}
  Consider the Legendrian $\vartheta$-loop associated to the link $\La(\sigma_1\sigma_2\sigma_2\sigma_1\Delta^2, 1, \sigma_1^{n-4})\cong \La(\tilde{D}_n)$. From Lemma \ref{lemma: satellite loop mutations}, we see that the induced action of the $\vartheta$-loop fixes a single quiver vertex, giving a cluster reducible automorphism.
\end{ex}

We will demonstrate below that many classes of $\vartheta$ loops are cluster reducible. In contrast, producing an example of a pseudo-Anosov Legendrian loop appears to be more challenging. The following example describes a composition of Legendrian loops that conjecturally induces a cluster pseudo-Anosov automorphism.

\begin{ex}\label{ex: pseudo-Anosov} 
Consider the Legendrian link $\La(\D^2_{n, 0, 0, 0}):=\La(\sigma_2\sigma_1\sigma_3\sigma_2\sigma_4\sigma_3^2\sigma_4\sigma_2\sigma_3\sigma_1\sigma_2\sigma_1^{n-4})$ with $n\geq 4$ formed by adding three unlinked meridians to $\La(2, n-4)$. As suggested by the notation, the initial quiver associated to $\La(\D^2_{n, 0, 0, 0})$ corresponds to a thrice-punctured disk with $n$ boundary marked points. Analogous to the case of $\La(\tilde{D_n})$, the link $\La(\D^2_{n, 0, 0, 0})$ admits Legendrian loops generating the mapping class group of $\D^2_{n, 0,0,0}$. 
    The cluster modular group $\SG(\SA(\D^2, 0\dots, 0))$ of the cluster $\SA$-space $\SA(\D^2, 0, \dots, 0)$ associated to the $n$-punctured disk has a subgroup isomorphic to the $n$-stranded braid group. When $n\geq 2$, the mapping class $\sigma_i\sigma_{i+1}^{-1}\in \Mod(\D^2, 0, \dots, 0)$ is psuedo-Anosov, hence the corresponding cluster automorphism is cluster pseudo-Anosov. We conjecture that a computation similar to that given in Appendix \ref{sec: appendix} realizes this cluster pseudo-Anosov automorphism as the composition of two Legendrian loops.
\end{ex}

\subsection{Fixed points}

Extending the analogy between mapping class groups and cluster modular groups, we study the fixed points of Legendrian loop actions on $\FM(\La).$ Define the positive real part $\FM(\La)_{>0}$ of the cluster $\mathcal{A}$-space $\FM(\La)$ to be the space given by requiring all of the cluster variables of the initial seed to be strictly positive real numbers. As positivity of cluster variables is preserved under cluster mutation, this definition gives a well-defined notion of positivity in the cluster algebra. 
The cluster modular group acts on $\FM(\La)_{>0}$ by permuting the positive real cluster tori. Interpreting \cite[Theorem 2.2]{Ishibashi2018}  
in the context of contact geometry immediately yields Theorem \ref{thm: fixed points}, restated here for clarity. 

\begin{theorem}[Theorem \ref{thm: fixed points}]
    The induced action of any periodic Legendrian loop $\varphi$ of a Legendrian link $\La$ has a fixed point in $\FM(\La)_{>0}$. 
\end{theorem}

We can also use Theorem \ref{thm: fixed points} to detect Legendrian loops of infinite order. The following description mimics a definition of Casals and Ng given for Legendrian loops in the augmentation variety.

\begin{definition}
    The induced action $\tilde{\varphi}$ of a Legendrian loop is entire on a toric chart $\mathcal{C}_L=(\C^\times)^n$ induced by an exact Lagrangian filling $L$ of $\La$ if for $k\neq l\in \Z$ we have $\varphi^k(\mathcal{C}_L)\neq \varphi^l(\mathcal{C}_L).$
\end{definition}

As a direct consequence of Theorem \ref{thm: fixed points}, we obtain the following corollary:

\begin{corollary}
    The induced action of a Legendrian loop on $\FM(\La)$ is entire on any cluster chart induced by an exact Lagrangian filling if it has no fixed points in $\FM(\La)_{>0}$.  
\end{corollary}

\begin{proof}
    Let $\varphi$ be a Legendrian loop of $\La$ with no fixed points in $\FM(\La)_{>0}.$ Since $\varphi$ has no fixed points, this in particular implies that no power of it can fix a cluster chart. Therefore, $\varphi^k(\mathcal{C}_L)\neq \varphi^l(\mathcal{C}_L)$ for any $k \neq l \in \Z$.  
\end{proof}

We now present examples of the fixed point behavior of Legendrian loops of $\La(A_n)$ and $\La(E_{8}^{(1,1)})$.

\begin{ex}
Consider the initial seed $x_1 \leftarrow x_2$ in $\FM(\La(A_2))$. The cluster automorphism $\tilde{\rho}$ induced by the K\'alm\'an loop has a single fixed point $x_1 = x_2 = \frac{1+\sqrt{5}}{2}$. In the case of $A_n$-type, Ishibashi's work implies that the cluster modular group action is properly discontinuous on $\FM(\La(A_n))_{>0}$ \cite[Theorem 3.8]{Ishibashi2018}. Therefore, the existence of this fixed point is equivalent to the fact that $\rho$ has finite order.
\end{ex}

\begin{ex}\label{ex: T36FixedPoints}
    Now consider the initial seed of $\La(3, 6)$ corresponding to the front $\La(\beta)$ for $\beta=(\sigma_1\sigma_2)^6$.   
    The loop $\Sigma_1$ defined in \cite{CasalsGao} and described in Section \ref{sec: Torus Links} has no fixed points in $\FM(\La(3, 6))_{>0}$, recovering the fact that $\sigma_1$ (conjugate to $\vartheta_1$ by \cite[Theorem 6.1]{KaufmanGreenberg}) has infinite order and implying that it is entire on any seed. This is verified by showing that there are no positive real solutions to the system of equations below obtained from performing the Legendrian loop and setting the corresponding cluster variables equal to each other. 
    
    \begin{align*}
a_1&= \frac{a_2 + a_3 + a_1a_4}{a_2}, \qquad  a_2= a_4, \qquad  a_3= \frac{a_2 + a_3}{a_1},
 \qquad a_4= \frac{a_3a_6 + a_4a_7}{a_5},\\
 a_5&= \frac{a_1a_3a_6 + (a_1a_4 + (a_2 + a_3)a_5)a_7}{a_1a_3a_5}, \qquad a_6= a_6,\\
a_7&= \frac{a_1a_3a_5a_6a_8 + (a_1a_3a_6^2 + ((a_1a_4 + (a_2 + a_3)a_5)a_6)a_7)a_9 + (a_1a_3a_6a_7 + (a_1a_4 + (a_2 + a_3)a_5)a_7^2)a_{10}}{a_1a_3a_5a_7a_8},\\
a_8&= a_{10}, \qquad a_9= \frac{a_1a_3a_5a_8 + (a_1a_3a_6 + (a_1a_4 + (a_2 + a_3)a_5)a_7)a_9}{a_1a_3a_5a_7}, \qquad a_{10}= a_9\\
    \end{align*}

Inputting the above system of equations into a computer algebra system verifies that no positive real solution exists. 
\end{ex}

\subsection{Cluster Reduction}

In this section, we discuss the process of cluster reduction, analogous to the concept of a reduction system in the theory of mapping class groups. From a reducible mapping class $\varphi$, one can obtain a mapping class on a simpler surface by cutting along curves fixed by $\varphi$. The analogous process in cluster theory allows us to gain additional information about reducible Legendrian loops.

We define a cluster reducible automorphism to be proper reducible if in addition to fixing some collection of cluster variables setwise, it also fixes at least one cluster variable. Note that a high enough power of a cluster reducible automorphism necessarily yields a proper reducible automorphism. Ishibashi defines the process of cluster reduction as follows: given a cluster reducible cluster automorphism $\varphi$, one freezes quiver vertices corresponding to cluster variables fixed by some power of $\varphi$. This induces a cluster automorphism of a cluster algebra with a smaller mutable part. The induced automorphism is referred to as the cluster reduction of $\varphi$.

\begin{ex}\label{ex: cluster reduction}
    For $\La(\tilde{D}_n)$, the Legendrian loop $\vartheta_1$ fixes a single cluster variable. The cluster reduction of this quiver produces a new quiver with the fixed cluster variable either frozen or deleted. Upon inspection, the mutable part of the quiver obtained by cluster reduction corresponds to a (tagged) triangulation of an annulus with $n$ marked points on the outer boundary and two marked points on the inner boundary. Note that this triangulation, pictured in Figure \ref{fig: ReductionAffineDn}, can be readily obtained from the triangulation of the twice punctured disk we started with by replacing the edge corresponding to the newly frozen cluster variable with an additional boundary component. The induced cluster automorphism is a $2\pi/n$ rotation of the outer boundary component. 
\end{ex}

	\begin{center}
		\begin{figure}[h!]{ \includegraphics[width=.3\textwidth]{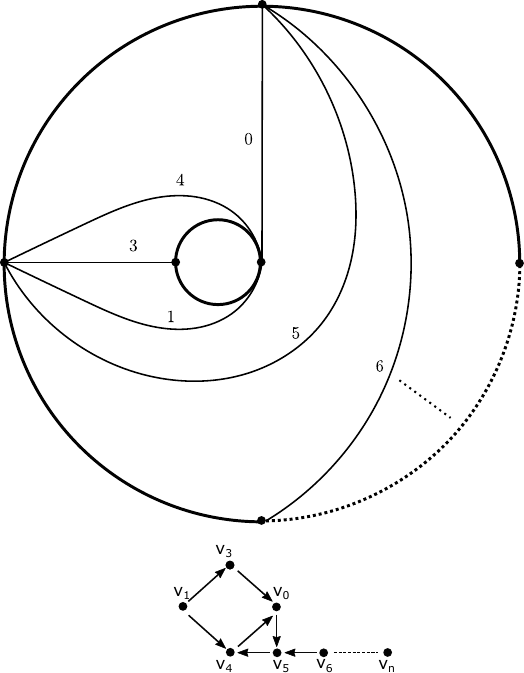}}\caption{Triangulation of $\D^2_{n_1-2, 2}$ corresponding to the cluster reduction of $\vartheta_1$. Replacing the inner boundary component by a single edge recovers the triangulation of the twice-punctured disk in Figure \ref{fig:TwicePuncturedTaggedTriangulations} (right).}
			\label{fig: ReductionAffineDn}\end{figure}
	\end{center}

\begin{remark}
    Note that the process of cluster reduction presented here is entirely algebraic. 
    As a contact-geometric analogue, one can consider modifying the ambient symplectic manifold of the Lagrangian filling so that the cycle $\gamma$ no longer bounds an embedded Lagrangian disk with which to perform a mutation. 
    Such a modification might plausibly 
    be obtained by removing the Lagrangian 2-disk bounding $\gamma$ from the Lagrangian skeleton constructed in \cite[Section 1.1]{CasalsLagSkel}. 
\end{remark}

Example \ref{ex: cluster reduction} motivates the notion of a cluster Dehn twist, which we introduce below.
Denote by $Q_i$ the quiver with two mutable vertices and $k$ edges from vertex $v_1$ to vertex $v_2$. The quiver $Q_k$ admits a cluster automorphism $T_k=(\mu_1, (12)).$  We have the following definition due to Ishibashi \cite{Ishibashi2020}.

\
\begin{definition}\label{def: cluster Dehn twist}
    A cluster Dehn twist is a cluster automorphism $\varphi$ such that after a finite number of cluster-reductions, the induced automorphism $\tilde{\varphi}$ satisfies $\tilde{\varphi}^n=T_k^m$ for some nonzero integers $m, n$ and $k\geq 2$. 
\end{definition}

 We say that a Legendrian loop is a cluster Dehn twist if its induced action is. Note that a cluster Dehn twist is necessarily of infinite order because the induced cluster automorphism is of infinite order. Any Dehn twist (or fractional twist) of a tagged triangulation is a cluster Dehn twist \cite{Ishibashi2020}. Indeed, the quiver $Q_2$ corresponds to an annulus with one marked point on each boundary component and $T_2$ corresponds to a Dehn twist in this annulus. Example \ref{ex: cluster reduction} serves as an example of a Legendrian loop that is a cluster Dehn twist.

\subsubsection{Legendrian satellite loops}

In order to prove Theorem \ref{thm: fixed pt converse} and related corollaries, we will apply the process of cluster reduction to $\vartheta$-loops. As in the case of Example \ref{ex: cluster reduction}, the induced cluster automorphisms can be understood as actions on surface-type cluster algebras. The cluster modular groups of these algebras are relatively well understood and exhibit certain crucial dynamical properties. 

We define a group action of a group $G$ on a topological space $X$ to be properly discontinuous if for every compact subset $K\subseteq X$, the set $\{g\in G | gK\cap K \neq \emptyset\}$ is finite. Ishibashi shows that if the cluster modular group action on the cluster $\mathcal{A}$-space or cluster $\mathcal{X}$-space is properly discontinuous, then the implication of Theorem \ref{thm: fixed points} can be upgraded to an equivalence. More precisely, in such a case a cluster automorphism has a fixed point in the positive real part of a cluster $\SA$-space or cluster $\mathcal{X}$-pace if and only if it is of finite order. This property holds in the case of surface-type cluster algebras \cite[Theorem 3.8]{Ishibashi2018}, which partially motivates the following technical lemma:

\begin{lemma}\label{lemma: reduction to surface}
Let $\La=\La(\beta\Delta^2, i, \sigma_1^k)$ be a Legendrian satellite link 
The cluster reduction of the induced $\vartheta$-loop action is a tagged mapping class on a surface-type cluster algebra.    
\end{lemma}

The technique for the following proof generalizes an argument of Fraser in the case of cluster modular groups of Grassmannians.\footnote{Chris Fraser, Personal Communication, 2/22/22.}

\begin{proof}
Let $\La(\beta\Delta_n^2, i, \sigma_1^k)\cong\La(\beta')$ be a Legendrian link formed by satelliting $\sigma_1^k$ about the $i$th strand of the $(-1)$-framed closure of $\beta\Delta_n^2$ and assume $k\geq 3$. Recall from Section \ref{sub: Mutation sequences} that $Q_\gamma$ and $Q_i$ denote the subset of quiver vertices corresponding to the satellited braid $\gamma=\sigma_1^k$ and the $i$th strand of $\beta'\Delta_{n+1}^2$, respectively. By Lemma \ref{lemma: satellite loop mutations}, the fixed vertices are precisely those not in $Q_\gamma$ or $Q_i$. Denote the quiver of the remaining non-fixed vertices by $Q_{nf}$, i.e. the quiver obtained by deleting the vertices not in either $Q_\gamma$ or $Q_i$ and all adjacent arrows. 

To show that $Q_{nf}$ is a surface-type quiver, consider the vertices for which the sets $Q_\gamma$ and $Q_i$ intersect. Between every two vertices $v_{i_j}, v_{i_{j+1}}$, the quiver $Q_{nf}$ is described as follows: for vertices in $Q_{\gamma}$, we have a directed path with arrows oriented from left to right; for vertices in $Q_i$ we have a directed path from right to left. Together, these two oriented paths form an oriented cycle. The only remaining arrows are from $v_{i_{j}\pm 1}$ to the adjacent vertices in $Q_\gamma$, forming additional oriented 3-cycles. This pattern holds in general except for at the leftmost and rightmost section of the quiver $Q_{nf}$. The leftmost section corresponds to an unoriented cycle, as the face of the plabic fence corresponding to the leftmost vertex of $Q_\gamma$ is to the left of $v_{i_1}$. The rightmost section of $Q_{nf}$ corresponds to a directed path of $k-3$ vertices corresponding to the satellited braid $\sigma_1^k$.

The crucial observation is that each section of $Q_{nf}$ corresponds to a piece of a triangulated disk with possible punctures and boundary components with marked points. In particular, every oriented $\ell$-cycle appearing between two vertices $v_{i_j}$ and $v_{i_{j+1}}$ corresponds to an additional boundary component with $\ell-4$ of marked points on the boundary, or in the case of $\ell=4$, a single puncture. These pieces of the triangulation, depicted in Figure \ref{fig: ClusterReductionTriangulation}, are then glued together to form the desired surface.

\begin{figure}
    \centering
\includegraphics[width=.9\textwidth]{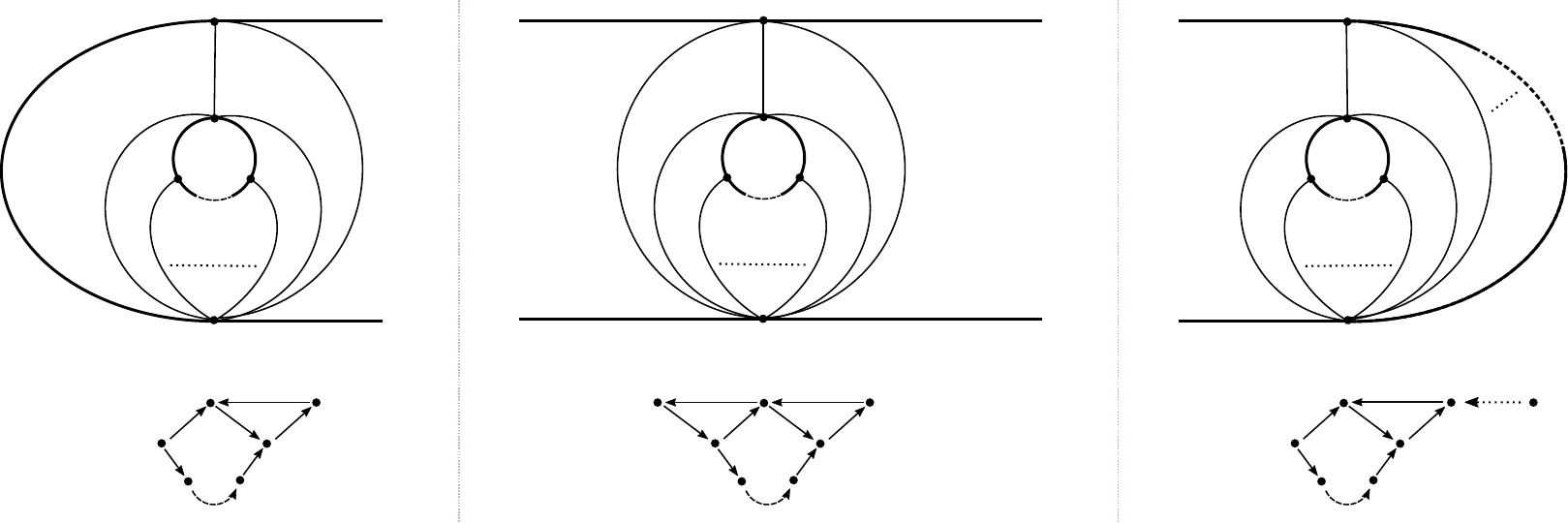}
    \caption{Pieces of a triangulation corresponding to the non-fixed portion of the intersection quiver of an initial filling of $\La(\beta\Delta^2, i, \sigma_1^k)$.}
    \label{fig: ClusterReductionTriangulation}
\end{figure}

Following Lemma \ref{lemma: satellite loop mutations}, we can compute the mutation sequence induced by $\vartheta$-loop. In the case of an individual $\ell$-cycle in $Q_{nf}$, this is readily seen to induce a partial Dehn twist about the corresponding boundary component.
\end{proof}

\begin{remark}
    When $k=2$, the satellited component of the link $\La(\beta\Delta_n^2, i, \sigma_1^2)$ is a two-stranded unlink, analogous to the case of $\tilde{D}_n$. Given that the combinatorics of plabic fences does not appear to capture the associated $\vartheta$ loop, computing the non-fixed quiver for such a loop is much more challenging. Nevertheless, one might expect to obtain a surface-type quiver in this case as well. 
\end{remark}

Lemma \ref{lemma: reduction to surface} provides a relatively straightforward combinatorial algorithm for evaluating the order of the induced action of a Legendrian $\vartheta$-loop, at least in the setting of rainbow closures of positive braids. Namely, the proof of Lemma \ref{lemma: reduction to surface} describes a process for obtaining a mapping class of a surface with boundary marked points and interior punctures from a Legendrian $\vartheta$-loop. The order of this mapping class is precisely the order of the induced action of the $\vartheta$-loop. This action has finite order if and only if the surface-type quiver obtained via cluster reduction corresponds to a disk with at most one interior puncture, i.e. the cluster algebra is of type $A_n$ or $D_n$. 

Equipped with Lemma \ref{lemma: reduction to surface}, we now prove Theorem \ref{thm: fixed pt converse} and the resulting corollaries.

\begin{proof}[Proof of Theorem \ref{thm: fixed pt converse}]
    Suppose that $\La=\La(\beta\Delta^2, i, \sigma_1^k)$ admits a $\vartheta$ loop whose induced action on $SM_1(\La)$ has infinite order. By Lemma \ref{lemma: reduction to surface}, the induced action of $\vartheta$ on the cluster reduction $\mathcal{A}'_{>0}$ is an infinite order tagged mapping class. Therefore, \cite[Theorem 3.8]{Ishibashi2018} implies that the induced action on $\mathcal{A}'_{>0}$ 
     is properly discontinuous. Since $\mathcal{A}'$ was obtained from $\FM(\La)$ by freezing the fixed cluster variables, it follows that the action of $\vartheta$ on $\FM(\La)_{>0}$ is properly discontinuous. In particular, $\vartheta$ has no fixed points in $\FM(\La)_{>0}$.
\end{proof}

As a result of Lemma \ref{lemma: reduction to surface}, we obtain an additional corollary, providing further similarities between cluster modular groups and mapping class groups. See also \cite[Corollary 6.5]{KaufmanGreenberg}. 

\begin{corollary}
    For any $\La\in \mathcal{H}$, the group $\SG(\FM(\La))$ contains a finite-index subgroup generated by cluster Dehn twists. 
\end{corollary}

\begin{proof}
    In the proof of Theorem \ref{thm: gens} in Section \ref{sec: CMGs}, we describe a generating set of $\SG(\FM(\La))$ where all of the Legendrian loops inducing infinite order generators are $\vartheta$ loops and together they generate a finite-index subgroup of $\SG(\FM(\La))$. Applying Lemma \ref{lemma: reduction to surface} immediately implies the desired result.
\end{proof}

As a final application of Lemma \ref{lemma: reduction to surface}, we show that in the case of the Legendrian links of affine and extended affine type considered above, the action of the entire cluster modular group $\SG(\FM_1(\La), T)$ is properly discontinuous.

\begin{corollary}\label{cor: Properly Discontinuous}
For every $\La\in\mathcal{H}$, the action of $\SG(\FM(\La, T))$ on $\FM(\La, T)_{>0}$ is properly discontinuous.
\end{corollary}

\begin{proof}
The statement is true for finite Dynkin types as these necessarily have finite cluster modular groups. 
For the remaining affine type and extended affine type cluster modular groups, the generating sets of $\SG(\FM(\La))$ given in the proof of Theorem \ref{thm: gens} above all contain at least one infinite order cluster automorphism that generates a finite index subgroup of $\SG(\FM(\La))$ and is generated by a Legendrian $\vartheta$ loop. Combining Lemma \ref{lemma: reduction to surface} with \cite[Theorem 3.8]{Ishibashi2018} then yields the desired result. 
\end{proof}

Note that the statement and method of proof of Corollary \ref{cor: Properly Discontinuous} holds for any Legendrian link $\La$ satisfying the property that $\SG(\FM(\La, T))$ has a finite index subgroup generated by a $\vartheta$ loop. However, for Legendrian links admitting multiple $\vartheta$ loops this property does not appear to be satisfied in general. This suggests that an alternative approach may be required to establish that the action of $\SG(\FM(\La, T))$ on $\FM(\La, T)_{>0}$ is properly discontinuous.

\bibliographystyle{alpha}
\bibliography{main.bib}

\appendix

\section{Mutation sequence computation via Legendrian weaves}\label{sec: appendix}
In this appendix, we show that the Legendrian loop $\vartheta_2$ defined for $\La(\tilde{D}_n)$ induces the cluster automorphism $\tilde{\vartheta}_2=(1, 4; (1\,4\,0\,3))$, computed as a sequence of mutations in the initial quiver coming from the plabic fence $\mathbb{G}(\tilde{D}_n)$ pictured in Figure \ref{fig: DnPlabicFence}. We do this by mutating at cycles $\gamma_1$ and $\gamma_4$ in the initial weave $\w(\mathbb{G}(\tilde{D}_6))$ -- constructed following \cite[Section 3.3]{CasalsWeng} -- and then showing that the resulting weave simplifies to the concatenation of the trace of $\vartheta_2$ with the initial weave, up to relabeling the homology cycles. The example readily generalizes to $\La(\tilde{D}_n)$ by replacing the dashed blue short {\sf I}-cycle by $n-5$ blue short {\sf I}-cycles.

In the figures below, $\mathbb{L}$-compressing cycles are color coded as follows: $\gamma_0$ is light blue, $\gamma_1$ is orange, $\gamma_2$ is light green, $\gamma_3$ is purple, $\gamma_4$ is pink, $\gamma_5$ is yellow, and $\gamma_6$ is the dashed blue short {\sf I}-cycle. When an edge of the 4-graph carries two cycles, as in the second and third weaves of Figure \ref{fig:LoopWeave1}, we choose one of the colors for the edge itself and then highlight the edge in the color corresponding to the additional homology cycle. The numerals correspond to Legendrian surface Reidemeister moves -- Legendrian isotopies that result in the combinatorial changes of fronts depicted in Figure \ref{fig:Moves}. 
	\begin{center}\begin{figure}[h!]{ \includegraphics[width=\textwidth]{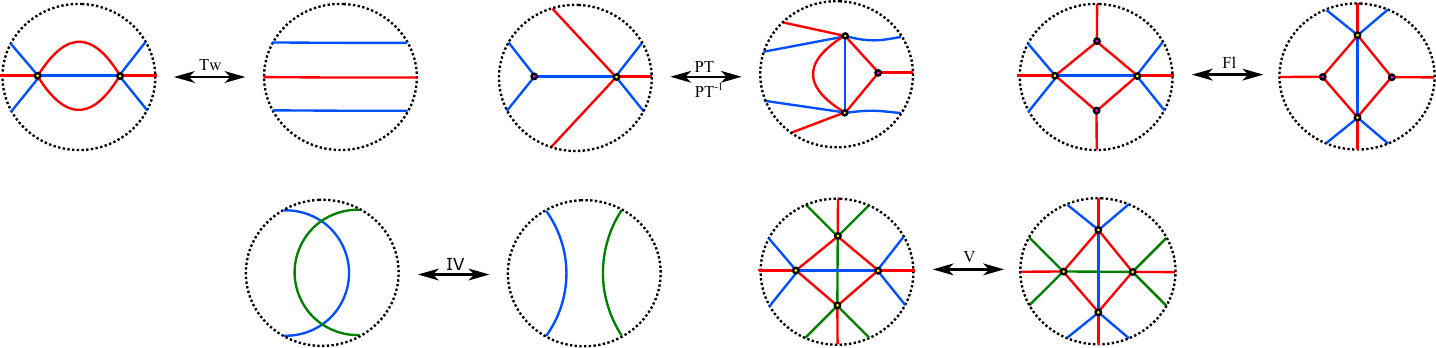}}\caption{Legendrian Surface Reidemeister moves for $N$-graphs. Clockwise from top left, a candy twist, a push-through, a flop, and two additional moves, denoted by I, II, III, IV, and V respectively.} 
		\label{fig:Moves}\end{figure}
	\end{center}
We freely apply Move IV by passing edges labeled by $\sigma_3$ (colored dark green) and edges labeled by $\sigma_1$ (colored blue) over each other. See Figure 26 of \cite{ABL22} for a list describing the combinatorial behavior of homology cycles under these Legendrian surface Reidemeister moves. In cases where multiple push-through moves are used, we omit some intermediate steps when the computation is otherwise straightforward. 

Note that the choice of braid word for $\Delta^2$ appearing in the weave $\w(\mathbb{G}(\tilde{D}_n))$ differs from the choice of braid word for $\Delta^2$ appearing in the Legendrian front pictured in Figure \ref{fig:LoopEx}. The front for $\partial\w(\mathbb{G}(\tilde{D}_n))$ differs from the front appearing in Figure \ref{fig:LoopEx} by a sequence of Reidemeister III moves collecting a pair of $\sigma_1$ and $\sigma_3$ crossings at the left of the braid word for $\Delta^2$. The sheaf moduli obtained from the two fronts are isomorphic.

\begin{figure}
    \centering
    \includegraphics[width=\textwidth]{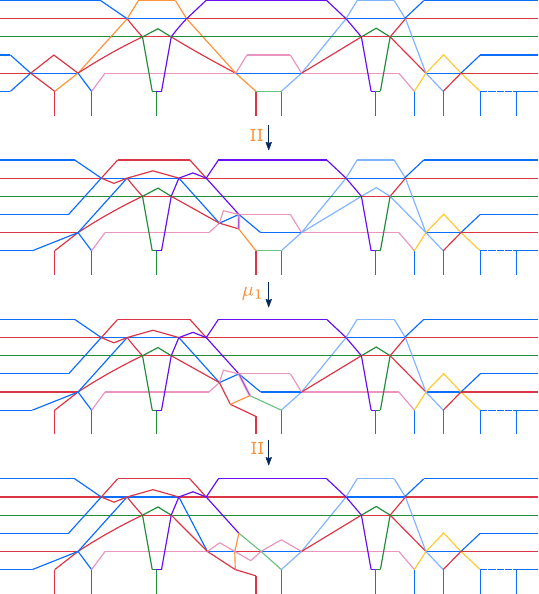}
    \caption{From top to bottom: a sequence of three push-throughs to isolate $\gamma_1$ as a short {\sf I}-cycle; mutation at $\gamma_1$; a push through to remove the edge carrying cycles $\gamma_1$ and $\gamma_4$.}
    \label{fig:LoopWeave1}
\end{figure}

\begin{figure}
    \centering
    \includegraphics[width=\textwidth]{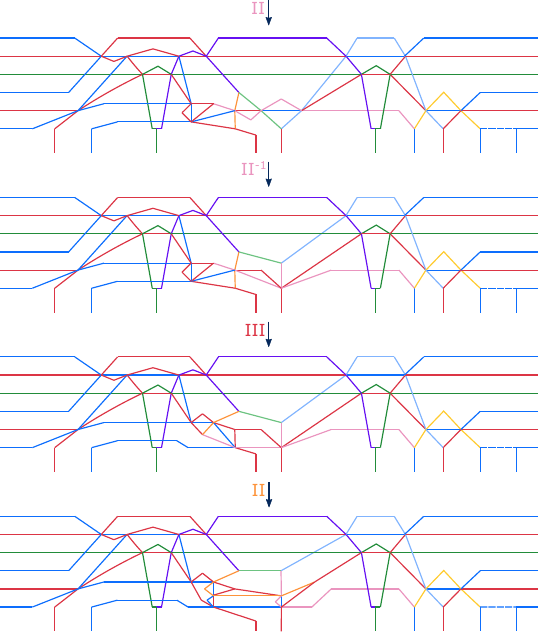}
    \caption{A series of weave equivalence moves designed to produce $\gamma_4$ as a long {\sf I}-cycle.}
    \label{fig:LoopWeave2}
\end{figure}
\begin{figure}
    \centering
    \includegraphics[width=\textwidth]{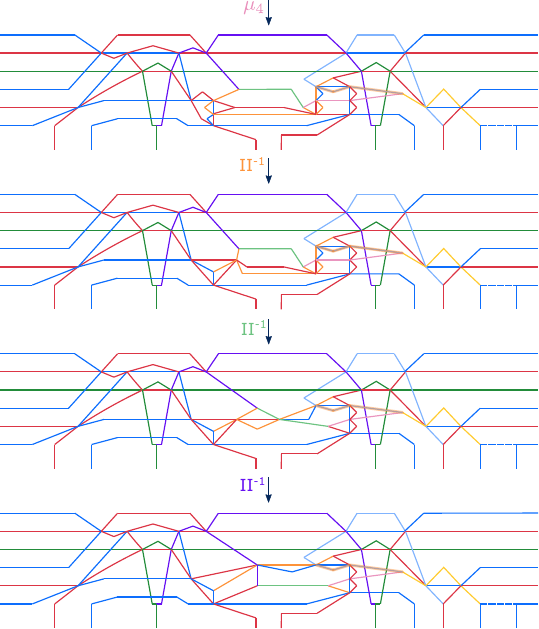}
    \caption{Mutation at $\gamma_4$ and a sequence of push-throughs simplifying $\gamma_1$.}
    \label{fig:LoopWeave3}
\end{figure}

\begin{figure}
    \centering
    \includegraphics[width=\textwidth]{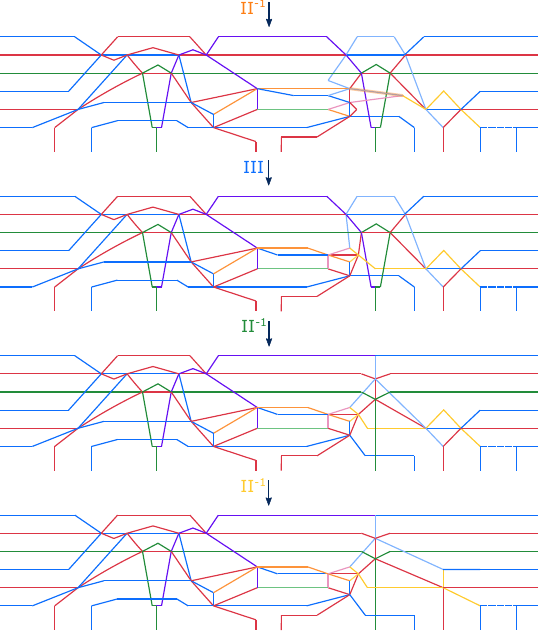}
    \caption{A sequence of weave equivalence moves removing the geometric intersections between $\gamma_1$ and $\gamma_0$.
    }
    \label{fig:LoopWeave4}
\end{figure}

\begin{figure}
    \centering
    \includegraphics[width=\textwidth]{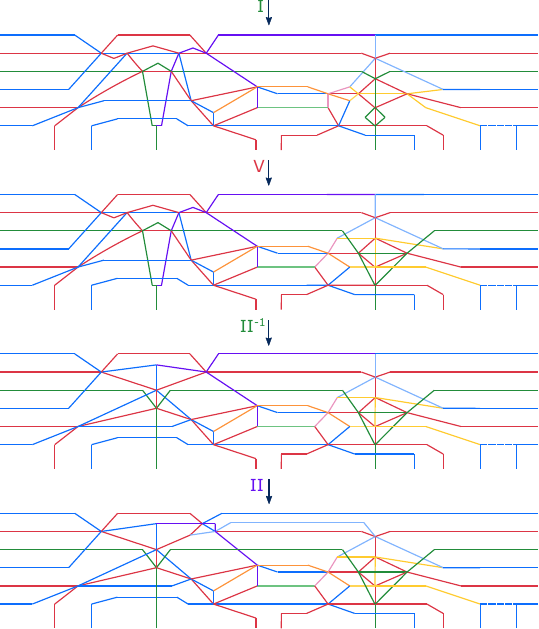}
    \caption{From top to bottom: a candy twist to introduce the necessary $\sigma_2-\sigma_3$ hexavalent vertices appearing in $\vartheta_2$; Move V; simplification of the left half of the weave starting with a pair of push-throughs involving the leftmost cycle in the weave; a push-through involving $\gamma_0$ and $\gamma_3$.  }
    \label{fig:LoopWeave5}
\end{figure}

\begin{figure}
    \centering
    \includegraphics[width=\textwidth]{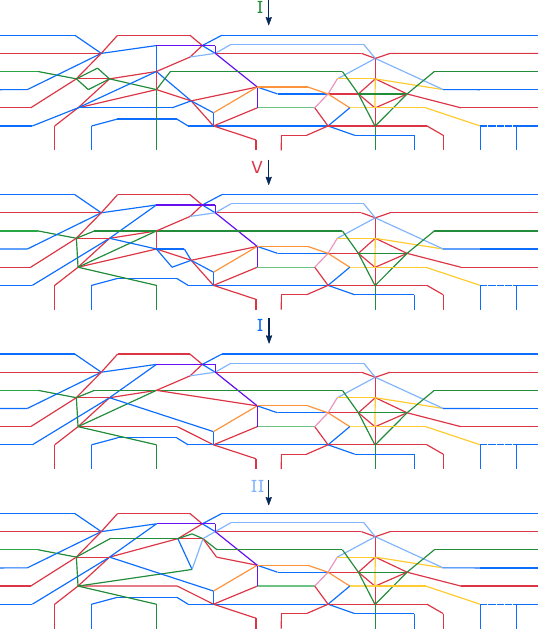}
    \caption{A series of weave equivalence moves continuing to simplify the left half of the weave}
    \label{fig:LoopWeave6}
\end{figure}

\begin{figure}
    \centering
    \includegraphics[width=\textwidth]{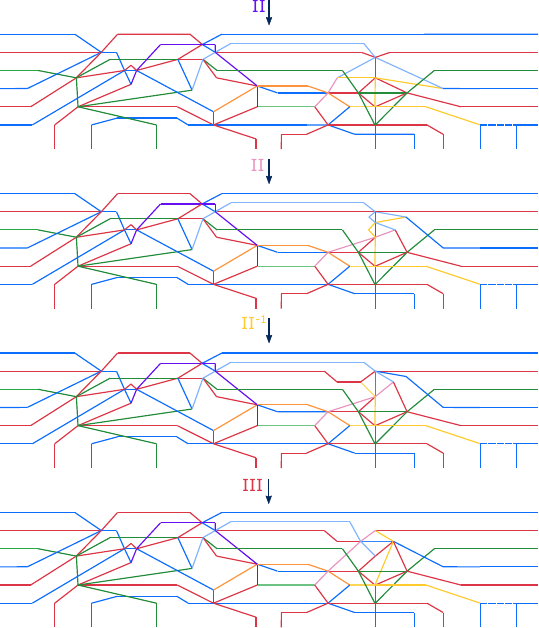}
    \caption{A series of weave equivalence moves continuing to simplify the right half of the weave}
    \label{fig:LoopWeave7}
\end{figure}

\begin{figure}
    \centering
    \includegraphics[width=\textwidth]{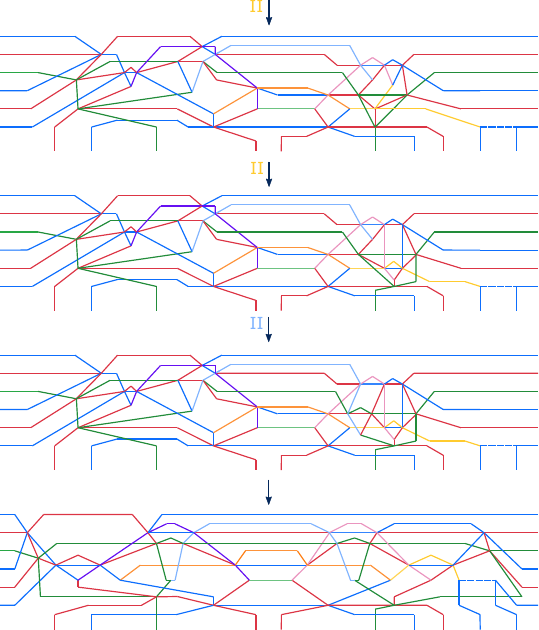}
    \caption{The final weave appearing here differs from the one above it by planar isotopy and freely applying Move IV. It is readily identified with the concatenation of the trace of $\vartheta_2$ to the initial weave appearing in Figure \ref{fig:LoopWeave1} up to relabeling the cycles by the permutation (1\,4\,0\,3).}
    \label{fig:LoopWeave8}
\end{figure}

\end{document}